\documentclass[11pt]{amsart}
\usepackage{amssymb,amsmath,amsthm,amscd,mathrsfs,graphicx,color,multicol,booktabs}
\usepackage[cmtip,all,matrix,arrow,tips,curve]{xy}

\numberwithin{equation}{section}
\theoremstyle{plain}
\newtheorem{theorem}{Theorem}[section]

\newtheorem{lemma}[theorem]{Lemma}
\newtheorem{claim}[theorem]{Claim}

\newtheorem{proposition}[theorem]{Proposition}
\newtheorem{corollary}[theorem]{Corollary}

\theoremstyle{remark}
\newtheorem{remark}[theorem]{Remark}

\theoremstyle{definition}
\newtheorem{definition}[theorem]{Definition}

\newtheorem{notationconv}[theorem]{Notation/Convention}

\newtheorem{example}[theorem]{Example}

\def\Pic{\operatorname{Pic}}
\def\Spec{\operatorname{Spec}}

\def\Hom{\operatorname{Hom}}

\def\Sym{\operatorname{Sym}}

\newcommand{\bP}{\mathbb{P}}

\newcommand{\bQ}{\mathbb{Q}}
\newcommand{\bZ}{\mathbb{Z}}

\newcommand{\cF}{\mathscr{F}}

\newcommand{\bC}{\mathbb{C}}

\newcommand{\gp}{\mathfrak{p}}
\newcommand{\sH}{\mathscr{H}}

\newcommand{\calO}{\mathcal{O}}

\newcommand{\calM}{\mathcal{M}}

\newcommand{\Aut}{\mathrm{Aut}}
\newcommand{\Stab}{\mathrm{Stab}}

\newcommand{\PGL}{\mathrm{PGL}}

\newcommand{\SL}{\mathrm{SL}}

\newcommand{\sX}{\mathscr{X}}
\newcommand{\sD}{\mathscr{D}}

\newcommand{\Proj}{\mathrm{Proj}}
\newcommand{\Sing}{\mathrm{Sing}}
\newcommand{\im}{\operatorname{Im}}

\newcommand{\Gr}{\mathrm{Gr}}

\newcommand{\SO}{\mathrm{SO}}

\newcommand{\gquot}{/\!\!/}

\newcommand{\II}{\textrm{II}}

\DeclareMathOperator{\Bl}{Bl}

\DeclareMathOperator{\Ort}{O}
\DeclareMathOperator{\PO}{PO}
\newcommand\aff{{\mathbb A}}

\newcommand{\CC}{{\mathbb C}}
\newcommand{\cA}{{\mathscr A}}

\newcommand{\cD}{{\mathscr D}}

\newcommand{\cH}{{\mathscr H}}
\newcommand{\cI}{{\mathscr I}}
\newcommand{\cJ}{{\mathscr J}}
\newcommand{\cK}{{\mathscr K}}
\newcommand{\cL}{{\mathscr L}}

\newcommand{\cO}{{\mathscr O}}

\newcommand{\cU}{{\mathscr U}}

\newcommand{\cX}{{\mathscr X}} 
\newcommand{\cY}{{\mathscr Y}}

\newcommand{\dra}{\dashrightarrow}

\newcommand{\es}{\varnothing}

\newcommand{\gM}{\mathfrak{M}}

\newcommand{\hra}{\hookrightarrow}
\newcommand{\la}{\langle}

\newcommand{\lra}{\longrightarrow}

\newcommand{\NN}{{\mathbb N}}
\newcommand{\ov}{\overline}
\newcommand{\PP}{{\mathbb P}}
\newcommand{\QQ}{{\mathbb Q}}
\newcommand{\ra}{\rangle}
\newcommand{\RR}{{\mathbb R}}

\newcommand{\wh}{\widehat}
\newcommand{\wt}{\widetilde}
\newcommand{\ZZ}{{\mathbb Z}}

\DeclareMathOperator{\Cl}{Cl}

\DeclareMathOperator{\diag}{diag}
\DeclareMathOperator{\divisore}{div}

\DeclareMathOperator{\Ind}{Ind}

\DeclareMathOperator{\mult}{mult}

\DeclareMathOperator{\PSL}{PSL}
\DeclareMathOperator{\sing}{sing}
\DeclareMathOperator{\supp}{supp}
\DeclareMathOperator{\Symm}{S}

\usepackage{ifthen}
\usepackage{rotating}
\usepackage{supertabular}
\newcommand{\cit}[1]{{\rm \textbf{#1}}}
\newcommand{\Ref}[2]{\cit{%
\ifthenelse{\equal{#1}{thm}}{Theorem}{}%
\ifthenelse{\equal{#1}{ass}}{Assumption}{}%
\ifthenelse{\equal{#1}{chp}}{Chapter}{}%
\ifthenelse{\equal{#1}{prp}}{Proposition}{}%
\ifthenelse{\equal{#1}{lmm}}{Lemma}{}%
\ifthenelse{\equal{#1}{cnj}}{Conjecture}{}%
\ifthenelse{\equal{#1}{crl}}{Corollary}{}%
\ifthenelse{\equal{#1}{dfn}}{Definition}{}%
\ifthenelse{\equal{#1}{expl}}{Example}{}%
\ifthenelse{\equal{#1}{hyp}}{Hypothesis}{}%
\ifthenelse{\equal{#1}{rmk}}{Remark}{}%
\ifthenelse{\equal{#1}{clm}}{Claim}{}%
\ifthenelse{\equal{#1}{notconv}}{Notation/Convention}{}%
\ifthenelse{\equal{#1}{pred}}{Prediction}{}%
\ifthenelse{\equal{#1}{prb}}{Problem}{}%
\ifthenelse{\equal{#1}{table}}{Table}{}%
\ifthenelse{\equal{#1}{exe}}{Exercise}{}%
\ifthenelse{\equal{#1}{qst}}{Question}{}%
\ifthenelse{\equal{#1}{rcp}}{Recipe}{}%
\ifthenelse{\equal{#1}{sec}}{Section}{}%
\ifthenelse{\equal{#1}{subsec}}{Subsection}{}%
\ifthenelse{\equal{#1}{subsubsec}}{Subsubsection}{}%
\ifthenelse{\equal{#1}{univ}}{Universal Property}{}%
\ifthenelse{\equal{#1}{trm}}{Terminology}{}%
\ifthenelse{\equal{#1}{tbl}}{Table}{}%
\  \ref{#1:#2}%
}}

\usepackage{fullpage}
\begin{document}
 \title[Moduli of quartic $K3$s (2)]{GIT versus Baily-Borel compactification for quartic $K3$ surfaces}
 
 \author[R. Laza]{Radu Laza}
\address{Stony Brook University,  Stony Brook, NY 11794, USA}
\email{radu.laza@stonybrook.edu}

 \author[K. O'Grady]{Kieran O'Grady}
\address{``Sapienza'' Universit\'a di Roma, Rome, Italy}
\email{ogrady@mat.uniroma1.it}
\date{\today}

\begin{abstract}
Looijenga has  introduced new compactifications of locally symmetric varieties that give a complete understanding  
of the period map from the GIT moduli space of plane sextics to the Baily-Borel compactification   of the moduli space  polarized $K3$'s of degree $2$, and also  of the period map of cubic fourfolds. On the other hand, the period map of the GIT moduli space of quartic surfaces  is significantly more subtle.  In our paper~\cite{log1} we introduced a Hassett-Keel--Looijenga program for certain locally symmetric varieties of Type IV.  As a consequence, we gave a complete conjectural decomposition  into a product of elementary birational modifications of  the period map  for  the GIT moduli spaces of quartic surfaces. 
 The purpose of this note is to provide compelling evidence in favor of our program. 
Specifically, we propose a matching between the arithmetic strata in the period space and suitable  strata of the GIT moduli spaces of quartic surfaces.   We then partially verify that the proposed matching actually holds.
\end{abstract}

\thanks{Research of the first author is supported in part by NSF grants DMS-125481 and DMS-1361143. Research of the second author is supported in part by PRIN 2013.}
  \maketitle
  
\bibliographystyle{amsalpha}

%%%%%%%%%%%%%%%%%%%%%%%%%%%%%%%%%%%%%%%
%\noindent{\color{blue} {\bf (Kieran  12/8/2016)}} 

%%%%%%%%%%%%%%%%%%%%%%%%%%%%%%%%%%%%%%%

\section{Introduction}
The general context of our paper is the search for a geometrically meaningful compactification of  moduli spaces of polarized $K3$ surfaces, and similar varieties (with Hodge structure of $K3$ type). While  there exist well-known geometrically meaningful compactifications of  moduli spaces of smooth curves and of (polarized) abelian varieties, the situation for $K3$'s is much murkier. The basic fact about the moduli space of degree-$d$ polarized $K3$ surfaces $\cK_d$ is that, as a consequence of Torelli and properness of the period map, it is isomorphic to a locally symmetric variety $\cF_d= \Gamma_d\backslash\cD$, where  $\cD$ is a $19$-dimensional Type IV Hermitian symmetric domain, and $\Gamma_d$ is an arithmetic group. As such, $\cF_d$ has many known compactifications (Baily-Borel, toroidal, etc.), but the  question is whether some of these are modular  (by way of comparison, we recall that  the second Voronoi toroidal compactification of  $\cA_g$ is modular, cf.~Alexeev \cite{alexeevag}). The most natural approach to this question is to compare  birational models of $\cK_d$ (e.g.~those given by GIT moduli spaces of plane sextic curves,  quartic surfaces, complete intersections of a quadric and a cubic in $\PP^4$) and the known compactifications of $\cF_d$      via the period map. 
The  most basic compactification  of $\cF_d$ is the one introduced by Baily-Borel; we   denote it by  $\cF_d^*$.   In ground-breaking work, Looijenga \cite{looijenga1,looijengacompact} gave a framework for the  comparison of GIT and Baily-Borel compactifications of moduli spaces of low degree $K3$ surfaces and  similar examples (e.g.~cubic fourfolds). Roughly speaking, Looijenga proved that, under suitable hypotheses, natural GIT birational models of a moduli space of polarized $K3$ surfaces can be obtained by arithmetic modifications from the Baily-Borel compactification. In particular, Looijenga and others have given a complete, and unexpectedly nice, picture of the period map  for the GIT moduli space of plane sextics (which is birational to the moduli space of polarized $K3$'s of degree $2$), 
 see \cite{shah,looijengavancouver,friedmanannals}, and for the GIT moduli space of cubic fourfolds (which is birational to the moduli space of polarized hyperk\"ahler varieties of Type $K3^{[2]}$ with a polarization of degree $6$ and divisibility $2$), see \cite{lcubic,gitcubic,cubic4fold}. By contrast,  at  first glance,  Looijenga's framework appears not to apply to the GIT moduli space of quartic surfaces (and their cousins, double EPW sextics): \cite{shah4} and~\cite{epw} showed that the GIT stratification of moduli spaces of quartic surfaces and EPW sextics, respectively,  is  much more complicated than the analogous stratification of the GIT moduli spaces of plane sextics or cubic fourfolds, and there is no  decomposition of the (birational) period map to the Baily-Borel compactification into a product of elementary modifications  as simple as that of the period map of degree $2$ $K3$'s or cubic fourfolds. In our 
paper~\cite{log1}, we refined Looijenga's work and we proved that, morally speaking, Looijenga's framework can be successfully applied to the period map of quartic surfaces and EPW sextics. In fact, we have noted that Looijenga's work should be viewed as an instance of the study of variation of (log canonical) models for moduli spaces (a concept that matured more recently, starting with the work of Thaddeus \cite{thaddeus}, and continued, for example,  with the so-called Hassett--Keel program). This led to the introduction, in~\cite{log1}, of a  program, which might be dubbed  \emph{Hassett--Keel--Looijenga program}, whose aim is to study the log-canonical  models of locally symmetric varieties of Type IV equipped with a collection of Heegner divisors (in that paper we concentrated on a specific series of locally symmetric varieties and Heegner divisors, but the program makes sense in complete generality). In particular,  in~\cite{log1} we made very specific predictions for the decomposition into  products of elementary birational modifications  of the period maps  for the GIT moduli spaces of quartic $K3$ surfaces. 

Our predictions  are in the spirit of Looijenga \cite{looijengacompact}, i.e.~the elementary birational modifications are dictated by arithmetic. There are two related issues arising here: First, the various strata in the period space  should correspond to geometric strata in the GIT compactification. Secondly, our work in \cite{log1} is only predictive, i.e.~there is no guarantee that the given list of birational modifications is complete, or even that all these modifications occur. The purpose of this note is to partially address these two issues. Namely, we give what we believe to be a complete matching between the geometric and arithmetic strata, thus  addressing the first issue. We view this result as strong evidence towards the completeness and accuracy of our predictions.  While our previous paper~\cite{log1} looks at the period map from the point of view of the target (the Baily-Borel compactification of the period space), the present paper's vantage point is that of the GIT moduli spaces of quartic surfaces: we get what appears to be a  snapshot of  the predicted decomposition of the period map into a product of simple birational modifications.

Let us discuss more concretely the content of this note, and its relationship to \cite{log1}. To start with, we recall that in~\cite{log1} we have introduced, for each $N\ge 3$, an $N$-dimensional locally symmetric variety $\cF(N)$ associated to the $D$ lattice  $U^2\oplus D_{N-2}$. The space $\cF(19)$ is the period space of degree-$4$ polarized $K3$ surfaces, and also $\cF(18)$, $\cF(20)$ are period spaces for natural polarized varieties (see~\Ref{subsec}{predictq} for details). The main goal of that paper is to predict the behavior of the schemes 
\begin{equation*}
 \cF(N,\beta)=\Proj R(\cF(N),\lambda(N)+\beta\Delta(N)),\qquad \beta\in[0,1]\cap\QQ,
\end{equation*}
where $\lambda(N)$ is the Hodge (automorphic) divisor class on $\cF(N)$,  $\Delta(N)$ is a \lq\lq boundary\rq\rq\ divisor, with  a clear geometric meaning for $N\in\{18,19,20\}$, and $R(\cF(N),\lambda(N)+\beta\Delta(N))$ is the graded ring associatde to the $\QQ$-Cartier divisor class $\lambda(N)+\beta\Delta(N)$. 
  
For all $N$, the scheme $\cF(N,0)$ is  the Baily-Borel compactification $\cF(N)^{*}$. 
 At the other extreme, for $N=19,18$, the scheme  $\cF(N,1)$ is isomorphic to a natural GIT moduli space $\gM(N)$ (and we are confident that the same remains true for $N=20$).  From now on, we will concentrate our attention on   $\cF:=\cF(19)$. The relevant GIT moduli space is that of quartic surfaces, i.e.
\begin{equation*}
\gM:=|\cO_{\PP^3}(4)|\gquot\PGL(4).
\end{equation*}
 The period map 
\begin{equation*}
\gp\colon\gM\dra \cF^{*}  
\end{equation*}
is birational by Global Torelli.  We expect (following Looijenga) that the inverse   $\gp^{-1}$  decomposes as the product of a $\QQ$-factorialization,  a series of flips, and, at the last step,  a divisorial contraction. 

In order to be more specific, we need to describe the boundary divisor  $\Delta$ for $\cF$. First, let $H_h,H_u\subset\cF$ be  the (prime) divisors parametrizing periods of hyperelliptic degree $4$ polarized $K3$'s,  and unigonal degree $4$ polarized $K3$'s respectively - they are both Heegner (i.e.~Noether-Lefschetz) divisors. 
The boundary divisor is given by 
\begin{equation*}
\Delta:=(H_h+H_u)/2.
\end{equation*}
The  birational transformations mentioned above are obtained by considering $\cF(\beta):=\cF(19,\beta)$ for $\beta\in[0,1]\cap\QQ$. 

The main result of our previous paper is the prediction of  the critical values of $\beta$ corresponding to the flips, together with the description of (the candidates for) the centers of the flips on the $\cF$  side.  
In fact in~\cite{log1} we have defined  towers of closed subsets (see~\eqref{zetaquart})
\begin{equation}\label{zedtower}
Z^{9}\subset Z^{8} \subset Z^{7}\subset Z^{5}\subset Z^{4}\subset Z^{3}\subset Z^{2}\subset Z^{1}=\supp\Delta\subset \cF,
\end{equation}
where  $k$ denotes the codimension ($Z^6$ is  missing, \emph{no typo}). Our prediction is that the critical values of $\beta$ are
\begin{equation}\label{betacrit}
0,\frac{1}{9},\frac{1}{7},\frac{1}{6},\frac{1}{5},\frac{1}{4},\frac{1}{3},\frac{1}{2},1,
\end{equation}
and that the center of the $n$-th flip (corresponding to the $i$-th critical value) is the closure of the strict transform of the $n$-th term in the relevant tower (the $\QQ$-factorialization corresponds to small $0<\beta$, hence the corresponding $0$-th critical $\beta$ is $0$). The last  critical value of $\beta$, i.e.~$\beta=1$ corresponds to the contraction of the strict transform of the boundary divisor.

On the GIT moduli space side,   Shah~\cite{shah4} has defined a closed locus  $\gM^{IV}\subset\gM$  containing the indeterminacy locus of the period maps 
(we predict that it  coincides with the indeterminacy locus), which has a natural stratification (see~\Ref{dfn}{doubleuq})
\begin{equation*}
\scriptstyle
\gM^{IV}= (W_8\sqcup\{\upsilon\})\supset (W_7\sqcup\{\upsilon\})\supset (W_6\sqcup\{\upsilon\})\supset (W_4\sqcup\{\upsilon\})\supset 
(W_3\sqcup\{\upsilon\})\supset (W_2\sqcup\{\upsilon\})\supset (W_1\sqcup\{\upsilon\})\supset (W_0\sqcup\{\upsilon\}),
\end{equation*}
where $\upsilon$ is the point corresponding to the tangent developable of a twisted cubic curve, and the index denote dimension. As predicted by Looijenga, and refined by us, we expect that the center in $\gM$ corresponding to the center $Z^k$ (respectively $Z^k_h$) is $W_{k-1}\sqcup\{\upsilon\}$. The purpose of this note is to give evidence in favor of the above matching. We prove that the described matching holds for $Z^1$ and $Z^2$ (equivalently, for $(W_1\sqcup\{\upsilon\})$ and $(W_0\sqcup\{\upsilon\})$), and we provide evidence for the matching between $Z^9,Z^8,Z^7$ and $(W_8\sqcup\{\upsilon\})$, 
$(W_7\sqcup\{\upsilon\})$, $(W_6\sqcup\{\upsilon\})$ respectively.

In \Ref{sec}{overview} we give a very brief overview of the  framework developed by Looijenga in order to compare the GIT and Baily-Borel compactifications of moduli spaces of polarized $K3$ surfaces, or similar varieties, and we will illustrate it by giving a bird's-eye-view of the period map for degree-$2$ $K3$'s and cubic fourfolds. We then introduce the point of view developed in~\cite{log1}, and we describe in detail the predicted decomposition of the inverse of the period map for quartic surfaces as  product of elementary  birational maps (i.e.~flips or contractions), see~\eqref{diagramflip}. 

We continue in \Ref{sec}{shahreview}, by revisiting the work of Shah \cite{shah4} on the GIT for quartic surfaces.  Usually, in a GIT analysis, by boundary one understands the locus (in the GIT quotient) parameterizing strictly semistable objects, which then can be stratified in terms of stabilizers of the polystable points (see Kirwan \cite{kirwan}). In his works on periods of quartic surfaces, Shah (see~\cite{shahinsignificant,shah}) noted that a more refined stratification emerges when studying the period map, resulting into four Types of quartic surfaces, labeled I--IV, with corresponding locally closed subsets of $\gM$ denoted $\gM^{I},\ldots,\gM^{IV}$. A quartic is of Type I--III if it is cohomologically insignificant (or from a more modern point of view, it is  semi-log-canonical), and thus the period map extends  over the open subset of the moduli space parametrizing such surfaces; moreover the Type determines whether the period point belongs to the period space (Type I), or it belongs to one of the Type II or Type III boundary components of the Baily-Borel compactification. The remaining surfaces are of Type IV, in particular the  indeterminacy locus of $\gp\colon\gM\dra\cF(19)^{*}$ is contained in $\gM^{IV}$ (we predict that it coincides with $\gM^{IV}$). In the analogous case of the period map from the GIT moduli space of plane sextics to the period space for polarized $K3$'s of degree $2$, the Type IV locus consists of a single point (corresponding to the triple conic). On the other hand, for quartic surfaces the Type IV locus is of big dimension and it has a complicated structure. 
In our revision of Shah's work, we shed some light on the structure of Type IV (and Type II and III) loci. While arguably everything that we do here is contained in Shah, we believe that the structure becomes transparent only after one knows the predicted arithmetic behavior. In some sense, the main point of Looijenga is to bring order to the world of GIT quotients of varieties of $K3$ type, by relating it to the orderly world of  hyperplane arrangements.

In~\Ref{sec}{refshah}, we define partitions of $\gM^{II}$ and $\gM^{III}$ into locally closed subsets (our partitions are slightly finer than partitions which have already been defined by Shah in~\cite{shah4}), and we define the stratification of $\gM^{IV}$  discussed above.

In \Ref{sec}{hklquartics} and~\Ref{sec}{dolgachev} we provide evidence in favor of the predictions of~\cite{log1} for $\gp\colon\gM\dra\cF^{*}$. We start (\Ref{sec}{hklquartics}) by showing that the period map behaves as predicted in neighborhoods of the points $\upsilon,\omega\in\gM$ corresponding to the   tangent developable of a twisted cubic curve and a double (smooth) quadric respectively. 
By blowing up those points one \lq\lq improves\rq\rq\ the behavior of the period map; the exceptional divisor over $\upsilon$ maps regularly to the (closure) of the unigonal divisor in $\cF^{*}$,  the exceptional divisor over $\omega$ maps  to the (closure) of the hyperelliptic divisor $H_h$  in $\cF^{*}$, and the image of the set of regular points for the map in $H_h$  is precisely the complement of $Z^{2}$. 
This result is essentially present in~\cite{shah4} (and belongs to \lq\lq folk\rq\rq\ tradition); we take care in specifying the weighted  blow up that one needs to perform around 
$\upsilon$ in order to make the map regular above $\upsilon$. In the language that we introduced previously, the above results match $Z^1$ with $W_0\sqcup\{\upsilon\}$. Next, we match $Z^2$ and $W_1\sqcup\{\upsilon\}$. This is the first flip in the chain of birational modifications transforming the GIT into the Baily-Borel compactification, and it is more involved than the blow-ups of $\upsilon$ and $\omega$. It suffices here   to mention that $W_1$ parametrizes quartics $Q_1+Q_2$, where $Q_1,Q_2$ are quadrics tangent along a smooth conic. (Warning: we do not provide full details of some of the proofs.) We note that while some similar arguments and computations occur previously in the literature (esp.~in work of Shah \cite{shah,shah4}), to our knowledge, the discussion here is the most complete and detailed analysis of an explicit (partial) resolution of a period map for $K3$ surfaces (esp.~the discussion of the flip is mostly new).

In~\Ref{sec}{dolgachev}, we provide evidence in favor of the matching of $Z^9,Z^8,Z^7$ and
 $W_8\sqcup\{\upsilon\},W_7\sqcup\{\upsilon\},W_6\sqcup\{\upsilon\}$. It is interesting to note that the flips of $Z^9,Z^8,Z^7$  
 are associated to the so-called Dolgachev singularities (aka triangle singularities or exceptional unimodular singularties) $E_{12}$, $E_{13}$, and $E_{14}$ respectively. These are the simplest non-log canonical singularities, essentially analogous to  cusp for curves. The geometric behavior of variation of $\cF(\beta)$ at the corresponding critical values is analogous to the behavior of the Hassett-Keel space $\gM_g(\alpha)$ around $\alpha=\frac{9}{11}$ (when stable curves with an elliptic tail are replaced by curves with cusps, see \cite{hh1}). While hints of this behavior exist in the literature (see Hassett \cite{hassett}, Looijenga \cite{ltriangle1}, Shepherd-Barron, and Gallardo \cite{gallardo}), our $\cF(\beta)$ example is the first genuine analogue of a Hassett-Keel behavior for surfaces (the existence of this is well-known speculation among experts in the field).

In the final section (\Ref{sec}{LooijengaQ}), we  discuss Looijenga's $\QQ$-factorialization of $\cF^{*}$, that we denote $\wh{\cF}$, and the matching between the irreducible components of $\gM^{II}$ (i.e.~the elements of the partition of $\gM^{II}$ defined in~\Ref{sec}{refshah}) and the irreducible components of $\cF^{II}$ (i.e.~the Type II boundary components of $\cF^{*}$).  From our point of view, Looijenga's $\QQ$-factorialization of $\cF^{*}$ is nothing else but $\cF(\epsilon)$ for $\epsilon>0$ small  (the prediction of~\cite{log1} is that 
$0<\epsilon<1/9$ will do). We compute the dimensions of the inverse images in $\wh{\cF}$ of the Type II boundary components of $\cF^{*}$. Lastly, we match the irreducible components of $\gM^{II}$ and the Type II boundary components of $\cF^{*}$. This matching deserves a more detailed discussion elsewhere. On the GIT side, $\gM^{II}$ has  $8$  components  (of varying dimension), while $\cF^{*}$   has  $9$  Type II boundary components  (as computed in~\cite{scattone}), each of them is a modular curve.  By adapting arguments  of Friedman in~\cite{friedmanannals}, we can match each of the $8$  components of $\gM^{II}$ to one  of the $9$ Type II boundary  components of $\cF^{*}$, and hence exactly one Type II boundary component is left out.   The discrepancy of dimensions between GIT and Baily-Borel strata (for the $8$ matching strata)  is  explained by  Looijenga's $\bQ$-factorialization of the Baily-Borel compactification (one of the main results of \cite{looijengacompact}). A mystery, at least for us, was the presence of a \lq\lq missing\rq\rq\ Type II  boundary of $\cF^{*}$. 
This has to do with what we call the second order corrections to Looijenga's predictions (one of the main  discoveries of~\cite{log1}). 

\smallskip

To conclude, we believe that while further work is needed (and small adjustments might occur), there is very strong evidence that our predictions from \cite{log1} are 
 accurate. In any case, Looijenga's visionary idea that the natural (or ``tautological'') birational models (such as GIT) of the moduli space of polarized $K3$s are controlled by the arithmetic of the period space is validated in the highly non-trivial case of quartic surfaces (by contrast, in the previous known examples \cite{shah,looijengavancouver,lcubic,cubic4fold,ls,act} only first order phenomena were visible, and thus a bit misleading). As possible applications of our program,  starting from the period domain side, one can bring structure and order to the (a priori) wild side of GIT. Conversely, starting from GIT and the work of Kirwan \cite{kirwan,kirwanhyp}, one can follow our factorization of the period map (and do \lq\lq wall crossing\rq\rq\ computations) and compute, say, the Betti numbers of $\cF$.

\begin{notationconv}
A \emph{$K3$ surface} is a complex projective surface $X$ with DuVal singularities,  trivial dualizing sheaf  $\omega_X$, and $H^1(\cO_X)=0$. 

We let $U$ be the hyperbolic plane, and root lattices are always negative definite. Let $\Lambda$ be a lattice, and $v,w\in\Lambda$. We let $(v,w)$ be the value of the bilinear symmetric form on the couple $v,w$, and we let $q(v):=(v,v)$.
The \emph{divisibility} of $v$ is the positive integer $\divisore(v)$ such that $(v,\Lambda)=\divisore(v)\ZZ$. Let $v$ be \emph{primitive} (i.e.~$v=mw$ implies that $m=\pm 1$); if $\Lambda$ is unimodular, then 
$\divisore(v)=1$, in general it might be greater than $1$.
\end{notationconv}

\subsection*{Acknowledgements} 

While the same acknowledgements as in \cite{log1} apply here, we repeat our thanks to E. Looijenga for his inspiring work and some helpful feedback, and to IAS for hosting us during the Fall of 2014, which represents the start of our investigations. Finally, much of this paper was written while the first author was on sabbatical and visiting \'Ecole Normale Sup\'erieur. He thanks ENS for hospitality and  {\it Simons Foundation} and {\it Fondation Sciences Math\'ematiques de Paris} for partially supporting this stay.

%%%%%%%%%%%%%%%%%%%%%%%%%%%%%%%%%%%%%%%

\section{GIT vs.~Baily-Borel  for locally symmetric varieties of Type IV}\label{sec:overview}
\setcounter{equation}{0}
The purpose of this section is to give a very brief account of Looijenga's framework and our enhancement from \cite{log1} (with a focus on quartic surfaces). We start with the simplest non-trivial example that fits into Looijenga's framework -  degree-$2$ $K3$ surfaces (see \cite{shah}, \cite{looijengavancouver}, \cite[\S5]{friedmanannals}, and \cite[\S1]{k3pairs} for a concise account). We then briefly touch on the general case,  and we recall how it applies   to the moduli space of cubic fourfolds. Lastly, we describe in detail our (conjectural) decomposition of the period map  for  quartic surfaces into a product of elementary birational modifications, see~\cite{log1}. 
\begin{remark}
To the best of our knowledge, the first instance of Looijenga's framework is in Igusa's celebrated paper~\cite{igusag2} on modular forms of genus $2$. The paper by Igusa analyzes the (birational) period map between the  compactification of the moduli space of (smooth) genus $2$ curves provided by the GIT quotient of binary sextics and the Satake compactification of $\cA_2$ (notice that $\cA_2$ is a locally symmetric variety of Type IV). Igusa describes explicitly the blow-up of a non-reduced point in the GIT moduli space needed to resolve the period map. See~\cite{hassg2} for a more recent version of this story.
\end{remark}
\subsection{Degree-$2$ $K3$ surfaces}\label{subsec:degtwo}
Let $\cF_2$ be the period space of  degree-$2$ polarized $K3$ surfaces, i.e.~$\cF_2=\Gamma_2\backslash\cD$, where   $\Gamma_2$ and $\cD$ are defined as follows. Let $\Lambda:=U^2\oplus E_8^2\oplus A_1$. Thus  $\Lambda$ is isomorphic to the primitive integral cohomology of a polarized $K3$ of degree $2$. Then
\begin{equation}\label{howmanytimes}
\cD:=\{[\sigma]\in\PP(\Lambda\otimes\CC) \mid q(\sigma)=0,\quad q(\sigma+\ov{\sigma})>0\}^{+},\qquad \Gamma_2:= O^{+}(\Lambda).
\end{equation}
Here the first superscript $+$ means that we choose one connected component (there are two, interchanged by complex conjugation), the second one means that $\Gamma_2$ is the index-$2$ subgroup of $O(\Lambda)$ which maps $\cD$ to itself.
Let $\cF_2\subset\cF_2^{*}$ be the Baily-Borel compactification. Let $\gM_2:=|\cO_{\PP^2}(6)|\gquot\PGL(3)$ be the GIT moduli space of plane sextics. 
We let
\begin{equation*}
\gp\colon\gM_2\dra\cF_2^{*},\qquad \gp^{-1}\colon\cF_2^{*}\dra\gM_2
\end{equation*}
be the (birational) period map and its inverse, respectively. By Shah~\cite{shah} the period map $\gp$ is regular away from the point $q\in\gM_2$  parametrizing the  $\PGL(3)$-orbit of 
 $3C$, where $C\subset\PP^2$ is a smooth conic (a closed orbit in $|\cO_{\PP^2}(6)|^{ss}$). Let $\gM_2^{I}\subset\gM_2$ be the open dense subset of orbits of curves with simple singularities, and let $H_u\subset\cF_2$ be the unigonal divisor, i.e.~the divisor parametrizing periods of unigonal degree-$2$ $K3$'s. Thus $H_u$ is a Heegner divisor; it is the image in $\cF_2$ of a hyperplane $v^{\bot}\cap\cD$, where $v\in\Lambda$ is such that $q(v)=-2$ and $\divisore(v)=2$  (any two such elements of $\Lambda$ are $\Gamma_2$-equivalent). Then 
  the period map defines an isomorphism $\gM_2^{I}\overset{\sim}{\lra}(\cF_2\setminus H_u)$. Let $L\in\Pic(\gM_2)_{\QQ}$ be the class induced by the hyperplane class on 
 $|\cO_{\PP^2}(6)|$, let
  $\lambda$ be the Hodge divisor class on $\cF_2$, and $\Delta:=H_u/2$;  a computation similar (but simpler) to those carried out in Sect.~4 of~\cite{log1} gives that 
 \begin{equation}\label{keypull}
\gp^{-1}L|_{\cF_2}=\lambda+\frac{1}{2}H_u=\lambda+\Delta
\end{equation}
(the $\frac{1}{2}$ factor indicates that $H_u$ is a ramification divisor of the quotient map $\cD\to\cF_2$). 
 Arguing as in Sect.~4.2 of~\cite{log1}, one shows that $\gp^{-1}$ is regular on all of $\cF_2$ (one key point is that $\cF_2$ is $\QQ$-factorial). On the other hand $\gp^{-1}$ is \emph{not} regular on all of $\cF_2^{*}$. In order to describe $\gp^{-1}$ on the boundary of $\cF_2$, let 
  $\widehat \cF_2\subset\cF_2^{*}\times\gM_2$ be the graph of  $\gp^{-1}$, and 
let $\Pi\colon  \widehat \cF_2\to\cF_2^{*}$, $\Phi\colon  \widehat \cF_2\to\gM_2$  be the projections:
\begin{equation}\label{triangolo}
\xymatrix{
& \widehat \cF_2\ar@{->}[ld]_{\Pi}\ar@{->}[rd]^{\Phi}&\\ 
\cF_2^*\ar@{-->}^{\gp^{-1}}[rr]&&\gM_2&
}
\end{equation}
 Thus $\Pi$ is an isomorphism over $\cF_2$ (because $\gp^{-1}$ is regular on $\cF_2$). On the other hand, it follows from Shah's description of semistable orbits in  $|\cO_{\PP^2}(6)|$, that the fibers of $\Pi$ over two of the four $1$-dimensional boundary components of $\cF_2^{*}$ are $1$-dimensional (namely those labeled by $E_8^2\oplus A_1$ and $D_{16}\oplus A_1$; see Remark 5.6 of~\cite{friedmanannals} for the notation), and they are $0$-dimensional over the remaining two boundary components. From this it follows that $\cF_2^{*}$ is not $\QQ$-factorial, because if it were $\QQ$-factorial, the exceptional set of $\Pi$ would have pure codimension $1$. Moreover, it follows that $\Pi$ is a $\QQ$-factorialization of $\cF_2^{*}$. In fact, since $\cF_2$ is $\QQ$-factorial, with rational Picard group freely generated by $\lambda$ and $H_u$,  the rational class group $\Cl(\cF_2^{*})_{\QQ}$ is   freely generated by $\lambda^{*}$ and $H^{*}_u$ (obvious notation). Since $\lambda^{*}$ is the class of a $\QQ$-Cartier divisor, it follows that $H^{*}_u$ is not $\QQ$-Cartier. Let $\widehat{H}_u\subset  \widehat \cF_2$ be the strict transform of $H_u$. Then $\widehat{H}_u$ is $\QQ$-Cartier, because by~\eqref{keypull},  there exists $m\gg 0$ such that $m\widehat{H}_u$  is the divisor of a section of the  line-bundle $\Phi^{*}L^{2m}\otimes\Pi^{*}(\lambda^{*})^{-2m}$. Moreover we can identify  $\widehat \cF_2$ with $\Proj(\bigoplus_{n\ge 0}\cO_{\cF_2^{*}}(nH^{*}_u))$, because $\widehat{H}_u$ is $\Pi$-ample (clearly $a\pi^{*}(\lambda^{*})+b\Phi^{*}L$ is ample for any $a,b\in\QQ_{+}$, using~\eqref{keypull} and the triviality of $\pi^{*}(\lambda^{*})$ on fibers of $\Pi$, it follows that  $\widehat{H}_u$ is $\Pi$-ample). Thus (as 
 in~\cite{looijengavancouver})  we have decomposed $\gp^{-1}$ as follows: first we construct  the $\QQ$-factorialization of $\cF_2^{*}$ given by  $\Proj(\bigoplus_{n\ge 0}\cO_{\cF_2^{*}}(nH^{*}_u))$, then we blow down the strict transform of $H_u^{*}$, i.e.~$\widehat{H}_u$. In this case the Mori chamber decomposition of the quadrant  $\{\lambda+\beta \Delta \mid \beta\in[0,1]\cap\QQ\}$ is very simple; there are exactly two  walls, corresponding to $\beta=0$ and $\beta=1$.   
\subsection{A quick overview of Looijenga's framework}\label{subsec:looijframe}
Let  $\gM^0$ be  a moduli space of  (polarized) varieties which are smooth or \lq\lq almost\rq\rq\ smooth (e.g.~surfaces with ADE singularities), with Hodge structure of $K3$ type. In particular the corresponding period space is $\cF=\Gamma\backslash \cD$, where $\cD$ is a Type IV domain or a complex ball, and $\Gamma$ is an arithmetic group. An example of $\gM^0$ is provided by the moduli space of degree-$d$ polarized $K3$ surfaces, embedded by a suitable multiple of the polarization (one also has to specify the linearized ample line-bundle on the relevant Hilbert scheme), and $\cF=\cF_d$ - in particular the example
 discussed in~\Ref{subsec}{degtwo}. 
We let  $\gM^0\subset \gM$ be a GIT compactification, and we let $\cF\subset\cF^{*}$ be the Baily-Borel compactification. Let
$$\gp\colon \gM\dra \cF^{*}$$
be  the period map, and assume that it  is \emph{birational}. Looijenga~\cite{looijenga1,looijengacompact} tackled the problem of resolving   $\gp$. 
 First, he  observed that in many instances   $\gp(\gM^0)=\cF\setminus \supp\Delta$, where $\Delta$ is an effective linear combination of Heegner divisors - in the example of~\Ref{subsec}{degtwo}, one chooses $\Delta=H_u/2$. It is reasonable to expect that 
\begin{equation}\label{gitdel}
\gM\cong \Proj R(\cF,\lambda+\Delta),
\end{equation}
where $\lambda$ is the Hodge (automorphic) $\QQ$-line bundle on $\cF$ (of course here the choice of coefficients for $\Delta$ is crucial), and for a $\QQ$-line bundle $\cL$ on $\cF$ we let $R(\cF,\cL)$ be the graded ring of sections associated to $\cL$.
In  the example of~\Ref{subsec}{degtwo}, Equation~\eqref{gitdel} holds by~\eqref{keypull}.
On the other hand, Baily-Borel's compactification is characterized as 
$$\cF^{*}=\Proj R(\cF,\lambda).$$
Thus, in order to analyze the period map, we must examine $\Proj R(\cF,\lambda+\beta\Delta)$ for $\beta\in(0,1)\cap\QQ$ (we assume throughout that $R(\cF,\lambda+\beta\Delta)$ is finitely generated). Let us first consider the two extreme cases: $\beta$ close to $0$ or to $1$, that we denote $\beta=\epsilon$ and $\beta=(1-\epsilon)$, respectively.
The space 
$$\widehat \cF:=\Proj R(\cF,\lambda+\epsilon \Delta)$$
constructed by Looijenga \cite{looijengacompact} as a semi-toric compactification,  has the effect of making $\Delta$ $\bQ$-Cartier (notice that the period space $\cF$ is $\QQ$-factorial, the problems occur only at the Baily-Borel boundary). The map $\widehat \cF\to \cF^*$ 
is a small map - in  the example of~\Ref{subsec}{degtwo} this is the map $\Pi\colon\wh{\cF}_2\to \cF_2^{*}$. 
At the other extreme, we expect that $\widetilde \gM:=\Proj R(\cF,\lambda+ (1-\epsilon) \Delta)$ is 
 a Kirwan type blow-up of the GIT quotient $\gM$ with exceptional divisor  the strict transform of $\Delta$ - in  the example of~\Ref{subsec}{degtwo} this is the map $\Phi\colon\wh{\cF}_2\to \gM_2$. 

In between, we expect  a series of flips, dictated  by the structure of the preimage of $\Delta$ under the quotient map $\pi\colon \cD\to\cF$. More precisely, let $\cH:=\pi^{-1}(\supp\Delta)$; then $\cH$  is a union of hyperplane sections of $\cD$, and hence is stratified by closed subsets, where a stratum is determined by the number of independent sheets (\lq\lq independent sheets\rq\rq\ means that their defining  equations have linearly independent  differentials) of
 $\cH$ containing the general point of the stratum. The stratification of $\cH$ induces a stratification of $\supp\Delta$, where the strata of $\supp\Delta$ are indexed by the \lq\lq number of sheets\rq\rq\ (in $\cD$, \emph{not} in $\cF=\Gamma\backslash\cD$). Roughly speaking, Looijenga predicts that a stratum of $\supp\Delta$ corresponding to $k$ (at least) sheets meeting (in $\cD$) is flipped to a dimension $k-1$ locus on the GIT side. In  the example of~\Ref{subsec}{degtwo}, the divisor $\cH:=\pi^{-1}H_u$ is smooth, and this is the reason why  no flips appear in the resolution of $\gp$ given by~\eqref{triangolo}. In~\Ref{subsec}{quattrocubica} we give an example in which one flip occurs.

Summarizing, Looijenga predicts that in order to resolve the inverse of the period map $\gp$ one has to follow the steps below:
\begin{enumerate}
\item $\bQ$-factorialize $\Delta$.
\item Flip the strata of $\Delta$ defined above, starting from the lower dimensional strata, 
\item Contract the strict transform of $\Delta$.
\end{enumerate}
All these operations have arithmetic origin, and thus, when applicable, give a meaningful stratification of the GIT moduli space. 
\subsection{Cubic fourfolds}\label{subsec:quattrocubica}
The period space is  similar to that of degree-$2$ polarized $K3$ surfaces (see~\eqref{howmanytimes}). Specifically,  $\Lambda$ is replaced by 
$\Lambda':=U^2\oplus E_8^2\oplus A_2$, and the arithmetic group is $\wt{O}^{+}(\Lambda'):=\wt{O}(\Lambda')\cap O^{+}(\Lambda')$, where $\wt{O}(\Lambda')$ is the stable orthogonal group. The divisor $\Delta$ is $H_u/2$, where this time $H_u$ is the image in $\cF$ of $v^{\bot}\cap\cD$ for $v\in\Lambda$ such that $q(v)=-6$ and $\divisore(v)=3$. 
In this case at most two sheets of $\cH:=\pi^{-1}H_u$ meet, and correspondingly 
 there is exactly one flip $f$, fitting into the diagram
$$
\xymatrix{
\widehat{\cF}\ar@{-->}[rr]^{f}\ar@{->}[d]_{\Pi}&&\widetilde\gM \ar@{->}[d]^{\Phi}&\\ 
\cF^*\ar@{-->}^{\gp^{-1}}[rr]&& \gM&
}
$$
Here, $\Phi$ is the blow-up of the polystable point corresponding to the secant variety of a Veronese surface. The map $f$ is the flip of the codimension $2$ locus  where two sheets of $\cH:=\pi^{-1}H_u$ meet, and the corresponding locus in $\gM$ is  the curve parametrizing cubic fourfolds singular along a rational normal curve. For a detailed treatment, see~\cite{lcubic,gitcubic,cubic4fold}.   
\subsection{Periods of polarized $K3$'s of degree $4$ according  to~\cite{log1}} \label{subsec:predictq}
We start by recalling notation and constructions from~\cite{log1}.   For $N\ge 3$, let $\Lambda_N:=U^2\oplus D_{N-2}$. In~\cite{log1} we defined a group $\wt{O}^{+}(\Lambda_N)<\Gamma_N< O^{+}(\Lambda_N)$ which is equal to $O^{+}(\Lambda_N)$ if $N\not\equiv 6\pmod{8}$,  and is of index $3$ in  $O^{+}(\Lambda_N)$ if $N\equiv 6\pmod{8}$, see
Proposition~1.2.3 of~\cite{log1}. Next, we let
\begin{eqnarray}\label{ancora}
\cD_N & := & \{[\sigma]\in\PP(\Lambda_{N}\otimes\CC) \mid q(\sigma)=0,\quad q(\sigma+\ov{\sigma})>0\}^{+},\\
\cF(N)  & := & \Gamma_N\backslash \cD_N.
\end{eqnarray}
(The meaning of the superscript $+$ is as in~\eqref{howmanytimes}.) Then $\cF:=\cF(19)$ is the period space for polarized $K3$'s of degree $4$ - we will explain the relevance of the other $\cF(N)$ at the end of the present subsection. Let  $(X,L)$ be a polarized $K3$ surface of degree $4$; we let and $\gp(X,L)\in\cF$ be its period point.

The \emph{hyperelliptic}   divisor $H_h\subset\cF$ is  the image  of $v^{\bot}\cap\cD_{19}$ for $v\in\Lambda_{19}$ such that $q(v)=-4$, and $\divisore(v)=2$  
(any two such $v$'s are $O^{+}(\Lambda_{19}$)-equivalent).  Let  $(X,L)$ be a polarized $K3$ surface of degree $4$;  then $\gp(X,L)\in H_h$ if and only if $(X,L)$ is hyperelliptic, i.e.~$\varphi_L\colon X\dra |L|^{\vee}$ is  a regular map of degree $2$ onto a quadric - this explains our terminology. 

 The  \emph{unigonal}  divisors   $H_u\subset \cF$,  is  the image  of $v^{\bot}\cap\cD_{19}$ for $v\in\Lambda$ such that $q(v)=-4$, and $\divisore(v)=4$ (any two such $v$'s are $O^{+}(\Lambda_{19}$)-equivalent). 
 If  $(X,L)$ is a polarized $K3$ surface of degree $4$,  then  $\gp(X,L)\in H_u$ if and only if $(X,L)$ is unigonal, i.e.~$L\cong \cO_X(A+3B)$, where $B$ is an elliptic curve and $A$ is a section of the elliptic fibration $|B|$.

  We let $\Delta:=(H_h+H_u)/2$. 
For $k\ge 1$, let $\Delta^{(k)}\subset\supp\Delta$ be the $k$-th stratum of the stratification defined in~\Ref{subsec}{looijframe}, i.e.~the closure of the image of the locus in $\cH:=\pi^{-1}(\supp\Delta)$ where $k$ (at least) independent sheets of $\cH$ meet. One has $\Delta^{(19)}\not=\es$, and there is a strictly increasing ladder $\Delta^{(19)}\subsetneq\Delta^{(18)}\subsetneq\ldots \subsetneq\Delta^{(1)}=(H_h\sqcup H_u)$. This is in stark contrast with the cases discussed above: in fact (with analogous notation)  in the case   of degree $2$ $K3$ surfaces one has $\Delta^{(k)}=\es$ for $k\ge 2$, and  in the case  of cubic fourfolds  one has $\Delta^{(k)}=\es$ for $k\ge 3$. In fact, since for quartic surfaces there are $0$-dimensional strata of $\Delta$, 
strictly speaking Looijenga's theory  does not apply (see Lemma 8.1 in~\cite{looijengacompact}). Our refinement in~\cite{log1} takes care of this issue and, at least to first order, Looijenga's framework still applies, as we proceed to explain. For the rest of the paper, the GIT moduli space $\gM$ is that of quartic surfaces:
\begin{equation*}
\gM:=|\cO_{\PP^3}(4)|\gquot\PGL(4).
\end{equation*}
 Of course, we do not \lq\lq see\rq\rq\ hyperelliptic polarized $K3$'s of degree $4$  among quartic surfaces, nor do we see unigonal polarized $K3$'s of degree $4$ - and that is where all the action takes place.
 Let $\lambda$ be the Hodge $\QQ$-Cartier divisor class on $\cF$.    The period map $\gp\colon\gM\dra \cF^{*}$ (denoted $\gp_{19}$ in~\cite{log1}) is birational by Global Torelli, and it defines an isomorphism
\begin{equation*}
 \gM\cong\Proj R(\cF,\lambda+\Delta)
\end{equation*}
 by Proposition~4.1.2 of~\cite{log1}.  On the other hand, the Baily-Borel compactification $\cF^{*}$ is identified with $\Proj R(\cF,\lambda)$. 
For $\beta\in[0,1]\cap\QQ$, we let
 $$\cF(\beta)=\Proj R(\cF,\lambda+\beta \Delta).$$
\begin{equation}\label{diagramflip}
\xymatrix @R=.07in @C=.07in{
\scriptstyle \widehat\cF\cong \cF(0,\frac{1}{9})\ar[dddd]_{\scriptstyle \Pi}\ar[rdd] & & \scriptstyle \cF ( \frac{1}{9} , \frac{1}{7} ) \ar@{<--}[ll]\ar@{-->}[rr]\ar[ldd] \ar[rdd] & & \scriptstyle \cdots \ar[ldd] & \scriptstyle  \cF ( \frac{1}{m+1} , \frac{1}{m} ) \ar[rdd]\ar@{-->}[rr] & &  
\scriptstyle  \cF ( \frac{1}{m} , \frac{1}{m-1} ) \ar[ldd] &\scriptstyle \cdots  & \scriptstyle  \cF( \frac{1}{3} , \frac{1}{2} )  \ar[rdd] \ar@{-->}[rr] & & \scriptstyle \cF (\frac{1}{2},1)\cong\widetilde \gM\ar[ldd] \ar[dddd]^{\scriptstyle \Phi} \\
&&&&&&\\
& \scriptstyle \cF(\frac{1}{9}) & & \scriptstyle \cF( \frac{1}{7} ) & & & \scriptstyle \cF( \frac{1}{m} )  & & &  & \scriptstyle \cF(\frac{1}{2})&&\\
&&&&&&\\ 
\scriptstyle \cF^*\cong \cF(0) & & & & & & & & & & & \scriptstyle \cF(1)\cong \gM   
 }
 \end{equation}
The predictions of~\cite{log1} are as follows. First, we expect that $R(\cF,\lambda+\beta \Delta)$ is finitely generated for all $\beta\in[0,1]\cap\QQ$, and that the critical values of $\beta\in[0,1]\cap\QQ$ are given by
\begin{equation}
\beta\in\left\{0,\frac{1}{9},\frac{1}{7},\frac{1}{6},\frac{1}{5},\frac{1}{4},\frac{1}{3},\frac{1}{2},1\right\}.
\end{equation}
(Note: $\beta=1/8$ is missing, \emph{no typo}.) This means that for $\beta_i<\beta\le \beta'<\beta_{i+1}$, where  $\beta_i,\beta_{i+1}$ are consecutive critical values, the birational map $\cF(\beta)\dra\cF(\beta')$ is an isomorphism. We let 
\begin{equation}\label{intmodel}
\cF(\beta_i,\beta_{i+1}):=\cF(\beta),\qquad \beta\in[\beta_i,\beta_{i+1}]\cap\QQ.
\end{equation}
As we have already mentioned, 
$\cF(\epsilon)$ is expected to be the $\QQ$-factorialization of $\cF^{*}$. On the other hand, $\cF(1-\epsilon)$ is the blow-up of $\gM$ 
with center a scheme supported on  the two points representing the tangent developable of a twisted cubic curve, and a double (smooth) quadric. For later reference  
 we denote by $\upsilon$ and $\omega$ the corresponding points of $\gM$; explicitly
 \begin{eqnarray}
\upsilon & := & [V(4(x_1 x_3-x_2^2)(x_0 x_2-x_1^2)-(x_1 x_2-x_0 x_3)^2)], \label{tangsvil} \\
\omega & := & [V((x_0^2+x_1^2+x_2^2+x_3^2)^2)]. \label{quadop}
\end{eqnarray}
We predict that one goes from $\cF(\epsilon)$ to $\cF(1-\epsilon)$ via a stratified flip, summarized in~\eqref{diagramflip}. More precisely, in~\cite{log1} we have defined a tower of closed subsets
\begin{equation}\label{zetaquart}
Z^{9}\subset Z^{8} \subset Z^{7}\subset Z^{5}\subset Z^{4}\subset Z^{3}\subset Z^{2}\subset Z^{1}=H_u\cup H_h\subset \cF,
\end{equation}
where  $k$ denotes the codimension ($Z^6$ is missing, \emph{no typo}). In fact, with the notation of~\cite{log1}, 
\begin{itemize}
\item for $k\le 5$, $Z^k=\Delta^{(k)}$,
\item $Z^7=\im(f_{13,19}\circ q_{13}\colon \cF(\II_{2,10}\oplus A_2)\hra\cF)$, 
\item $Z^8=\im(f_{12,19}\circ m_{12}\colon \cF(\II_{2,10}\oplus A_1)\hra\cF)$, and
\item $Z^9=\im(f_{11,19}\circ l_{11}\colon \cF(\II_{2,10})\hra\cF)$ ($Z^9$ is one of the two components of $\Delta^{(9)}$). 
\end{itemize}
Let $m\in\{2,3,\ldots,7,9\}$; we predict that the birational map $\cF (a(m) , \frac{1}{m} )\dra \cF ( \frac{1}{m} , \frac{1}{m-1} )$ (here $a(m)= \frac{1}{m+1}$ if $m\not=7,9$, $a(7)=1/9$, and $a(9)=0$) is a flip with center the strict transform of (the closure) of $Z^k$, where $k=m$, except  for $m=7,6$, in which case $k=m+1$. 
Thus we expect that $Z^k$ is replaced by a closed $W_{k-1}\subset\gM$ of dimension $k-1$. Correspondingly, we should have a stratification of the indeterminacy locus $\Ind(\gp)$  of the period map.
Now, according to Shah, the indeterminacy locus $\Ind(\gp)$ is contained in the locus $\gM^{IV}$ parametrizing polystable quadrics of Type IV (i.e.~those which do not have slc singularities, see~\Ref{subsec}{hodgestrata}) - and it is natural to expect that $\Ind(\gp)=\gM^{IV}$. The first evidence in favor of our predictions  is  that, as we will show, $\gM^{IV}$  has a natural stratification 
\begin{equation}\label{wquart}
\scriptstyle
\gM^{IV}= (W_8\sqcup\{\upsilon\})\supset (W_7\sqcup\{\upsilon\})\supset (W_6\sqcup\{\upsilon\})\supset (W_4\sqcup\{\upsilon\})\supset 
(W_3\sqcup\{\upsilon\})\supset (W_2\sqcup\{\upsilon\})\supset (W_1\sqcup\{\upsilon\})\supset (W_0\sqcup\{\upsilon\}),
\end{equation}
where $W_0=\{\omega\}$, and each $W_i$ is (closed) irreducible of dimension $i$. It is well known that  the period map \lq\lq improves\rq\rq\ on the blow-up $\wt{\gM}$ of a certain subscheme of $\gM$ supported on 
$\{\upsilon,\omega\}$. More precisely,  it is regular on the exceptional divisor over $\upsilon$, with image the closure of the unigonal divisor $H_u$, it is regular on the dense open subset of the exceptional divisor over $\omega$ parametrizing double covers of $\PP^1\times\PP^1$ ramified over a  curve with ADE singularities (the exceptional divisor  over $\omega$ is the GIT quotient of $|\cO_{\PP^1\times\PP^1}(4,4)|$ modulo $\Aut(\PP^1\times\PP^1)$),  mapping it to $H_h\setminus\Delta^{(2)}$. This is discussed with much more detail than previously available in the literature (e.g. \cite{shah4}) in~\Ref{subsec}{sblowupsilon} and~\Ref{subsec}{sblowomega} respectively. In~\Ref{subsec}{tropforest} we identify $\wt{\gM}$ with $\cF(1-\epsilon)$, for small $\epsilon>0$. 
\Ref{subsec}{rescodim2} is devoted to a  proof (without full details) that the blow up of  a suitable scheme supported on  the strict transform of $W_1$ in $\wt{\gM}$ can be contracted to produce $\cF(1/2)$.
Lastly, in~\Ref{sec}{dolgachev}, we give evidence that $W_{k-1}$ is related to $Z^k$ as predicted, for $k\in\{7,8,9\}$. Namely,  $Z^9, Z^8, Z^7$ correspond precisely to  $T_{2,3,7}, T_{2,4,5}, T_{3,3,4}$ marked $K3$ surfaces respectively, while $W_6, W_7, W_8$ correspond to the equisingular loci of quartics with $E_{14},E_{13}, E_{12}$  singularities respectively (on the GIT side). The flips replacing $W_{k-1}$ with $Z^k$ (in this range) are analogous to the semi-stable replacement that occurs for curves in the Hassett--Keel program (e.g. curves with cusps are replaced by stable curves with elliptic tails).

We end the subsection by going back to $\cF(N)$ for arbitrary $N\ge 3$. First, there are other values of $N$ for which $\cF(N)$ is the period space for geometrically meaningful varieties of $K3$ type. In fact, $\cF(18)$ is the period space  for hyperelliptic polarized $K3$'s of degree $4$, and $\cF(20)$ is the period space for double EPW sextics~\cite{epwperiods} (modulo the duality involution), and of EPW cubes~\cite{ikkr}. Secondly, there is a \lq\lq hyperelliptic divisor\rq\rq\ on $\cF(N)$ for arbitrary $N$ (and a \lq\lq unigonal\rq\rq\ divisor on $\cF(N)$ for $N\equiv 3\pmod{8}$). More precisely, if $N\not\equiv 6\pmod{8}$  the \emph{hyperelliptic}   divisor $H_h(N)\subset\cF(N)$ is  the image  of $v^{\bot}\cap\cD$ for $v\in\Lambda_N$ such that $q(v)=-4$, and $\divisore(v)=2$,   if   $N\equiv 6\pmod{8}$ the definition of the hyperelliptic divisor is subtler (there  is a link with the fact that  $[O^{+}(\Lambda_N):\Gamma_N]=3$). The key aspect of our analysis in~\cite{log1} is that we have a tower of locally symmetric spaces
\begin{equation}\label{ditorre}
\ldots\hra\cF(18)\overset{f_{19}}{\hra} \cF(19)\overset{f_{20}}{\hra} \cF(20)\hra\ldots \hra\cF(N-1)\overset{f_{N}}{\hra} \cF(N)\hra\ldots
\end{equation}
where $\cF(N-1)$ is embedded into $\cF(N)$ as the hyperelliptic divisor $H_h(N)$.  Our paper~\cite{log1} contains analogous predictions  for the behaviour  of 
$\Proj R(\cF(N),\lambda(N)+\Delta(N))$, where $\lambda(N)$ is the Hodge $\QQ$-Cartier divisor class, and $\Delta(N)$ is a $\QQ$-Cartier boundary divisor class (equal to $\Delta$ for $N=19$), which are compatible with the tower~\eqref{ditorre}.  Thus the period map and birational geometry of $\cF=\cF(19)$  sits between $\cF(18)$, i.e.~the period space of hyperelliptic quartic surfaces,  and $\cF(20)$, i.e.~the period space of double EPW sextics (modulo the duality involution), or equivalently that of EPW cubes.
 
%%%%%%%%%%%%%%%%%%%%%%%%%%%%%%%%%%%%%%%
%%% Revisit Shah

\section{GIT and Hodge-theoretic stratifications of $\gM$}\label{sec:shahreview}
\setcounter{equation}{0}
\subsection{Summary}
The analysis of GIT (semi)stability for quartic surfaces was carried out by Shah in~\cite{shah4}. In this section we will review some of his results. In particular we will go over the GIT stratification (N.B. as usual, a \emph{stratification} of a topological space $X$ is   a partition of $X$ into locally closed subsets such that   the closure of a stratum  is a union of  strata) determined by the stabilizer groups of polystable quartics. After that, we will review Shah's  Hodge-theoretic 
stratification~\cite{shah4,shahinsignificant}  
\begin{equation}\label{unduetrequattro}
\gM=\gM^{I}\sqcup \gM^{II}\sqcup \gM^{III}\sqcup \gM^{IV}
\end{equation}
 from a modern perspective (due to Steenbrink \cite{steenbrinkinsignificant}, Koll\'ar, Shepherd-Barron and others \cite{ksb}, \cite{sbnef}, \cite{kk}). 
The  period map $\gp\colon\gM\dra\cF^{*}$ extends regularly away from $\gM^{IV}$, and it maps $\gM^I$, $\gM^{II}$ and $\gM^{III}$  to the interior $\cF$, to the union of the  Type II boundary components, and to the  Type III locus (a single point) respectively. 

A large part of this paper is concerned with the behaviour of the period map for quartic surfaces along $\gM^{IV}$. 
\begin{notationconv}
Shah also defined a refinement of the stratification in~\eqref{unduetrequattro}, see Theorem 2.4 of \cite{shah4} (and \Ref{sec}{refshah} below). We will follow the notation of Theorem 2.4 of \cite{shah4}, with an S prefix, and with the symbol IV replacing \lq\lq Surfaces with significant limit singularities\rq\rq. Thus the strata
will be denoted by S-I, S-II(A,i), S-II(A,ii) S-III(B,ii), S-IV(A,i), etc. We recall that  the roman numerals I, II, III, IV refer to the stratum of~\eqref{unduetrequattro} to which a stratum belongs, 
and the letter A (B)  indicates whether the stratum is contained in the stable locus or in the properly  semistable locus. We will refer to Shah's stratification before discussing the stratification in~\eqref{unduetrequattro}; this is not an issue, because the strata are defined explicitly    by Shah in terms of singularities, see  Theorem 2.4 of \cite{shah4}.
\end{notationconv}

\subsection{The GIT (or Kirwan) stratification for quartic surfaces}
Shah~\cite{shah4} essentially established a relation between GIT (semi)stability  of a quartic surface and the nature of its singularities. In particular he proved that a quartic with ADE singularities is stable, and hence there is an open dense subset $\gM^I\subset\gM$ parametrizing isomorphism classes of polarized $K3$ surfaces $(X,L)$ such that $L$ is very ample, i.e.~$(X,L)$ is neither hyperelliptic, nor unigonal - see~\Ref{thm}{circeo}. In the present subsection the focus  is on stabilizers (in $\SL(4)$) of properly semistable polystable quartics (a quartic is \emph{properly semistable} if it semistable and not stable, it it \emph{polystable} if its $\PGL(4)$-orbit is closed in the semistable locus $|\cO_{\PP^3}(4)|^{ss}$), and the associated stratification of $\gM$. The point of view is essentially due to Kirwan \cite{kirwan,kirwanhyp}.
Let $\gM^s\subset\gM$ be the open dense subset parametrizing isomorphism classes of GIT stable quartics. Points of the \emph{GIT boundary} $\gM\setminus \gM^s$ parametrize isomorphism classes of semistable polystable quartics. The stabilizer of such an orbit is a positive dimensional reductive subgroup.
  The classification of $1$-dimensional stabilizers leads to the decomposition of the GIT boundary    into irreducible components. 
\begin{lemma}\label{lmm:lem1ps}
Let $X=V(f)$ be a strictly semistable polystable quartic. Then $f$ is stabilized by one of the following four  $1$-PS's of  $\SL(4)$ (up to conjugation): 
\begin{equation}\label{1234}
\lambda_1=(3,1,-1,-3),\ \lambda_2=(1,0,0,-1), \ \lambda_3=(1,1,-1,-1), \ \lambda_4=(3,-1,-1,-1). 
\end{equation}
For $i=1,\dots,4$, let $\sigma_i\subset\gM$ be the closed subset parametrizing polystable points stabilized by $\lambda_i$. Then the following hold:
\begin{enumerate}
\item
 $\sigma_1,\ldots,\sigma_4$ are the irreducible components of the GIT boundary $ \gM\setminus \gM^s$.
\item
 The  $\sigma_i$'s are related to  Shah's stratification as follows: 
 \begin{enumerate}
 \item $\sigma_1$ is the closure of the $\widetilde E_8$ component  (see  B, Type II, (i)  in Theorem 2.4 in~\cite{shah4}) of S-II(B,i).  
 \item 
 $\sigma_2$ is the closure of the   $\widetilde E_7$ component  (see  B, Type II, (i)  in Theorem 2.4 in~\cite{shah4})  of S-II(B,i). 
 \item $\sigma_3=\overline{\textrm{S-II(B,ii)}}$.
 \item $\sigma_4=\overline{\textrm{S-II(B,iii)}}$.
 \end{enumerate}
\item 
$\dim \sigma_1=2$, $\dim \sigma_2=4$, $\dim \sigma_3=2$, and $\dim \sigma_4=1$.
\end{enumerate}
\end{lemma}
\begin{proof}
This follows from Proposition 2.2 of~\cite{shah4} (see also Kirwan \cite[\S4]{kirwanhyp} for a discussion focused on stabilizers). Specifically, the first 3 cases correspond to $1$-PS subgroups of type $(n,m,-m,-n)$ (i.e. Case (1) in loc. cit.). Thus $\lambda_1,\lambda_2,\lambda_3$ correspond to (1.1), (1.2), and (1.3) respectively in Shah's analysis. The last case, $\lambda_4$ corresponds to the cases (2.1) or (4.1) of Shah (N.B. the two cases are dual, so they result in a single case in our lemma; the previous case (1) is self-dual). It is easy to see that the other cases in Shah's analysis can be excluded (i.e.~either they lead to unstable points, or to cases that are already covered by one of $\lambda_1,\dots,\lambda_4$ -- it is possible to have a polystable orbit stabilized by another $1$-PS $\lambda$, but then the stabilizer contains a higher dimensional torus, which in turn contains a conjugate of one of $\lambda_1,\dots,\lambda_4$). In conclusion, the GIT boundary consists of the $4$ boundary components $\sigma_i$ as stated (they intersect, but none is included in another). 

Item (B) of Theorem 2.4 of Shah \cite{shah4} describes the strictly polystable locus in the GIT compactification. It is clear (from the geometric description and proofs) that the strictly semistable locus in $\gM$ is the closure of the Type II strata, i.e. 
$$ \gM\setminus \gM^s=\cup_{i=1,4} \sigma_i=\overline{\textrm{S-II(B,i)}}\cup \overline{\textrm{S-II(B,ii)}}\cup \overline{\textrm{S-II(B,iii)}}.$$
Finally, the stratum S-II(B,i) has two components corresponding to quartics with two $\widetilde E_8$ singularities and two $\widetilde E_7$ singularities respectively (see \cite[Thm. 2.4 (B, II(i))]{shah4} for precise definitions of the two cases). 

In order to compute the dimensions, one can write down normal forms for the quartics stabilized by the $1$-PS $\lambda_i$. For instance, it is immediate to see that a quartic stabilized by $\lambda_4=(3,-1,-1,-1)$ is of the form $x_0f_3(x_1,x_2,x_3)$ (i.e. the union of the cone over a cubic curve with a transversal hyperplane, or same as S-II(B,iii)). Furthermore, we can still act on this equation with the centralizer of $\lambda_4$ in $\SL(4)$. In particular, with $\SL(3)$ acting on the variables $(x_1,x_2,x_3)$. It follows that the dimension in $\gM$ of the locus of polystable points with stabilizer $\lambda_4$ (i.e. $\sigma_4$) is $1$. At the other extreme, we have the case $\lambda_1=(3,1,-1,-3)$. In this case, the centralizer is the maximal torus in $\SL(4)$. There are five degree $4$ monomials  stabilized by $\lambda_1$, namely $x_0x_2^3,x_1^3x_3, (x_0x_3)^a(x_1x_2)^b$ with $a+b=2$.  It follows that $\dim \sigma_1=2$.  The other cases are similar. 
\end{proof}
Note that $\sigma_1,\ldots,\sigma_4$ are closed subsets of $\gM$. As a general rule, subsets of $\gM$ denoted by Greek letters are closed.

The intersections of the components of the GIT boundary are determined by considering stabilizers that are tori of dimension larger than $1$. More in general, special strata inside the $\sigma_i$ are determined by other reductive (non-tori) stabilizers. The stratification of GIT quotients in terms of stabilizer subgroups plays an essential role in the work of Kirwan \cite{kirwan}, and the case of hypersurfaces of low degree  was analyzed in \cite{kirwanhyp}.  Given a quartic $X$, we let $\Stab(X)<\SL(4)$ be the stabilizer of $X$, and we let 
$\Stab^0(X)<\Stab(X)$ be the connected component of the identity.

We start by noting that we have already defines two points which are GIT strata, namely $\upsilon$ and $\omega$, see~\eqref{tangsvil} and~\eqref{quadop}. 
\begin{remark}
Let $X\subset\PP^3$ be the tangent developable of a twisted cubic curve, thus in suitable homogeneous coordinates the equation of $X$ is given in the right hand side of~\eqref{tangsvil}. Then $X$ is a properly semistable polystable quartic, and the corresponding point in $\gM$ is denoted by $\upsilon$. The group $\Aut^0(X)$ is conjugated to $\SL(2)$ embedded in 
$\SL(4)$ via the $\Sym^3$ representation. 
\end{remark}
\begin{remark}
Let $X\subset\PP^3$ be twice a smooth quadric, thus in suitable homogeneous coordinates the equation of $X$ is given in the right hand side of~\eqref{quadop}. Then $X$ is a properly semistable polystable quartic, and the corresponding point in $\gM$ is denoted by $\omega$. The group $\Aut^0(X)$ is conjugated to $\SO(4)$. 
\end{remark}
The following result is due to Kirwan (and essentially contained also in \cite{shah4}). 
\begin{proposition}[{Kirwan~\cite[\S6]{kirwanhyp}}]\label{prp:classifystab}
Let $X$ be a properly semistable  polystable quartic. Then  $\Aut^0(X)$ is one of the following (up to conjugation):
\begin{enumerate}
\item 
The trivial group $\{1\}$ (i.e.~$X$ is stable).
\item 
One of the $1$-PS's $\lambda_1,\ldots,\lambda_4$ in~\eqref{1234}.
\item 
The two-dimensional torus $\diag(s,t,t^{-1},s^{-1})\subset\SL(4,\bC)$. Equivalently, $X=Q_1+Q_2$ where $Q_1,Q_2$ are smooth quadrics meeting along $2$ pairs of skew lines (special case of S-III(B,ii)). 
Let $\tau\subset\gM$  be the \emph{closure} of the set of points representing such quartics.  Then $\tau$ is a curve, and
$$\tau=\sigma_1\cap\sigma_2\cap\sigma_3$$
(in fact $\tau$ is the intersection of any two of the $\sigma_1,\sigma_2,\sigma_3$). 
\item 
The maximal torus in $\SL(4,\bC)$. Equivalently, $X$ is a tetrahedron (S-III(B,i)). We let $\zeta\in \gM$ be the corresponding point.  Then 
$$\{\zeta\}=\sigma_1\cap\sigma_2\cap\sigma_3\cap\sigma_4.$$
\item 
$\SO(3,\bC)$, or equivalently   $X=Q_1+Q_2$ where $Q_1,Q_2$ are  quadrics  tangent along a smooth conic (S-IV(B,ii)).  This defines a curve $\chi\subset \sigma_2\subset \gM$. The only incidence with the other strata is $\chi\cap \tau=\{\omega\}$. 

\item
 $\SL(2,\bC)$ embedded in $\SL(4)$ via the $\Sym^3$-representation. 
Then $X$ is the tangent developable of a twisted cubic curve (special case of S-IV (B,i)), and $\upsilon$ is  be the corresponding point  in $\gM$. 
One has $\upsilon\in \sigma_1$, and $\upsilon\notin \sigma_i$ for $i\in\{2,3,4\}$. 
\item 
$\SO(4,\bC)$. Then $X=2Q$, where $Q$ is a smooth quadric (S-IV(B,iii)), and  $\omega$ is   the corresponding point in  $\gM$. Then $\omega\in \tau$, and thus $\omega\in \sigma_1\cap\sigma_2\cap\sigma_3$ (and $\omega\not\in \sigma_4$).

\end{enumerate}
\end{proposition}
\begin{proof}[Elements of the proof.]
We refer to Kirwan~\cite[\S6]{kirwanhyp} for the complete proof. Here we only describe polystable quartics parametrizd by $\tau$ and $\chi$. 
First we consider $\tau$. If $X$ consists  of two quadrics meeting in two pairs of skew lines, then (in suitable homogeneous coordinates)   
it has equation 
$$(a_1 x_0x_3+b_1 x_1x_2)(a_2 x_0x_3+b_2 x_1x_2)=0.$$
Clearly this is a pencil, and we have the following two special cases:
\begin{itemize}
\item  the tetrahedron (case $\zeta$) if any of the $a_i$ or $b_i$ vanish;
\item the double quadric (case $\omega$) if $[a_1,b_1]=[a_2,b_2]\in\PP^1$.
\end{itemize}
If both $a_i$ or $b_i$ vanish simultaneously, the associated quartic is unstable and thus the two cases above are distinct. 

Next we consider $\chi$. If $X$ consists  of two quadrics  tangent along a conic,  then (in suitable homogeneous coordinates)   
it has equation 
$$f_{a,b}:=(q(x_0,x_1,x_2)+a x_3^2)(q(x_0,x_1,x_2)+b x_3^2)=0,$$
for $[a,b]\in \bP^1$ (N.B. if $a=b=0$, one gets the double quadric cone, which is unstable; similarly, if $a=\infty$ or $b=\infty$, one gets an unstable quadric). Note that $f_{a,a}=0$  is the equation of the double (smooth) quadric (case $\omega$). 
\end{proof}
\subsection{The stratification  by Type}\label{subsec:hodgestrata}
Shah \cite{shahinsignificant},  influenced by Mumford, defined the  concept  of \lq\lq {\it insignificant limit singularity}\rq\rq, and  used it to study the period map for degree $2$ and degree $4$ $K3$ surfaces (see \cite{shah,shah4}). 
 One defines $\gM^{IV}$ as  the subset of $\gM$ parametrizing quartics with  \emph{significant} limit singularities.  
 The main point is that the restriction of the period map to $(\gM\setminus\gM^{IV})$  is regular. Next,  
 \begin{equation}\label{romani}
 \cF^{*}=\cF\sqcup\cF^{II}\sqcup\cF^{III},
\end{equation}
   where $\cF^{II}$ is the union of the Type II boundary components, and  $\cF^{III}$ is the (unique) Type III boundary component; this is a  stratification of $\cF^{*}$. Then~\eqref{romani} 
   defines, by pull-back via $\gp$, strata $\gM^I$, $\gM^{II}$ and $\gM^{III}$ (of course  $\gM^I$ coincides with the set that we have already defined).  (Literally speaking, we will not define 
   $\gM^I$, $\gM^{II}$ and $\gM^{III}$ this way.)

We will  give an updated view of the concept of  insignificant limit singularity. Briefly, Steenbrink \cite{steenbrinkinsignificant} noticed that an insignificant limit singularity
is du Bois. On a different track, from the  perspective of moduli, Shepherd-Barron \cite{sbnef} and then Koll\'ar--Shepherd-Barron \cite{ksb} noticed that the right notion of singularities is that of semi-log-canonical (slc) singularities. More recently (with \cite{kk} as the last step), it was proved that an slc singularity is du Bois. Lastly, one can check by direct inspection that Shah's list of insignificant singularities coincides with the list of Gorenstein slc surface singularities (which are then du Bois). Of course, in the situation studied here, this is just a long-winded highbrow reproof of Shah's results from 1979, but what is gained is a conceptual understanding of the situation. 

We should also point out the connection between slc singularities and GIT. On one hand, an easy observation (\cite{hacking},\cite{kimlee}) shows that a quartic with slc singularities is GIT semistable. A much deeper result (due to Odaka \cite{odakacy,odaka}), which can be viewed as some sort of converse of this, is giving a close connection between slc singularities and $K$-stability. Finally, $K$-stability should be viewed as a refined notion of asymptotic stability. We caution however that the precise connection between asymptotic stability and K-stability/slc for $K3$ surfaces is not known. More precisely, an example of Shepherd-Barron \cite{sbnef,sbpolarization} shows that for $K3$s of big enough degree there is no (usual) asymptotic GIT stability. The results of \cite{wangxu} strengthen the meaning of this failure of asymptotic stability. Nonetheless, it is still possible that a certain (weaker) asymptotic stabilization exists. We  hope that  our HKL program will  eventually address this issue. 
\subsubsection{ADE singularities}
We recall that  $\gM^{I}\subset\gM$ is (by definition)  the subset parametrizing isomorphism classes of quartics with  ADE singularities.  
The following identification of $\gM^{I}$ (as a quasi-projective variety) with an open  subset of the projective variety $\cF^*$ is well known:
\begin{theorem}\label{thm:circeo}
The period map defines an isomorphism
$$\gM^I\overset{\sim}{\lra} (\cF\setminus H_h\setminus H_u).$$
\end{theorem}

\subsubsection{Insignificant Limit Singularities} 
We recall the following important result about slc singularities.

\begin{theorem}[{Koll\'ar--Kov\'acs~\cite{kk}, Shah~\cite{shahinsignificant} (for dimension 2)}]\label{thm:dubois}
Let $X_0$ be a projective reduced variety (not necessarily irreducible)  with slc singularities. Then $X$ has du Bois singularities. In particular, if $\sX/B$ is a smoothing of $X_0$ over a pointed smooth curve $(B,0)$, then the natural map $H^n(X_0)\to H^n_{\lim}$
induces an isomorphism
$$I^{p,q}(X_0)\cong I^{p,q}_{\lim}$$
on the $I^{p,q}$ components of the MHS with $p\cdot q=0$. 
\end{theorem}
The key point (for us) of the above result  is that, if the generic fiber of $\sX/B$ is a (smooth) $K3$ surface, then the MHS of the central fiber $X_0$  essentially determines the limit MHS associated to $\sX^*/(B\setminus\{0\})$. This is a result due to Shah~\cite{shahinsignificant} in dimension $2$ and Gorenstein singularities (the case relevant for us). Steenbrink~\cite{steenbrinkinsignificant}
connected this result to the notion of du Bois singularities. 
 \begin{definition}
A reduced  (not necessarily irreducible)  projective surface $X_0$ is a \emph{degeneration of $K3$ surfaces} if it is the central fiber of a flat proper family $\sX/B$ over a pointed smooth curve $(B,0)$  such that $\omega_{\sX/B}\equiv 0$ and the general fiber $X_b$ is a smooth $K3$ surface. 
 We say that $X_0$ has \emph{insignificant limit singularities} if $X_0$ has semi-log-canonical singularities.
 \end{definition}
 
 \begin{remark}
 The list of singularities baptized as \emph{insignificant limit singularities} by Shah \cite{shahinsignificant} coincides with the list of Gorenstein slc singularities (see \cite{sbnef}, \cite{ksb}). For a degeneration of $K3$ surfaces, the Gorenstein assumption is automatic. 
 \end{remark}

 Let $X_0$ be a degeneration of $K3$ surfaces with insignificant singularities. On $H^2(X_0)$ we have a MHS of weight $2$. Denote by $h^{p,q}$ the associated Hodge numbers ($h^{p,q}=\dim_\bC I^{p,q}$). \Ref{thm}{dubois} gives that  one, and only one, of the following $3$ equalities holds:
\begin{itemize}
\item 
$h^{2,0}(X_0)=1$.
\item 
$h^{1,0} (X_0)=1$.
\item 
$h^{0,0} (X_0)=1$.
\end{itemize}
In fact this follows from the isomorphism of the theorem, and the fact that $h^{2,0}_{\lim}+h^{1,0}_{\lim}+h^{0,0}_{\lim}=1$ for a degeneration of $K3$'s.
\begin{definition} 
Let $X_0$ be a degeneration of $K3$s.
\begin{enumerate}
\item
 $X_0$ has \emph{Type I} if it has  insignificant limit singularities, and $h^{2,0}(X_0)=1$. 
\item
 $X_0$ has \emph{Type II} if it has  insignificant  limit singularities,   and $h^{1,0} (X_0)=1$.
\item
 $X_0$ has \emph{Type III} if it has  insignificant  limit singularities, and $h^{0,0} (X_0)=1$.
\item
 $X_0$ has \emph{Type IV} if it has  significant  limit  singularities.
\end{enumerate}
\end{definition}

We are interested in the case of Gorenstein slc surfaces. These are classified by Koll\'ar-Shepherd-Barron \cite{ksb} and Shepherd-Barron. They are
\begin{itemize}
\item[(A)] ADE singularities (canonical case)
\item[(B)] simple elliptic singularities (for hypersurfaces the relevant cases are $\widetilde E_r$ with $r=6,7,8$), surfaces singular along a curve,  generically normal crossings (or equivalently $A_\infty$ singularities) and possibly ordinary pinch points (aka $D_\infty$).
\item[(C)] cusp and degenerate cusp singularities. 
\end{itemize}
\begin{remark}
We note that a normal crossing degeneration without triple points is a Type II degeneration, while a normal crossing degeneration with triple points is a Type III degeneration (a triple point is a particular  degenerate cusp singularity). 
\end{remark}

By applying results  of Shah \cite{shahinsignificant} and Kulikov-Persson-Pinkham's Theorem (see also Shepherd-Barron \cite{sbnef}), one obtains the following.
\begin{theorem}\label{thm:cavalier}
Let $X_0$ be a degeneration of $K3$ surfaces with insignificant singularities. Then the following hold: 
\begin{itemize}
\item[i)]  $X_0$ is of Type I if and only if it has ADE singularities,
\item[ii)]  if $X_0$ is of Type II then, with the exception of rational double  (i.e.~ADE) points, $X_0$ has a simple elliptic singularity or it is singular along a curve which is either smooth elliptic (and has no pinch points), or rational with $4$ pinch points. (N.B.~all non-ADE singularities are of this type, and at least one occurs.)	
\item[iii)] if $X_0$ is of Type III then, with the exception of   ADE  and $A_\infty$ singularities, all singularities of $X_0$ are either cusp or degenerate cusps, and at least one of these occurs. 
\end{itemize}
\end{theorem}

\begin{remark}
We recall that the Type (I, II, III) of a $K3$ degeneration is nothing else than the nilpotency index ($1$, $2$, $3$) for the monodromy action $N(=\log T_s)$ on a general fiber of a degeneration $\sX/B$. \Ref{thm}{dubois} allows us to read the Type in terms of the central fiber $X_0$ (as long as $X_0$ has slc singularities). The theorem above says that furthermore the Type of the degeneration can be determined simply by the combinatorics of $X_0$.  We point out that this fact holds much more generally - for $K$-trivial varieties (see esp. \cite[Section 2]{klsv} and \cite[Theorem 3.3.3]{hn}). 
\end{remark}
\subsubsection{The stratification and the period map}
\begin{proposition}
Let $X_0$ be a quartic surface with insignificant singularities. Then $X_0$ is GIT semistable. 
\end{proposition} 
\begin{proof}
This follows from the general fact observed by Hacking and Kim--Lee~\cite{kimlee} (see esp.~the proof of Proposition 10.2 in~\cite{hacking}):  GIT (semi)stability (via the numerical criterion) and the log canonical threshold are computed via the same recipe, with the difference that in the case of GIT  (semi)stability one allows only linear changes of coordinates (vs.~analytic in the other case).  Thus, the inequality needed for log canonicity implies the inequality needed for semistability. 
The result  also follows by inspection from Shah \cite{shah4} (i.e.~an unstable quartic does not have slc singularities). 
\end{proof}
\begin{definition}\label{dfn:limhod}
We let $\gM^I,\gM^{II},\gM^{III}\subset\gM$ be the subsets of points represented by polystable quartics with insignificant limit singularities of Type I, Type II and Type III respectively (note that $\gM^I$ is the same subset as the  previously defined $\gM^I$, by~\Ref{thm}{cavalier}). We let  
$\gM^{IV}\subset\gM$ be the subset of points represented by polystable quartics with significant limit singularities.  
\end{definition} 
Below is the result that was described at the beginning of the present section.
 \begin{proposition}\label{prp:est3}
$\gM^I,\gM^{II},\gM^{III},\gM^{IV}$ define a stratification of $\gM$. The  period map $\gp\colon\gM\to\cF^{*}$  is regular away from $\gM^{IV}$, and 
\begin{equation*}
\gp(\gM^I)\subset\cF,\qquad \gp(\gM^{II})\subset\cF^{II},\qquad  \gp(\gM^{III})\subset\cF^{III}.
\end{equation*}
(Recall that $\cF^{II}$ is the union of the Type II boundary components of $\cF^{*}$, and  $\cF^{III}$ is the (unique) Type III boundary component.)
 \end{proposition}
Before proving~\Ref{prp}{est3}, we prove a result on the period map $\wt{\gp}\colon |\cO_{\PP^3}(4)|\dra \cF^{*}$. Define subsets 
$|\cO_{\PP^3}(4)|^{I},|\cO_{\PP^3}(4)|^{II},|\cO_{\PP^3}(4)|^{III},|\cO_{\PP^3}(4)|^{IV}$ of $ |\cO_{\PP^3}(4)|$ by mimicking~\Ref{dfn}{limhod}.
 \begin{lemma}\label{lmm:est3}
The  period map $\wt{\gp}$  is regular away from $ |\cO_{\PP^3}(4)|^{IV}$, and 
\begin{equation}\label{doveva}
\wt{\gp}( |\cO_{\PP^3}(4)|^I)\subset\cF,\qquad \wt{\gp}( |\cO_{\PP^3}(4)|^{II})\subset\cF^{II},\qquad  \wt{\gp}( |\cO_{\PP^3}(4)|^{III})\subset\cF^{III}.
\end{equation}
 \end{lemma}
\begin{proof}
Let $X_0\in  (|\cO_{\PP^3}(4)|\setminus  |\cO_{\PP^3}(4)|^{IV})$ be a quartic surface. Suppose that $f\colon (B,0)\to  (|\cO_{\PP^3}(4)|,X_0)$ is a map from a smooth pointed curve, and that $f(B\setminus\{0\})$ is contained in the locus of \emph{smooth} quartics. Let $p_f^0\colon (B\setminus\{0\})\to\cF^{*}$ be the composition $\wt{\gp}\circ(f|_{B\setminus\{0\}})$, and let $p_f\colon B\to\cF^{*}$ be the extension to $B$. Then $p_f(0)$ is independent of  $f$. In fact, this follows from~\Ref{thm}{dubois}. In addition, we see that 
\begin{enumerate}
\item
if $X_0\in  |\cO_{\PP^3}(4)|^I$, then $p_f(0)\in\cF$,
\item
if $X_0\in  |\cO_{\PP^3}(4)|^{II}$, then $p_f(0)\in\cF^{II}$,
\item
and   if $X_0\in  |\cO_{\PP^3}(4)|^{III}$, then $p_f(0)\in\cF^{III}$. 
\end{enumerate}
Now suppose that $X_0\in  (|\cO_{\PP^3}(4)|\setminus  |\cO_{\PP^3}(4)|^{IV})$, and that $X_0$ is in the indeterminacy locus of $\wt{\gp}$. Then, since  $|\cO_{\PP^3}(4)|$ is smooth (normality would suffice), there exist 
 smooth pointed curves $(B_i,0_i)$ for $i=1,2$, and  maps $f_i\colon (B_i,0)\to  (|\cO_{\PP^3}(4)|,X_0)$ such that  $f(B_i\setminus\{0_i\})$  is contained in the locus of \emph{smooth} quartics, and the points $p_{f_i}(0_i)$ (defined as above) are different, contradicting what was just stated. This proves that  $\wt{\gp}$  is regular away from $ |\cO_{\PP^3}(4)|^{IV}$. Equation~\eqref{doveva} follows from Items~(1), (2), (3) above.
\end{proof}
\begin{proof}[Proof of~\Ref{prp}{est3}]
First we notice that $|\cO_{\PP^3}(4)|^{I},|\cO_{\PP^3}(4)|^{II},|\cO_{\PP^3}(4)|^{III},|\cO_{\PP^3}(4)|^{IV}$ define a stratification of $ |\cO_{\PP^3}(4)|$, because $\cF,\cF^{II},\cF^{III}$ define a stratification of $\cF^{*}$. 
Let $ |\cO_{\PP^3}(4)|^{ss}\subset  |\cO_{\PP^3}(4)|$ be the open subset of GIT semistable quartics, and let $\pi\colon |\cO_{\PP^3}(4)|^{ss}\to\gM$ be the quotient map.  By definition (and the remark about 
$|\cO_{\PP^3}(4)|^{I},\ldots,|\cO_{\PP^3}(4)|^{IV}$ defining a stratification of $ |\cO_{\PP^3}(4)|$)
$\pi^{-1}(\gM\setminus\gM^{IV})\subset ( |\cO_{\PP^3}(4)|\setminus  |\cO_{\PP^3}(4)|^{IV})$. Hence $\gp$ is regular away from $\gM^{IV}$ because of~\Ref{lmm}{est3}. Lastly,  $\gM^I,\gM^{II},\gM^{III},\gM^{IV}$ define a stratification of $\gM$ because $\cF,\cF^{II},\cF^{III}$ define a stratification of $\cF^{*}$. 
\end{proof}
%

%%%%%%%%%%%%%%%%%%%%%%%%%%%%%%%%%%%%%%%
\section{Shah's explicit description of the Hodge theoretic stratification of $\gM$}\label{sec:refshah}
\setcounter{equation}{0}
In the present section, we briefly review Shah's explicit description  (\cite[Theorem 2.4]{shah4}) of the strata in the  Hodge theoretic  stratification of  
$\gM$ defined in the previous section (essentially the strata are the intersections between   Hodge theoretic  and  GIT strata). Then, we will slightly refine Shah's stratification of $\gM^{IV}$,  so that the refined strata  match (in \lq\lq reverse order\rq\rq) the strata $Z^m$  in~\eqref{zetaquart}.  In many instances the refined strata are   connected components of one of Shah's Hodge theoretic strata. 
\subsection{Type II strata for $\gM$} 
The period map extends regularly away from $\gM^{IV}$, and maps $\gM^{II}$ to $\cF^{II}$.  The matching of the irreducible components of $\gM^{II}$ and the Type II boundary components  will be given in the following section (together with an explanation of the discrepancies in dimensions). For the moment being, we  note that Shah identified $8$  irreducible components of $\gM^{II}$,  and that each polystable quartic $X$ parametrized by a point of $\gM^{II}$ 
has a \lq\lq $j$-invariant\rq\rq. More precisely, either $X$ has  a simple elliptic singularity (of type $\widetilde E_6$, $\widetilde E_7$, or $\widetilde E_8$), or $\sing X$ contains  an elliptic curve, or a rational curve with $4$ pinch points. Hodge theoretically, this corresponds to the fact that $\Gr_1^WH^2(X_0)\neq 0$ (N.B.: simple Hodge theoretic considerations  show that if there is more than one source of $j$-invariant, e.g.~two simple elliptic singularities, then the $j$-invariants  coincide). 

 \begin{proposition}\label{prp:gittype2}
 The Type II GIT boundary $\gM^{II}$ consists of $8$ irreducible boundary components. We label these components by II(1)--II(8). Let $X$ be a quartic surface with closed orbit corresponding to the generic point of a Type II components. Then, $X$ has the following description:
 \begin{itemize}
 \item[II(1)] (cf. S-II(B,i, $\widetilde E_8$), also the generic locus in $\sigma_1$) -- $\Sing(X)$ consists of two double points of type $\widetilde E_8$.
  \item[II(2)] (cf. S-II(B,i, $\widetilde E_7$), also the generic locus in $\sigma_2$) -- $\Sing(X)$ consists of two double points of type $\widetilde E_7$ and some rational double points.
 \item[II(3)] (cf. S-II(B,ii), also the generic locus in $\sigma_3$) -- $\Sing(X)$ consists of two skew lines, each of which is an ordinary nodal curve with four simple pinch points.

   \item[II(4)] (cf. S-II(B,iii), also the generic locus in $\sigma_4$) -- $X$ consists of a plane and a cone over a nonsingular cubic curve in the plane (triple point of type $\widetilde E_6$).
     \item[II(5)] (cf. S-II(A,i)) -- $\Sing(X)$ consists of a double point $p$ of type $\widetilde E_8$ and some rational double points such that no line in $X$ passes through $p$.
       \item[II(6)] (cf. S-II(A,ii, $\deg$ $2$)) -- $\Sing(X)$  consists of a smooth conic $C$ and possibly some rational double points. $C$ is an ordinary nodal curve with $4$ pinch points.
 \item[II(7)] (cf. S-II(A,ii, $\deg$ $3$)) --  $\Sing(X)$  consists of a twisted cubic $C$ and possibly some rational double points. $C$ is an ordinary nodal curve with $4$ pinch points.

   \item[II(8)] (cf. S-II(A,ii, $\deg$ $4$)) -- $\Sing(X)$ consists of an elliptic normal curve of degree $4$ and possibly some rational double points (equivalently $X$ is the union of two quadric surfaces that meet transversally).  
  \end{itemize}
 Furthermore, the cases II(5)--II(8) correspond to stable quartics, while the cases II(1)--II(4) to strictly semistable quartics with generic stabilizer the $1$-PSs $\lambda_1,\dots,\lambda_4$ respectively (N.B. $\overline{II(i)}=\sigma_i$ cf.~\Ref{lmm}{lem1ps}). 
 \end{proposition}
 \begin{proof}
 This is precisely Shah \cite[Thm. 2.4]{shah4}. The corresponding cases in Shah's Theorem are labeled by S-II(A/B, Case). Some of Shah's cases (e.g. Theorem 2.4 II.A.ii) have  several geometric sub-cases that are labeled in an obvious way (e.g. S-II(A,ii, $\deg$ $3$) corresponding to the case when $\Sing(X)$ is a twisted cubic). 
 \end{proof}
 
 \begin{remark}[{Quartics with $\widetilde E_8$ singularities, cf.~Urabe \cite{urabee12}}]\label{rmk:e8tildecase}
 Let us note that there are two deformation classes of quartic surfaces with $\widetilde E_8$ singularities. The generic quartic $S$  in each of these two strata  has a unique singular point $p$, of type $\widetilde E_8$. The minimal resolution $\widetilde S\to S$ has the following properties:
 \begin{itemize}
 \item[i)] $\widetilde S$ is a rational surface (a consequence of iii) below);
 \item[ii)] the exceptional divisor $D$ of $\widetilde S\to S$ is a smooth elliptic curve of self-intersection $-1$ (this is the condition of having $\widetilde E_8$ singularities);
 \item[iii)] $(\widetilde S, D)$ is an anticanonical pair (i.e.~$D\in |-K_{\widetilde S}|$) (this is a consequence of $S$ being a degeneration of $K3$ surfaces);
 \item[iv)] $\widetilde S$ comes equipped with a nef and big class $h$ s.t.~$h^2=4$ and $h.D=0$ (i.e.~$S$ is a quartic). 
 \item[v)] Furthermore, we can assume that the linear system associated to $h$ contracts only $D$.
 \end{itemize}
 It is not hard to see (e.g.~\cite[Prop.~1.5]{urabee12}) that $\widetilde S$ is the blow-up of $\bP^2$ at $10$ points on a smooth cubic curve $C$ in $\bP^2$ (and  $D\subset \widetilde S$ is the strict transform of $C$). Thus, $\Pic \widetilde S=\langle \ell, e_1,\dots, e_{10}\rangle$, where $\ell$ is the pull-back of $\calO_{\bP^2}(1)$ and $e_1,\dots, e_{10}$ are the $10$ exceptional divisors. The classification of the possible divisor classes $h$ as above was done by Urabe \cite[Prop.~4.3]{urabee12}. Up to natural symmetries, there are two distinct possibilities:
 \begin{eqnarray*}
 (a)\ \ h&=&9l-3(e_1+\dots+e_8)-2e_9-e_{10}\\
(b)\ \  h&=&7l-3e_1-2(e_2+\dots+e_{10}).
 \end{eqnarray*}
 In other words, if $\widetilde S$ is the blow-up  of $\bP^2$ at $10$ (general) points on a cubic curve with a divisor class $h$ as above, then  $S=\phi_{|h|}(\widetilde S)$ is a quartic in $\bP^3$ with one $\widetilde E_{8}$ singularity $p$. The two cases are distinguished geometrically by the fact that  case (a), $S$ contains a line passing through $p$ (with class $e_{10}$), while in case (b) there is no such line. By construction, it is easy to see that $S$ depends on $10$ moduli in each of the cases (a) and (b) -- in particular, neither of the case is a specialization of the other one. 
 Finally,  Shah's analysis (\cite[Theorem 2.4]{shah}) shows that the generic surface of type (a) is strictly semistable with associated minimal orbit of Type II(1) (cf. the proposition above). In  case (b), the surface S is stable (Type II(5) above). 
 \end{remark}

 \begin{remark}[{Arithmetic of Quartics with $\widetilde E_8$ singularities}]\label{rmk:arithmetice8} Let us note that the two cases of the \Ref{rmk}{e8tildecase} are distinguished also from an arithmetic perspective. The arguments here are standard and are contained (with full details) in Urabe \cite{urabee12}. First note that since $\widetilde S$ is the blow-up of $\bP^2$ at $10$ points, $\left(H^2(\widetilde S), \langle,\rangle\right)$ is isometric as lattice to $I_{1,10}$. Since $K^2=-1$, it follows that the lattice $K_{H^2(\widetilde S)}^\perp$ (notation $\Gamma$ in  \cite{urabee12}) is an even unimodular lattice of signature $(1,9)$ (and thus isometric to $E_8\oplus U$). The polarization class $h$ has norm $4$ and belongs to $K^\perp\cong E_8\oplus U$. It is not hard to see that there are exactly (up to isometries) two choices for $h$ that are distinguished by the isometry class of the negative definite lattice $h^{\perp}_{K^\perp}$ (notation $\Lambda$ in  \cite{urabee12}) . Namely, $h^{\perp}_{K^\perp}$ is either $E_8\oplus D_1$ (recall $D_1=\langle -4\rangle$) or $D_9$. The case (a) corresponds to $E_8\oplus D_1$, while the case (b) corresponds to $D_9$ (e.g. see \cite[p.~1231]{urabee12}). 
 \end{remark}
\subsection{Type III strata for $\gM$}
 For completeness, we list Shah's strata contained in $\gM^{III}$. By~\cite{scattone}, there is  unique Type III boundary point in $\cF^{*}$, hence the period map sends  all these strata to the same point of $\cF^{*}$. 

\begin{proposition}
A polystable quartic $X$ corresponds to a point of $\gM^{III}$ if and only if one of the following holds:
\begin{itemize}
\item[III(1)] (cf. S-III(B,iii), also case $\zeta$) -- $X$ consists of four planes with normal crossings (the tetrahedron). This is a single point $\zeta\in \gM$ (cf. \ref{prp:classifystab} (i)). 
\item[III(2)] (cf. S-III(B,ii, 4 lines), also generic locus in $\tau$) -- $X$ consists of two, nonsingular, quadric surfaces which intersect in a reduced curve $C$ which consists of four lines, and whose singular locus consist of $4$ double points. This gives a curve $\tau^\circ\subset  \gM$ (cf. \ref{prp:classifystab} (ii)), where $\tau^\circ=\tau\setminus\{\omega,\zeta\}$.  
\item[III(3)]  (cf. S-III(B,ii, 2 conics)) -- $X$ consists of two, nonsingular, quadric surfaces which intersect in a reduced curve, $C$, of arithmetic genus $1$. C consists of two  conics such and its singular locus consist of $2$ double points; the dual graph of $C$ is homeomorphic to a circle. This case is a specialization of the case II(8) above. Stabilizer $\lambda_4=(1,0,0,-1)$. 
\item[III(4)] (cf. S-III(B,i, $\deg$ $3$)) --  $\Sing(X)$ consists of a nonsingular, rational curve of degree $3$, and some rational double points. $C$ is a strictly quasi-ordinary, nodal curve and its set of pinch points consists of two double pinch points. Each double pinch point lies on a line in $X$. Stabilizer $\lambda_3=(3,1,-1,-3)$. Also a specialization of the case II(7). 
\item[III(5)]  (cf. S-III(B,i, $\deg$ $2$)) -- $\Sing(X)$ consists of a nonsingular, rational curve of degree $2$, and some rational double points. $C$ is a strictly quasi-ordinary, nodal curve and the set of its pinch points consists of two double pinch points. Each double pinch point lies on a line in $X$. Stabilized by $\lambda_4=(1,0,0,-1)$. Specialization of the case II(6). 
\item[III(6)]  (cf. S-III(A,ii)) -- $\Sing(X)$ consists of a strictly quasi-ordinary nodal curve, $C$, and some rational double points such that no line in $X$ passes through a double pinch point. $C$ is a nonsingular, rational curve of degree $2$. $X$ has either two double pinch points on $C$ or one double
pinch point and two simple pinch points on $C$. Specialization of the case II(6).

\item[III(7)]  (cf. S-III(A,i)) -- $\Sing(X)$ consists of a double point, $p$, of type $T_{2,3,r}$ and some rational double points such that no line in $X$ passes through $p$. Specialization of the case II(5). 
\end{itemize}
If III(1)--III(5) holds, then $X$ is strictly semistable, if  III(6) or III(7) holds, then $X$ is stable. 
\end{proposition}

\subsection{Type IV strata for $\gM$} 
The period map is regular away from $\gM^{IV}$, hence in order to  decompose $\gp\colon \gM\dashrightarrow \cF^*$
into a composition of simple birational maps, we must study $\gM^{IV}$. The following is  a slight refinement of Shah \cite[Theorem 2.4]{shah4}:

\begin{proposition}\label{prp:proptype4}
The Type IV locus $\gM^{IV}$ decomposes in the following strata:
\begin{itemize}
 \item[IV(0a)] (cf. S-IV(B,iii)) -- $X$ consists of a non-singular quadric surface with multiplicity $2$.  (Case \ref{prp:classifystab}(iii)). The point $\omega\in  \gM$ corresponding to (generic) hyperelliptic quartics. 
   \item[IV(0b)] (cf. S-IV(B,i, $\deg$ 3)) --  $\Sing(X)$ consists of a nonsingular, rational curve, $C$, of degree $3$; $C$ is a simple cuspidal curve. The normalization of $X$ is nonsingular. This is the tangent developable to the twisted cubic (Case \ref{prp:classifystab}(iv)). The corresponding point $\upsilon\in  \gM$  corresponds to unigonal $K3$s.

\item[IV(1)] (cf. S-IV(B,ii)) -- $X$ consists of two quadric surfaces, $V_1$, $V_2$ tangent along a nonsingular conic $C$ such that $V_1\cap V_2=2C$. (Case \ref{prp:classifystab}(iii)). It corresponds to a curve inside $ \gM$. 
 \item[IV(2)]  (cf. S-IV(B,i, $\deg$ 2)) -- $\Sing(X)$ consists of a nonsingular, rational
curve, $C$, of degree $2$; $C$ is a simple cuspidal curve. The normalization of $X$ has exactly two rational double points. Stabilized by $\lambda_4=(1,0,0,-1)$
 \item[IV(3)] $\Sing(X)$  consists of a nodal curve, $C$, and rational double points such that no line in $X$ passes through a non-simple pinch point. $C$ is a nonsingular, rational curve of degree $2$. Every point of $X$ on $C$ is a double point and the set of pinch points consists of  a point of type $E_{4,\infty}$.
  \item[IV(4)] $\Sing(X)$  consists of a nodal curve, $C$, and rational double points such that no line in $X$ passes through a non-simple pinch point. $C$ is a nonsingular, rational curve of degree $2$. Every point of $X$ on $C$ is a double point and the set of pinch points consists of either a point of type $E_{3,\infty}$ and a simple pinch point or a point of type $E_{4,\infty}$.
 \item[IV(5)] $\Sing(X)$ consists of a double point, $p$, of
type  $E_{3,r}$ and some RDPs such that no line in $X$ passes through $p$. This case is a specialization of Case III(7) (and then II(8)).   
 \item[IV(6)] (cf. S-IV(A,i, $E_{14}$)) -- $\Sing(X)$ consists of a double point of
type $E_{14}$.
 \item[IV(7)] (cf. S-IV(A,i, $E_{13}$)) -- $\Sing(X)$ consists of a double point of
type $E_{13}$.
 \item[IV(8)]  (cf. S-IV(A,i, $E_{12}$)) --  $\Sing(X)$ consists of a double point of
type $E_{12}$.
\end{itemize}
 \end{proposition}

\begin{remark}\label{rmk:reminclusion}
There are natural inclusions $\mathrm{IV}(k)\subset \overline{\mathrm{IV}(k+1)}$ with the exception $k=4$ (N.B.  $\mathrm{IV}(4)\subset \overline{\mathrm{IV}(6)}$). For instance, we have the following adjacencies  for  the exceptional unimodal singularities (aka Dolgachev singularities): $E_{14}\longrightarrow E_{13}\longrightarrow E_{12}$ (see \cite[p. 159]{arnold}). 
\end{remark}

\begin{definition}\label{dfn:doubleuq}
We define
$$W_k=\overline{\mathrm{IV}(k)},$$
with the following two exceptions: $W_0=\overline{\mathrm{IV}(0a)}$, and we skip the case $k=5$.
\end{definition}

\begin{remark}
For quartics singular along a twisted cubic,  we have the inclusions
$$ IV(0b) \subset III(4) \subset II(7).$$
\end{remark}

\begin{remark}
Clearly, II(1), III(1), and IV(1) form a single stratum. The degeneracy condition is that there is a line passing through $p$, cf. \cite[Cor. 2.3 (i)]{shah4}: {\it an isolated, non rational, double point of Type 1 through which passes a line contained in $X$}.
\end{remark}

\begin{remark}
Cases II(5) and its specializations III(7) and IV(6--8) were studied by Urabe \cite{urabee12}. 
\end{remark}

Our predictions regarding the matching of strata in  $\gM$ and strata in $\cF$ is summarized in the Table~\ref{tablemainresult} below. 

\begin{table}[htb!]
\renewcommand{\arraystretch}{1.60}
\begin{tabular}{|c|c|l|l|}
\hline
Codim ($i$) & Critical $\beta$ & (Compn't of) corresponding $Z^{i}\subset\cF$ & (Compn't of) corresponding  $W_{i-1}\subset\gM$\\
\hline
1& 1 &$H_h$& IV(0a): double quadric\\
1& 1 &$H_u$& IV(0b): tangent developable\\

\hline
2& $\frac{1}{2}$ & $\Delta^{(2)}$& IV(1): 2 quadrics tangent along a conic\\
3& $\frac{1}{3}$ & $\Delta^{(3)}$& IV(2): double conic, cuspidal type\\
4&$\frac{1}{4}$ & $\Delta^{(4)}$&IV(3): $E_{4,\infty}$-locus\\
\hline
5&$\frac{1}{5}$ &$\Delta^{(5)}$ & IV(4):  $E_{3,0}$ \\
6&$\frac{1}{5}$ &  $\Delta^{(6)}$ & IV(5): $E_{3,\infty}$ and $E_{3,r}$\\
\hline
7&$\frac{1}{6}$ &unigonal in $\Delta^{(6)}$ ($T_{3,3,4}$-polarized K3)& IV(6): $E_{14}$-locus\\
8&$\frac{1}{7}$ &unigonal in $\Delta^{(7)}$ ($T_{2,4,5}$-polarized K3)& IV(6): $E_{13}$-locus\\
9&$\frac{1}{9}$ &unigonal in $\Delta^{(8)}$ ($T_{2,3,7}$-polarized K3)& IV(8): $E_{12}$-locus \\
\hline
\end{tabular}
\vspace{0.2cm}
\caption{The geometry of the variation of models $\cF(\beta)$}\label{tablemainresult}
\end{table}

\begin{remark}
The points IV(0a) and IV(0b)correspond to $H_h$ and $H_u$  respectively; this  is  discussed in  Sections 4 and 3 of~\cite{shah4}. We revisit the proof 
in~\Ref{subsec}{sblowomega} and~\Ref{subsec}{sblowupsilon} respectively. The  matching for $\beta=\frac{1}{2}$ is discussed in \Ref{subsec}{rescodim2}. Finally, in \Ref{sec}{dolgachev} we give some evidence for the matching corresponding to the case $\beta\in \left\{ \frac{1}{6}, \frac{1}{7},\frac{1}{9}\right\}$. We don't say much about the remaining cases. 
\end{remark}

 \begin{remark}
 We recall that the  locus $Z^9\subset\cF$ (described as the unigonal divisor inside $\Delta^{(8)}\cong \cF(11)$) is one of the two components of 
 $\Delta^{(9)}$. With this description, the jump from $\frac{1}{7}$ to $\frac{1}{9}$ is less surprising: the critical $\beta=\frac{1}{9}$ comes from having $9$ independent sheets of $\Delta$ meeting along the $Z^{9}$ locus. 
\end{remark}

\begin{remark}
While the entire framework of the paper is similar to the Hassett-Keel program for curves, the geometric analogy with Hassett-Keel  is particularly striking in the case of flips occurring for $\beta\in\left\{\frac{1}{6},\frac{1}{7},\frac{1}{9}\right\}$. Namely, to pass from the $E_l$ ($l=12,13,14$) locus on the GIT side to the periods side, one needs to perform a KSBA semistable replacement. This is completely analogous to the stable reduction for cuspidal curves, which leads to the elliptic tail replacement (or globally to the first birational modification: $ \gM_g\to  \gM_g^{ps}\cong \gM_g(\frac{9}{11})$). This part is closely related to the work of Hassett \cite{hassett} (stable replacement for curves). This is expanded on in \Ref{sec}{dolgachev}. 
\end{remark}

%%%%%
\section{The critical values $\beta=1$ and   $\beta=1/2$}\label{sec:hklquartics}
\setcounter{equation}{0}
The point of view of this paper is somewhat dual to that of \cite{log1}. Namely, while in \cite{log1} we have given a (conjectural) decomposition of the inverse of the period map $\gp^{-1}:\cF^*\dashrightarrow \gM$ based on arithmetic considerations, here we start from the other end and attempt to resolve the period map $\gp: \gM\dashrightarrow \cF^*$. As  is familiar  to those who have studied  the analogous period maps with domains the GIT moduli spaces of plane sextics ~\cite{shah} and cubic fourfolds~\cite{lcubic,gitcubic,cubic4fold}, the first step towards resolving the period map $\gp$ is to blow-up the most singular points, i.e.~those parametrizing polystable quartics with the largest (non virtually abelian) stabilizers (see~\Ref{prp}{classifystab}). There are two such points, namely $\upsilon$ corresponding to the tangent developable of a twisted cubic curve and  $\omega$  corresponding to a smooth quadric with multiplicity $2$. In~\Ref{subsec}{sblowupsilon} 
and~\Ref{subsec}{sblowomega} we discuss a suitable blow-up  $\wt{\gM}\lra \gM$ with center a subscheme whose support is $\{\upsilon, \omega\}$. 
\Ref{thm}{firstblowup} and~\Ref{thm}{secondblowup} give the main results regarding $\wt{\gp}$, the pull-back of the period map to $\wt{\gM}$. In short,  the component of the exceptional divisor mapping to $\upsilon$ is identified with $\gM_u$, a projective GIT compactification of the moduli space of unigonal $K3$ surface 
 (see~\eqref{gitunig}), and the  component of the exceptional divisor mapping to $\omega$ is identified with $\gM_h$, the GIT moduli space of $(4,4)$ curves on $\PP^1\times\PP^1$. Moreover, the lifted period map $\wt{\gp}$ is regular in a neighborhood of the exceptional divisor  $\gM_u$, but it is definitely not regular at all points of  the exceptional divisor  $\gM_h$, in fact the restriction to $\gM_h$ is almost as complex as $\gp$ is, there is an analogous tower of closed subsets of the relevant period space, only it has $7$ terms instead of $8$. It is worth remarking that the image of the restriction of $\wt{\gp}$ to  the regular locus is the complement of $\Delta^{(2)}$, while the image  of the restriction of $\gp$ to  the regular locus is the complement of $H_h\cup H_{u}$. In this sense, in going from $\gp$ to $\wt{\gp}$ we have improved the behaviour of the period map, and moreover $\widetilde \gp$ is an isomorphism in codimension $1$, while $\gp$ is not.  Lastly, we have an identification $\widetilde \gM\cong \cF(1-\epsilon)$ (see~\Ref{crl}{cor1minuseps}). 

We continue in \Ref{subsec}{rescodim2} with the analysis of the ``first flip'' that occurs when one tries to resolve the birational map $\widetilde \gp:\widetilde \gM\dashrightarrow \cF^*$  Briefly, we show that a blow-up of the curve $W_1$ (case IV(1) in \Ref{prp}{proptype4}) followed by a contraction, accounts for double covers of the quadric cone (stratum $Z^2\subset \cF^*$ in our notation). In other words, we essentially  verify\footnote{Some technical issues regarding the global construction of the flip still remain, but our analysis is fairly complete.} the predicted behavior of the variation of models $\cF(\beta)$ for $\beta\in(1/2-\epsilon,1]\cap\QQ$.

\subsection{Blow up of the point $\upsilon$}\label{subsec:sblowupsilon}
The point $\upsilon$ (see IV(0b) in \Ref{prp}{proptype4}) is an isolated point of the indeterminacy locus of the period map $\gp$.  The behavior of $\gp$ in a neighborhood of  $\upsilon$ is  analogous to that of the period map of the moduli space of plane sextics in a neighborhood of the orbit of $3C$  (see \cite{shah}, \cite{looijengavancouver}, \cite[Thm.~1.9]{k3pairs}), where $C\subset\PP^2$ is a smooth conic, and is treated  in Section 3 of Shah~\cite{shah4}.  Shah's results imply that by blowing up a subscheme of $\gM$ supported at $\upsilon$, one resolves the indeterminacy of $\gp$ in $\upsilon$; the main result is stated
 in~\Ref{subsubsec}{unknownideal}. 
\subsubsection{The germ of $\gM$ at  $\upsilon$ in the analytic topology}
We will apply Luna's \'etale slice Theorem in order to describe an analytic neighborhood of $\upsilon$ in the GIT quotient $\gM$. 
Let $T\subset\PP^3$ be the twisted cubic $\{[\lambda^3,\lambda^2\mu,\lambda\mu^2,\mu^3] \mid [\lambda,\mu]\in\PP^1\}$, and let $X$ be the tangent developable of $T$, i.e.~the union of lines tangent to $T$. A generator of the homogeneous ideal of $X$ is given by
\begin{equation}\label{sviluppabile}
f:=4(x_1 x_3-x_2^2)(x_0 x_2-x_1^2)-(x_1 x_2-x_0 x_3)^2.
\end{equation}
Thus $X$ is a polystable quartic representing the point $\upsilon$. The group $\PGL(2)$ acts on $T$ and hence on $X$; it is clear that  $\PGL(2)=\Aut(X)$. 
In order to describe an \'etale slice for the orbit $\PGL(4) X$ at $X$ we must  decompose  $H^0(\PP^3,\cO_{\PP^3}(4))$ into irreducible $\SL_2$-submodules.  For $d\in\NN$, let $V(d)$ be the irreducible $\SL_2$-representation with highest weight $d$ i.e.~$\Sym^d V(1)$ where $V(1)$ is the standard $2$-dimensional $\SL_2$-representation. A straightforward  computation gives the decomposition
\begin{equation}\label{decompquar}
H^0( \cO_{\PP^3}(4))  \cong  V(0)\oplus V(4)\oplus V(6)\oplus V(8)\oplus V({12}).
\end{equation}
The trivial summand $V(0)$ is spanned by $f$, and  the projective tangent space at $V(f)$ to the orbit $\PGL(4) V(f)\subset |\cO_{\PP^3}(4)| $ is equal to  $\PP(V(0)\oplus V(4)\oplus V(6))$. 
We have a natural  map
\begin{equation}\label{fettaluna}
\begin{matrix}
V(8)\oplus V({12})\gquot\SL(2) & \lra & \gM, \\
[g] & \mapsto & [V(f+g)]
\end{matrix}
\end{equation}
mapping $[0]$ to $\upsilon$. By Luna's \'etale slice Theorem, the map is \'etale at $[0]$. In particular we have an isomorphism of analytic germs
\begin{equation}\label{upgerme}
(V(8)\oplus V({12})\gquot\SL(2),[0]) \overset{\sim}{\lra} (\gM,\upsilon).
\end{equation}
\subsubsection{Moduli and periods of unigonal $K3$ surfaces}\label{subsubsec:modunig}
Let 
\begin{equation}\label{eccomega}
\Omega:=\Symm^{\bullet}(V(8)^{\vee}\oplus V(12)^{\vee}), 
\end{equation}
and define a grading of $\Omega$ as follows: non zero elemnts of $V(8)^{\vee}$ have degree $2$, non zero elements of $V(12)^{\vee}$ have degree $3$. Then $\SL(2)$ acts on $\Proj\,\Omega$, and $\cO_{\Proj\Omega}(1)$ is naturally linearized; let 
\begin{equation}\label{gitunig}
\gM_u:=\Proj\,\Omega\gquot\SL(2).
\end{equation}
Shah (see Theorem 4.3 in~\cite{shah}) proved that $\gM_u$ is a compactification of the moduli space for unigonal $K3$ surfaces, i.e.~there is an open dense subset  $\gM_u^{I}\subset\gM_u$ which is the moduli space for such $K3$'s. Moreover, the period map is regular
\begin{equation}\label{perunig}
\gM_u \overset{\gp_u}{\lra} \cF_{\II_{2,18}}(O^{+}(\II_{2,18}))^{*},
\end{equation}
and  it  defines  an isomorphism $\gM_u^{I}\overset{\sim}{\lra} \cF_{\II_{2,18}}(O^{+}(\II_{2,18}))$. We recall that we have a natural regular map 
\begin{equation}\label{unigbb}
 \cF_{\II_{2,18}}(O^{+}(\II_{2,18}))^{*}\lra \cF^{*},
\end{equation}
whose restriction to $ \cF_{\II_{2,18}}(O^{+}(\II_{2,18}))$ is an isomorphism onto the 
 unigonal divisor $H_u$, see Subsection~1.5 of~\cite{log1}.
\subsubsection{Weighted blow-up}\label{subsubsec:monomideal}
We recall the construction of the weighted blow up in the case where the base is smooth. We refer to~\cite{kollmori3flips,andblow} for details. Let $(x_1,\ldots,x_n)$ be the standard coordinates on $\aff^n$. Let $(a_1,\ldots,a_n)\in \NN_{+}^n$, and let $\sigma$ be the weight given by $\sigma(x_i)=a_i$. The weighted blow-up $B_{\sigma}(\aff^n)$ with weight $\sigma$ is  a toric variety defined as follows. Let $\{e_1,\ldots,e_n\}$ be the standard basis of $\RR^n$, and $C\subset\RR^n$ be the convex cone spanned by $e_1,\ldots,e_n$, i.e.~the cone of $(x_1,\ldots,x_n)$ with non-negative entries. Let $v:=(a_1,\ldots,a_n)\in\RR^n$, and for $i\in\{1,\ldots,n\}$ let $C_i\subset C$ be the convex cone spanned by $e_1,\ldots,e_{i-1},e_{i+1},\ldots,e_n$ and $v$.  The $C_i$'s generate a fan in $\RR^n$; $B_{\sigma}(\aff^n)$  is the associated toric variety. Since the $C_i$'s define a cone decomposition of $C$, we have a natural regular map $\pi_{\sigma}\colon B_{\sigma}(\aff^n)\to\aff^n$, which is an isomorphism over $\aff^n\setminus\{0\}$.  Let $E_{\sigma}\subset B_{\sigma}(\aff^n)$ be the exceptional set of $\pi_{\sigma}$; then $E_{\sigma}$ is isomorphic to the weighted projective space $\PP(a_1,\ldots,a_n)$. We denote by $[x_1,\ldots,x_n]$ (with $(x_1,\ldots,x_n)\not=(0,\ldots,0)$) a (closed) point of $\PP(a_1,\ldots,a_n)$; thus  $[x_1,\ldots,x_n]=[y_1,\ldots,y_n]$  if and only if there exists $t\in\CC^{*}$ such that $x_i=t^{a_i}y_i$ for  $i\in\{1,\ldots,n\}$. The composition
\begin{equation*}
\begin{matrix}
B_{\sigma}(\aff^n) & \overset{\pi_{\sigma}}{\lra} &  \aff^n & \dra & \PP(a_1,\ldots,a_n) \\
p & \mapsto & \pi_{\sigma}(p)=(x_1,\ldots,x_n) & \mapsto & [x_1,\ldots,x_n]
\end{matrix}
\end{equation*}
is regular; this follows from the formulae for $\pi_{\sigma}$ that follow Definition~2.1 in~\cite{andblow}. Thus we have a regular map 
\begin{equation}\label{notiso}
B_{\sigma}(\aff^n) \lra \aff^n\times \PP(a_1,\ldots,a_n).
\end{equation}
Let  $\mu_{\sigma}\colon E_{\sigma}\to  \PP(a_1,\ldots,a_n)$ be the restriction to $E_{\sigma}$ of the map in~\eqref{notiso}, followed by projection to the second factor. Then  
$\mu_{\sigma}$ is an isomorphism; we will identify $E_{\sigma}$ with $\PP(a_1,\ldots,a_n)$ via $\mu_{\sigma}$. 
The formulae for $\pi_{\sigma}$ that follow Definition~2.1 in~\cite{andblow} give the following result.
\begin{proposition}\label{prp:regonblow}
Keep notation as above, and let $\Delta\subset\CC$ be a disc centered at $0$. Let $\alpha\colon\Delta\to B_{\sigma}(\aff^n)$ be a holomorphic map such that
 $\alpha^{-1}(E_{\sigma})=\{0\}$. There exists $k>0$ such that 
\begin{equation}
\pi_{\sigma}\circ\alpha(t)=(t^{k a_1}\cdot\varphi_1,\ldots,t^{k a_n}\cdot\varphi_n)
\end{equation}
where   $\varphi_i\colon\Delta\to \CC$ is a holomorphic function, and moreover
\begin{equation}\label{rolrol}
\alpha(0)=[\varphi_1(0),\ldots,\varphi_n(0)].
\end{equation}
(In particular $(\varphi_1(0),\ldots,\varphi_n(0))\not=(0,\ldots,0)$.)
\end{proposition}
\begin{corollary}\label{crl:regonblow}
Let $Z$ be a projective variety, and $\gp\colon B_{\sigma}(\aff^n)\dra Z$ be a rational map, regular away from $E_{\sigma}$. Suppose that the following holds. 
Given  a disc  $\Delta\subset\CC$ centered at $0$, and  a holomorphic map $\alpha\colon\Delta\to B_{\sigma}(\aff^n)$  such that
  $\alpha^{-1}(E_{\sigma})=\{0\}$, the extension at $0$ of the map $\gp\circ\alpha|_{(\Delta\setminus\{0\})}$ \emph{depends only} on $\alpha(0)=[\varphi_1(0),\ldots,\varphi_n(0)]$ (notation as in~\eqref{rolrol}). Then $\gp$ is regular everywhere.
\end{corollary}
\begin{proof}
Follows from~\Ref{prp}{regonblow} and normality of $B_{\sigma}(\aff^n)$.
\end{proof}
\subsubsection{Blow-up of the \'etale slice and the period map}\label{subsubsec:blowups}
It will be convenient to denote by $Z$ the affine scheme $V(8)\oplus V(12)$, i.e.~$Z:=\Spec\Symm^{\bullet}(V(8)^{\vee}\oplus V(12)^{\vee})$. 
Let  $(x_1,\ldots,x_{22})$ be  coordinates on  $V(8)\oplus V(12)$ such that $V(8)$ has equations $0=x_{10}=\ldots=x_{22}$, and 
 $V(12)$  has equations $0=x_{1}=\ldots=x_{9}$.
 Let $\sigma$ be the weight defined by 
\begin{equation}
\sigma(x_i):=
\begin{cases}
4 & \text{if $i\in\{1,9\}$,} \\
6 & \text{if $i\in\{10,22\}$.}
\end{cases}
\end{equation}
Let $\wt{Z}:=\Bl_{\sigma}(Z)$ be the corresponding weighted blow up, and 
 let $E$ be the exceptional set of $\wt{Z}\to Z$; thus $E$ is the weighted projective space $\PP(4^{9},6^{13})$. 
The action of $\SL_2$ on $Z$ lifts to an action on $\wt{Z}$ (and on the ample line-bundle $\cO_{\wt{Z}}(-E)$). Thus there is an associated GIT quotient  $\wt{Z}\gquot \SL_2$. 
The map   $\wt{Z}\to Z$ induces a map
\begin{equation}\label{duemu}
 \wt{\mu}\colon\wt{Z}\gquot \SL_2 \lra Z\gquot\SL_2.
\end{equation}
  Moreover 
the set-theoretic inverse image  $\wt{\mu}^{-1}([0])_{red}$ is isomorphic to  $\Proj\,\Omega\gquot\SL_2=\gM_u$. Since the natural map $Z\gquot\SL_2\to\gM$ is dominant, it makes  sense to compose it with the (rational) period map $\gp\colon\gM\dra\cF(19)^{*}$. Composing with $\wt{\mu}$, we get a rational map
\begin{equation}\label{dueper}
 \wt{\gp}\colon\wt{Z}\gquot \SL_2 \dra \cF(19)^{*}.
\end{equation}
\begin{theorem}\label{thm:firstblowup}
With notation as above, the   map $\wt{\gp}$ is regular in a neighborhood of $\wt{\mu}^{-1}([0])_{red}=\gM_u$, and its restriction to  $\wt{\mu}^{-1}([0])_{red}$ is equal to the period map $\gp_u$ in~\eqref{unigbb}. 
\end{theorem}
\begin{proof}
This follows from the results of   Shah in~\cite{shah4}. More precisely, let $(F,G)\in V(8)\oplus V(12)$ be non-zero and such that $[(F,G)]\in\Proj\,\Omega$ is $\SL_2$-semistable. Let $\Delta\subset\CC$ be a disc centered at $0$, and
\begin{equation}
\begin{matrix}
\Delta & \overset{\varphi}{\lra} & V(8)\oplus V(12) \\
t & \mapsto & (t^{4m}F(t), t^{6m}G(t))
\end{matrix}
\end{equation}
where $m>0$, $F(t)$, $G(t)$ are holomorphic, and $F(0)=F$, $G(0)=G$. (This is the family on the second-to-last displayed equation of p.~293, with the difference that our $(0,0)\in Z$ corresponds to Shah's $F_0$.) We assume also that for $t\not=0$, the point $[\varphi(t)]$ is \emph{not} in the indeterminacy locus of the period map $Z\gquot\SL_2 \dra\cF(19)^{*}$. Let $\gp_{\varphi}\colon\Delta\to\cF(19)^{*}$ be the holomorphic extension of the composition $(\Delta\setminus\{0\}) \to Z\gquot\SL_2\dra\cF(19)^{*}$. Then by Theorem~3.17 of~\cite{shah4}, the value $\gp_{\varphi}(0)$ is equal to the period point $\gp_u([(F,G)])$. By~\Ref{crl}{regonblow} it follows that
 ${\gp}$ is regular in a neighborhood of $\wt{\mu}^{-1}([0])_{red}=\gM_u$,  and that  the restriction of the period map to  $\wt{\mu}^{-1}([0])_{red}$ is equal to the period map $\gp_u$ in~\eqref{unigbb}. 
\end{proof}
\subsubsection{Blow-up of $\gM$ at $\upsilon$}\label{subsubsec:unknownideal}
A weighted blow up $Bl_{\sigma}(\aff^n)\to\aff^n$ is equal to the  blow up of a suitable scheme supported at $0$, see Remark~2.5 of~\cite{andblow}.  It follows that also the map in~\eqref{duemu} is the blow up of an ideal $\cJ$ supported on $[0]$. 
Since the map in~\eqref{upgerme} is an isomorphism of analytic germs, the ideal sheaf $\cJ$  defines an ideal sheaf  in $\cO_{\gM}$,  cosupported at $\upsilon$, that we will denote by $\cI$. Let $\gM_{\upsilon}:=\Bl_{\cI}\gM$, and let $E_{\upsilon}\subset\gM_{\upsilon}$ be the  (reduced) exceptional divisor of 
  $\Bl_{\cI}\gM\to \gM$. Thus $E_{\upsilon}\cong\gM_u$, and $E_{\upsilon}$ is $\QQ$-Cartier. Let $\phi_{\upsilon}\colon \gM_{\upsilon}\to\gM$ be the natural map. 
By~\Ref{thm}{firstblowup}, the period map $\gM_{\upsilon}\dra\cF(19)^{*}$ is regular in a neighborhood of $E_{\upsilon}$. Moreover, letting $\cL$ be the ample $\QQ$- line bundle on $\gM$ descended from the ample generator of $|\cO_{\PP^3}(4)|$, the line-bundle $\phi_{\upsilon}^{*}\cL(-\epsilon E(\upsilon))$ is ample for $\epsilon$ positive and sufficiently small. 

\subsection{Blow up of the point $\omega$}\label{subsec:sblowomega}
\subsubsection{The GIT moduli space  for $K3$'s which are double covers of $\PP^1\times\PP^1$}
The   GIT moduli space that we will consider is 
\begin{equation}\label{hypermoduli}
\gM_h:=|\cO_{\PP^1}(4)\boxtimes \cO_{\PP^1}(4)|\gquot \Aut(\PP^1\times\PP^1).
\end{equation}
 Given $D\in|\cO_{\PP^1}(4)\boxtimes \cO_{\PP^1}(4)|$, we let $\pi\colon X_D\to\PP^1\times\PP^1$ be the double cover ramified over $D$, and  $L_D:=\pi^{*}\cO_{\PP^1}(1)\boxtimes \cO_{\PP^1}(1)$. If $D$ has ADE singularities, then $(X_D,L_D)$ is a hyperelliptic quartic $K3$.  We recall that if  $(X,L)$ is a hyperelliptic quartic $K3$ surface,  the map $\varphi_L$  associated to the complete linear system $ |L|\cong\PP^3$ is  regular, and it is the double cover of an irreducible quadric ${Q}$, branched over a divisor $B\in |\cO_{Q}(4)|$  with ADE singularities. Vice versa, the double cover of an irreducible quadric surface, ${Q}\subset\PP^3$, branched over a divisor $B\in |\cO_{Q}(4)|$  with ADE singularities is  a hyperelliptic quartic $K3$ surface. The period space for $\gM_h$ is $\cF_h$; we let 
\begin{equation}\label{perhyp}
\gp_h\colon \gM_h\dra \cF_h^{*}
\end{equation}
 be the extension of the period map to the Baily-Borel compactification. 
\begin{theorem}\label{thm:secondblowup}
\begin{enumerate}
\item
A divisor in $|\cO_{\PP^1}(4)\boxtimes \cO_{\PP^1}(4)|$ with ADE singularities is $\Aut(\PP^1\times\PP^1)$ stable, hence there exists an open dense subset   $\gM^I_h\subset\gM_h$  parametrizing  isomorphism classes of  hyperelliptic quartic $K3$ surfaces such that $\varphi_L(X)$ is a smooth quadric. 
\item
The period map $\gp_h$ defines an isomorphims between $\gM^I_h$ and the complement of the \lq\lq hyperelliptic\rq\rq\ divisor $H_h(\cF_h)$ in $\cF_h$ (the divisor $H_h(18)\subset\cF(18)$ in the notation of~\cite{log1}).
\end{enumerate}
\end{theorem}
\begin{proof}
Item~(1) is a result of Shah, in fact it is contained in Theorem 4.8 of~\cite{shah4}. Item~(2) follows from the discussion above. In fact let $y\in\cF_h$. Then there exists  a hyperelliptic quartic $K3$ surface $(X,L)$ (unique up to isomorphism) whose period point is $y$, and the quadric ${Q}:=\varphi_L(X)$ is smooth if and only if $y\notin H_h(\cF_h)$.  
\end{proof}
\subsubsection{The germ of $\gM$ at  $\omega$ in the analytic topology}
Let $q\in H^0(\PP^3,\cO_{\PP^3}(2))$ be a non degenerate quadratic form, and let  ${Q}\subset\PP^3$ be the smooth quadric with equation $q=0$. Let $O(q)$ be the associated orthogonal group; then $\PO(q)=\Aut {Q}$ is the 
stabilizer of $[q^2]\in |\cO_{\PP^3}(4)|$. We have a decomposition of $H^0(\PP^3,\cO_{\PP^3}(4))$ into $O(q)$-modules 
\begin{equation*}
H^0(\PP^3,\cO_{\PP^3}(4))=q\cdot H^0(\PP^3,\cO_{\PP^3}(2))\oplus H^0({Q},\cO_{Q}(2)). 
\end{equation*}
Note that the first submodule is reducible (it contains a trivial summand, spanned by $q^2$), while the second one is  irreducible. We identify ${Q}$ with $\PP^1\times\PP^1$, $\PO(q)$ with $\Aut(\PP^1\times\PP^1)$, and $H^0({Q},\cO_{Q}(2))$ with $H^0(\PP^1\times\PP^1,\cO_{\PP^1}(4)\boxtimes \cO_{\PP^1}(4))$. The projectivization of $q\cdot H^0(\PP^3,\cO_{\PP^3}(2))$ is equal to the projective (embedded) tangent space at $[q^2]$ of the orbit $\PGL(4)[q^2]$.  
Thus, by Luna's \'etale slice Theorem, we have natural \'etale map
\begin{equation*}
H^0({Q},\cO_{Q}(2))\gquot O(q) \lra \gM,
\end{equation*}
mapping $[0]$ to $\omega$. In particular we have an isomorphism of analytic germs
\begin{equation}\label{omgerme}
(H^0({Q},\cO_{Q}(2))\gquot O(q),[0]) \overset{\sim}{\lra} (\gM,\omega).
\end{equation}
\subsubsection{Partial extension of the period map on the blow up of $\omega$}
The map $\phi_{\upsilon}\colon\gM_{\upsilon}\to\gM$ is an isomorphism over $\gM\setminus\{\upsilon\}$; abusing notation, we denote by the same symbol $\omega$ the unique point in $\gM_{\upsilon}$ lying over $\omega\in\gM$.  Let 
$\phi_{\omega}\colon\widetilde \gM\lra  \gM_{\upsilon}$ be the blow-up of the reduced point $\omega$, and let $E_{\omega}\subset\widetilde \gM$ be the exceptional divisor. We let $\phi:=\phi_{\upsilon}\circ\phi_{\omega}$, and $\wt{\gp}=\gp\circ\phi$. Thus we have 
\begin{equation*}
\xymatrix{
& \wt{\gM}\ar@{->}[dl]_{\phi}\ar@{-->}[dr]^{\wt{\gp}}& \\ 
\gM & & \cF^{*}
}
\end{equation*} 
\begin{proposition}\label{prp:eomident}
Keeping notation as above,   $E_{\omega}$ is naturally identified  with the   hyperelliptic GIT moduli space $\gM_h$, and  
the  restriction  of  $\wt{\gp}$ to $E_{\omega}$ is equal to the period map $\gp_h$ of~\eqref{perhyp}.
\end{proposition}
\begin{proof}
Let $\psi_{\omega}\colon\gM_{\omega}\to\gM$ be the blow-up of the reduced point $\omega$, and let $\gp_{\omega}\colon \gM_{\omega}\to\cF^{*}$ be the composition $\gp\circ\psi_{\omega}$. Since $\omega$ and $\upsilon$ are disjoint subschemes of $\gM$, the exceptional divisor of $\psi_{\omega}$ is identified with $E_{\omega}$, and it suffices to prove that the statement of the proposition holds with $\wt{\gM}$ and $\wt{\gp}$ replaced by $\gM_{\omega}$ and $\gp_{\omega}$ respectively. Let ${\bf D}\subset |\cO_{\PP^3}(4)|$ be the closed subset of double quadrics, i.e.~the closure of the orbit $\PGL(4)(2{Q})$, where ${Q}\subset\PP^3$ is a smooth quadric. Let $\pi\colon P\to 
 |\cO_{\PP^3}(4)|$ be the blow up of (the reduced) ${\bf D}$, and let $E_{\bf D}\subset P$ be the exceptional divisor of $\pi$. Then $\PGL(4)$ acts on $P$ (because ${\bf D}$ is $\PGL(4)$-invariant), and the action lifts to an action on the line bundle $\cO_P(E_{\bf D})$. Let $\cL$ be the hyperplane line bundle on $ |\cO_{\PP^3}(4)|$, and let $t\in\QQ_{+}$ be such that $f^{*}\cL(-t E_{\bf D})$ is an ample $\QQ$-line bundle on $P$.  Then $\PGL(4)$ acts on the ring of global sections $R(P,\pi^{*}\cL(-t E_{\bf D}))$, and hence we may consider the GIT moduli space
 \begin{equation*}
\wh{\gM}(t):=\Proj\left(R(P,\pi^{*}\cL(-t E_{\bf D}))^{\PGL(4)}\right).
\end{equation*}
By Kirwan~\cite{kirwan}, there exists $t_0>0$ such that the blow down map $\pi\colon P\to  |\cO_{\PP^3}(4)|$ induces a regular map $\wh{\psi}(t)\colon\wh{\gM}(t)\to\gM$ 
for all $0<t<t_0$, and moreover  $\wh{\gM}(t)$ and $\wh{\psi}(t)$  are identified with $\gM_{\omega}$ and $\psi_{\omega}$ respectively. But now the identification of $E_{\omega}$   with the   hyperelliptic GIT moduli space $\gM_h$ follows at once from  the isomorphism of germs in~\eqref{omgerme}. The assertion on the period map follows from the description of the germ  $(\gM,\omega)$ and a standard semistable replacement argument.
\end{proof}
\subsection{Identification of $\cF(1-\epsilon)$ and $\wt{\gM}$}\label{subsec:tropforest}
Let $\cL$ be the  $\QQ$ line bundle on $\gM$ induced by the hyperplane line bundle on $|\cO_{\PP^3}(4)|$, and let $\wt{\cL}:=\phi^{*}\cL$. 
Let $E:=E_{\upsilon}+E_{\omega}$. Then
\begin{equation}\label{harrypotter}
(\wt{\gp}^{-1})^{*}( E)|_{\cF}=H_h+H_u=2\Delta.
\end{equation}
In fact we have the set-theoretic equalities $\wt{\gp}(E_{\upsilon})\cap\cF=H_u$, and $\wt{\gp}(E_{\omega}\setminus\Ind(\wt{\gp}))\cap\cF=H_h\setminus H_h^{(2)}$, thus 
in order to finish the proof of~\eqref{harrypotter} one only needs to compute multiplicities; they are equal to $1$ because $\wt{\gp}^{-1}$ has degree $1$.
By~\eqref{harrypotter} and Equation~(4.1.2) of~\cite{log1}, we get   
\begin{equation*}
(\wt{\gp}^{-1})^{*}(\wt{\cL}(-\epsilon E))|_{\cF}\cong \cO_{\cF}(\lambda+(1-2\epsilon)\Delta).
\end{equation*}
Thus $\wt{\gp}^{-1}$ induces a homomorphism 
\begin{equation}\label{sollevamento}
R(\gM,\wt{\cL}(-\epsilon E))\lra R(\cF,\lambda+(1-2\epsilon)\Delta).
\end{equation}
\begin{proposition}\label{prp:isoanello}
The homomorphism in~\eqref{sollevamento} is an isomorphism of rings. 
\end{proposition}
\begin{proof}
This is because $\wt{\gp}^{-1}$ is an isomorphism between $\cF\setminus H_h^{(2)}$, which has complement of codimension $2$ in $\cF$, and an open subset of $\wt{\gM}$ which again has complement of dimension $2$ in $\wt{\gM}$. 
\end{proof}
\begin{corollary}\label{crl:cor1minuseps}
The  restriction of $\wt{\gp}^{-1}$ to $\cF$ defines an isomorphism 
\begin{equation*}
\Proj(\cF,\lambda+(1-\epsilon)\Delta)\cong\wt{\gM}
\end{equation*}
for  small enough $0<\epsilon$.
\end{corollary}
\begin{proof}
If  $0<\epsilon$ is small enough, then $\wt{\cL}(-\epsilon E)$ is ample on $\wt{\gM}$, and hence $\Proj R(\gM,\wt{\cL}(-\epsilon E))\cong\wt{\gM}$. Thus the corollary follows from~\Ref{prp}{isoanello}.
\end{proof}
%

%%%%%%%%%%%%%%%%%%%%%%%%%%%%%%%%%%%%%%%
\subsection{The first flip of the GIT quotient ($\beta=1/2$)}\label{subsec:rescodim2}
We recall that the curve $W_1\subset\gM$ contains the point $\omega$ and does not contain $\upsilon$. We let $\wt{W}_1\subset\wt{\gM}$ be the strict transform of $W_1$. We will perform a surgery of $\wt{\gM}$ along  $\wt{W}_1$ in order to obtain our candidate for $\cF(1/3,1/2)$, notation as in~\eqref{intmodel}.
More precisely, we will start by constructing a birational map $\wh{\gM}\to\wt{\gM}$, which is an isomorphism away from $\wt{W}_1$, and  over $\wt{W}_1$ is a weighted blow along normal slices to $\wt{W_1}$. Let $E_1$ be the exceptional divisor of  $\wh{\gM}\to\wt{\gM}$; then $E_1\cong \wt{W}_1\times\gM_c$, where $\gM_c$ is a GIT compactification of the moduli space of degree-$4$ polarized $K3$ surfaces which are double covers of a quadric cone with branch divisor not containing the vertex of the cone. Let  $\wh{\gp}\colon \wh{\gM}\dra \cF$ be the period map and $\wh{\gM}_{reg}\subset \wh{\gM}$ be the subset of regular points of $\wh{\gp}$; 
we will show that, if $p\in\wt{W}_1$, then the  intersection $\wh{\gM}_{reg}\cap\{p\}\times \gM_c$ (here $\{p\}\times \gM_c\subset E_1$) coincides with the set of regular points of the period map $\gM_c\dra \cF$, and that the restriction of $\wh{\gp}$ is equal to the period map $\gM_c\dra \cF$. It follows that $\wh{\gp}$ is constant on the 
slices  $\{p\}\times \gM_c\subset E_1$, and the image of the restriction of $\wh{\gp}$ to the set of regular points of $E_1$ is the complement of $\Delta^{(3)}=\im(f_{16,19})$ in the codimension-$2$ locus $\Delta^{(2)}=\im(f_{17,19})$ (notation as in~\cite{log1}).  Now, $\wh{\gM}$ can be contracted along $E_1\to\gM_c$, let $\gM_{1/2}$ be the contraction; the results mentioned above strongly suggest that  $\gM_{1/2}$ is isomorphic to $\cF(1/3,1/2)$. 
\subsubsection{The action on quartics of the automorphism group of polystable surfaces in $W_1$.}\label{subsubsec:ortogonalrep}
Let $q:=x_0^2+x_1^2+x_2^2$, and let 
\begin{equation}\label{effeab}
f_{a,b}:=(q+ax_3^2)(q+bx_3^2), 
\end{equation}
where $(a,b)\not=(0,0)$. Then $V(f_{a,b})$ is a polystable quartic, and its equivalence class belongs to $W_1$. Conversely, if $V(f)$ is a polystable quartic whose equivalence class belongs to $W_1$, then up to  projectivities and rescaling,  $f=f_{a,b}$ for some $(a,b)\not=(0,0)$.  The points in $\gM$ representing  $V(f_{a,b})$ and $V(f_{c,d})$ are equal if and only if $[a,b]=[c,d]$, or $[a,b]=[d,c]$. Lastly,  $V(f_{a,b})$ represents $\omega$ if and only if $a=b$.  

Suppose that $a\not=b$. Then every element of $\Aut V(f_{a,b})$ fixes $V(x_3)$ and the point $[0,0,0,1]$. It follows that $\Aut V(f_{a,b})$  is equal to the image of the natural map $\Ort(q)\to\PGL(4)$. In particular $\SO(q)$ is an index $2$ subgroup of 
$\Aut  V(f_{a,b})$, and hence the double cover of $\SO(q)$, i.e.~$\SL_2$, acts on $ V(f_{a,b})$. The decomposition into irreducible representations of the action of $\SL_2$   on $\CC[x_0,\ldots,x_3]_4$ is as follows: 
\begin{equation}\label{longrepr}
\begin{matrix}
 \scriptstyle \CC[x_0,\ldots,x_2]_4 & \scriptstyle  \oplus & \scriptstyle \CC[x_0,\ldots,x_2]_3\cdot x_3 
 & \scriptstyle \oplus & \scriptstyle \CC[x_0,\ldots,x_2]_2\cdot  x^2_3 & \scriptstyle 
\oplus & \scriptstyle \CC[x_0,\ldots,x_2]_1\cdot  x^3_3 & \scriptstyle \oplus & \scriptstyle \CC\cdot  x_3^4 \\
\scriptstyle   V(8)\oplus V(4)\oplus V(0) &    & \scriptstyle V(6)\oplus V(2)  
 &  & \scriptstyle V(4)\oplus V(0) &  &  \scriptstyle V(2) &  &  \scriptstyle V(0)
\end{matrix}
\end{equation}
Now let us determine the sub-representation $U_{a,b}$ containing $[f_{a,b}]$ and such that $\Hom([f_{a,b}],U_{a,b}/[f_{a,b}])$ is the  tangent space at $V(f_{a,b})$ to the orbit $\PGL(4) V(f_{a,b})$ (we only assume that  $(a,b)\not=(0,0)$). Let $\ell_i\in\CC[x_0,\ldots,x_3]_1$ for $i\in\{0,\ldots,3\}$; we will write out the term multiplying $t$ in the expansion of   $f_{a,b} (x_0+t\ell_0,\ldots,x_3+t\ell_3)$ as element of $\CC[x_0,\ldots,x_3]_4[t]$ for various choices of $\ell_i$'s. For $\ell_i=\mu_i x_3$, we get
\begin{equation}\label{primadef}
4q\left(\sum_{i=0}^2\mu_i x_i \right)x_3+2(a+b)\mu_3 q x_3^2+2(a+b)\left(\sum_{i=0}^2\mu_i x_i\right)x^3_3+ 4 ab\mu_3  x_3^4.
\end{equation}
Letting $\ell_i\in\CC[x_0,x_1,x_2]_1$, we get
\begin{equation}\label{secondadef}
4q\left(\sum_{i=0}^2\ell_i x_i \right)+2(a+b) q \ell_3 x_3+2(a+b)\left(\sum_{i=0}^2\ell_i x_i\right)x_3^2+  4ab \ell_3 x_3^3.
\end{equation}
It follows that
\begin{equation}\label{decofu}
 U_{a,b} \cong
 \begin{cases}
 V(4)\oplus V(2)^2\oplus V(0)^2 & \text{if $a\not=b$,} \\
 V(4)\oplus V(2)\oplus V(0)^2 & \text{if $a=b$.} 
\end{cases}
\end{equation}
The difference between the two cases is due to the different behaviour of the $V(2)$-representations appearing  in~\eqref{primadef}, \eqref{secondadef} and contained in the direct sum $ \CC[x_0,\ldots,x_2]_3\cdot x_3 \oplus  \CC[x_0,\ldots,x_2]_1\cdot x^3_3 $. If  $a\not=b$, the representations in~\eqref{primadef} and~\eqref{secondadef}   are distinct, if $a=b$ they are equal. 
\subsubsection{The germ of $\wt{\gM}$ at  points of $\wt{W}_1\setminus E_{\omega}$.}\label{subsubsec:tangquad}
The map $\phi\colon\wt{\gM}\to\gM$ is an isomorphism away from $\{\omega,\upsilon\}$. Since $W_1$ does not contain $\upsilon$, 
the germ of $\wt{\gM}$ at  a point  $\wt{x}\in(\wt{W}_1\setminus E_{\omega})$ is identified by $\phi$ with the germ of  $\gM$ at $x:=\phi(\wt{x})$.
Let us examine the germ of  $\gM$ at a point  $x\in(W_1\setminus\{\omega\})$. There exists $(a,b)\in\CC^2$, with $a\not=b$, such that 
a polystable quartic representing $x$ is $V(f_{a,b})$, where $f_{a,b}$ is as in~\eqref{effeab}. Keeping notation as in~\Ref{subsubsec}{ortogonalrep}, $\SL_2$ acts on $ V(f_{a,b})$. 
Let $N_{a,b}\subset \CC[x_0,\ldots,x_3]_4$ be the sub $\SL_2$-representation
\begin{equation}\label{enneab}
N_{a,b}:=V(8)\oplus V(6)\oplus R\cdot x_3^2\oplus \la 2q x_3^2+(a+b)x_3^4\ra,
\end{equation}
where $R\subset \CC[x_0,x_1,x_2]_2$ is the summand isomorphic to $V(4)$, and let 
\begin{equation}\label{normslice}
{\bf N}_{a,b}:=\{V(f_{a,b}+g) \mid g\in N_{a,b}\}.
\end{equation}
\begin{proposition}
Keeping notation as above, ${\bf N}_{a,b}$ is an  $\Aut V(f_{a,b})$-invariant normal slice to the orbit $\PGL(4) V(f_{a,b})$. 
\end{proposition}
\begin{proof}
Let $U_{a,b}\subset \CC[x_0,\ldots,x_3]_4$ be as in~\Ref{subsubsec}{ortogonalrep}; thus
 $\PP(U_{a,b})$ is the projective tangent space at $V(f_{ab})$ to the orbit 
 $\PGL(4) V(f_{a,b})$. Then $U_{a,b}$ is the sum of the two  $\SL_2$-representations in~\eqref{primadef} and~\eqref{secondadef} (and, as representation,  it is given by the first case in~\eqref{decofu}), and  
it follows that  the  $\SL_2$-invariant affine space in~\eqref{normslice} is  transversal to $\PP(U_{a,b})$ at $V(f_{a,b})$. Lastly, ${\bf N}_{a,b}$ is   $\Aut V(f_{a,b})$-invariant because 
 $\Aut V(f_{a,b})$ is generated by the image of $\SL_2$ and the reflection in the plane $x_3=0$. 
\end{proof}
 The natural  map
\begin{equation}\label{triste}
\psi\colon {\bf N}_{a,b}\gquot  \Aut V(f_{a,b})   \lra   \gM
\end{equation}
 is \'etale at $V(f_{a,b})$ by Luna's \'etale slice Theorem. For later use, we make the following observation.
\begin{claim}\label{clm:normstrat}
Keep notation and assumptions as above, in particular $a\not=b$. Let $\eta\colon {\bf N}_{a,b}\to \gM$ be the composition of the quotient map ${\bf N}_{a,b}\to {\bf N}_{a,b}\gquot  \Aut V(f_{a,b})$ and the map $\psi$ in~\eqref{triste}. Then
\begin{equation}\label{inwuno}
\eta(\{V(f_{a,b}+t(2q x_3^2+(a+b)x_3^4)) \mid t\in\CC\})\subset W_1.
\end{equation}
Moreover, let $\cU\subset {\bf N}_{a,b}$ be an $\Aut V(f_{a,b})$-invariant open (in the classical topology)  neighborhood of $f_{a,b}$ such that the restriction of $\psi$ to 
$\cU\gquot  \Aut V(f_{a,b})$ is an isomorphism onto $\psi(\cU\gquot  \Aut V(f_{a,b}))$; then $x\in\cU$ is mapped to $W_1$ by $\eta$ and has closed $\SL_2$-orbit if and only if 
$x=V(f_{a,b}+t( 2q x_3^2+(a+b)x_3^4))$ for some $t\in\CC$.
\end{claim}
\begin{proof}
The first statement follows from a direct computation. In fact, an easy argument shows that there exist holomorphic functions $\varphi,\psi$ of the complex variable $t$ vanishing at $t=0$, such that 
$$(q+(a+\varphi(t)) x_3^2)\cdot (q+(b+\psi(t)) x_3^2)=f_{a,b}+t(2q x_3^2+(a+b)x_3^4).$$
The second statement holds because $W_1$ is an irreducible curve, and so is the left-hand side of~\eqref{inwuno}.
\end{proof}
\subsubsection{The germ of $\wt{\gM}$ at  the unique point in $\wt{W}_1\cap E_{\omega}$.}\label{subsubsec:norcia}
Let $\gM_{\omega}\to\gM$ be the blow-up of (the reduced) $\omega$. We may work on $\gM_{\omega}$, since $W_1$ does not contain $\upsilon$. Let $P\to |\cO_{\PP^3}(4)|$ be the blow-up with center the closed subset ${\bf D}$ parametrizing double quadrics, and let $E_{\bf D}$ be the exceptional divisor. 
 By~\cite{kirwan} the blow-up  $\gM_{\omega}$ is identified with the quotient of $P$ 
 by the natural action of $\PGL(4)$ (with a polarization close to the pull-back of the hyperplane line-bundle on  $|\cO_{\PP^3}(4)|$) - see the proof of~\Ref{prp}{eomident}).  We will describe an $\SL_2$-invariant normal slice in $P$ to the $\PGL(4)$-orbit of a point representing the unique  point in $\wt{W}_1\cap E_{\omega}$. 
First, recall that we have an identification $E_{\omega}=\gM_h$, where $\gM_h$ is the GIT hyperelliptic moduli space  in~\eqref{hypermoduli}, see~\Ref{prp}{eomident}.
The unique  point in $\wt{W}_1\cap E_{\omega}$ is represented by a point in $E_{\bf D}$ mapping to a smooth quadric ${Q}\subset\PP^3$, and  corresponding 
to $\ell^4\in \PP(H^0(\cO_{Q}(4)))$ (recall that the fiber of the exceptional divisor over $Q$ is identified with $\PP(H^0(\cO_{Q}(4)))$), where $0\not=\ell\in H^0(\cO_{Q}(1))$ is a section with \emph{smooth} zero-locus (a smooth conic); moreover the points we have described have closed orbit in the locus of $\PGL(4)$-semistable points. 
\begin{notationconv}\label{notconv:puntoinpi}
We  represent the unique  point in $\wt{W}_1\cap E_{\omega}$ by the point with closed orbit $(V(q+ax_3^2),x_3^4)\in E_{\bf D}$ (notation as above), where  $q$ is as 
in~\Ref{subsubsec}{ortogonalrep} and $a\not=0$. In order to simplify notation, we let ${Q}_a:=V(q+a x_3^2)$, and $p:=({Q}_a,x_3^4)\in E_{\bf D}$. 
\end{notationconv}
Now let $S\subset \CC[x_0,\ldots,x_3]_4$ be the sub $\SL_2$-representation
\begin{equation}\label{essea}
S:=V(8)\oplus V(6)\oplus R\cdot x_3^2\oplus \la x_3^4\ra,
\end{equation}
where $R\subset \CC[x_0,x_1,x_2]_2$ is the summand isomorphic to $V(4)$. (Notice the similarity with~\eqref{enneab}.) Let
\begin{equation}\label{boldessea}
{\bf S}_a:=\{V(f_{a,a}+g) \mid g\in S\}
\end{equation}
\begin{claim}\label{clm:essetrasv}
Keeping notation as above, the double quadric $V(f_{a,a})$ is an isolated and reduced point of   the scheme-theoretic intersection  between the affine space ${\bf S}_a$ and the closed 
  ${\bf D}\subset |\cO_{\PP^3}(4)|$ parametrizing  double quadrics.  
\end{claim}
\begin{proof}
Of course $V(f_{a,a})\in {\bf D}$, because $f_{a,a}=(q+a x_3^2)^2$. Let  $T_{V(f_{a,a})}{\bf S}_a$ and $T_{V(f_{a,a})}{\bf D}$  be the  tangent spaces 
to  ${\bf S}_a$ and ${\bf D}$ at $V(f_{a,a})$ respectively; we must show that their intersection (as subspaces of $T_{V(f_{a,a})} |\cO_{\PP^3}(4)|$) is trivial. We have
\begin{equation*}
 T_{V(f_{a,a})}{\bf S}_a=\Hom(\la f_{a,a}\ra,\la S,f_{a,a}\ra/\la f_{a,a}\ra), \qquad T_{V(f_{a,a})}{\bf D}=\Hom(\la f_{a,a}\ra,U_{a,a}/\la f_{a,a}\ra), 
\end{equation*}
where  $U_{a,a}$ is as in~\Ref{subsubsec}{ortogonalrep}. As  is easily checked, 
\begin{equation}
S\cap U_{a,a}=\{0\}.
\end{equation}
Thus $\la S, f_{a,a}\ra\cap U_{a,a}=\la f_{a,a}\ra$, and the claim follows. 
\end{proof}
By~\Ref{clm}{essetrasv} the scheme-theoretic intersection ${\bf D}\cap{\bf S}_a$ is the disjoint union of the reduced singleton $\{V(f_{a,a})\}$ and a subscheme $Y_a$. 
Let ${\bf U}_a:=  {\bf S}_a\setminus Y_a$; then   ${\bf U}_a$ is an open neighborhood of $V(f_{a,a})$ in ${\bf S}_a$, and it is invariant under the action of $\Aut V(f_{a,a})$. 
Let $\wt{\bf U}_a\subset P$ be the strict transform of ${\bf U}_a$  (recall that   $P\to |\cO_{\PP^3}(4)|$ is the blow-up with center  ${\bf D}$), and let $\varphi\colon \wt{\bf U}_a \to {\bf U}_a$ be the restriction of the contraction $P\to  |\cO_{\PP^3}(4)|$. By~\Ref{clm}{essetrasv} $\varphi$ is the blow-up of the (reduced) point $V(f_{a,a})$.  
\begin{remark}
Since $f_{a,a},x_3^4\in {\bf S}_a$, the point  $p=({Q}_a,x_3^4)\in E_{\bf D}$ (see~\Ref{notconv}{puntoinpi}) belongs to $\wt{\bf U}_a$. Moreover the stabilizer (in $\PGL(4)$) of  
 $p$ is equal to $O(q)$ i.e.~to $\Aut V(f_{a,b})$ for $a\not=b$ (see~\Ref{subsubsec}{ortogonalrep}), and it preserves  $\wt{\bf U}_a$.  
\end{remark}
\begin{proposition}\label{prp:essetrasv}
Keeping notation as above, $\wt{\bf U}_a$ is a $\Stab(p)$-invariant normal slice to the orbit $\PGL(4) p$ in $P$. 
\end{proposition}
\begin{proof}
Let $Y:=\PGL(4) p$. We must  prove 
 that the tangent space  to $\wt{\bf U}_a$ at $p$ is transversal to the tangent space  to  $Y$ 
at $p$.  First notice that $\dim Y=12$ and $\dim\wt{\bf U}_a=22$, hence $\dim Y+\dim\wt{\bf U}_a=\dim P$. Thus it suffices to prove that
\begin{equation}\label{zeroint}
T_{p}Y\cap T_{p}\wt{\bf U}_a=\{0\}.
\end{equation}
Let   $\pi\colon P\to |\cO_{\PP^3}(4)|$  be the blow up of ${\bf D}$. 
  By~\Ref{clm}{essetrasv}, 
  \begin{equation*}
d\pi(p)(T_p\wt{\bf U}_a)=\Hom(\la f_{a,a}\ra,\la f_{a,a},x_3^4\ra/\la f_{a,a}\ra).
\end{equation*}
On the other hand, $d\pi(p)(T_p Y)=T_{\pi(p)}{\bf D}$, and hence $d\pi(p)(T_p\wt{\bf U}_a)\cap d\pi(p)(T_p Y)=\{0\}$. It follows that the intersection on the left hand side of~\eqref{zeroint} is 
contained in the kernel of the restriction of $d\pi(p)$ to $T_p\wt{\bf U}_a$, i.e.~$T_p(\wt{\bf U}_a\cap E_{\pi(p)})$, where $E_{\pi(p)}$ is the fiber of $E_{\bf D}\to {\bf D}$ over $\pi(p)=V(f_{a,a})$. 
Hence it suffices to prove that
\begin{equation}\label{releone}
T_{p}Y\cap T_p(\wt{\bf U}_a\cap E_{\pi(p)})=\{0\}.
\end{equation}
The fiber  $E_{\pi(p)}$ is naturally identified with $\PP H^0(\cO_{{Q}_a}(4))$. With this identification, we have
\begin{eqnarray*}
T_p Y\cap T_p E_{\pi(p)} & = & \Hom(\la x_3^4\ra,\CC[x_0,\ldots,x_3]_1\cdot x_3^3/\la x_3^4\ra), \\
T_p(\wt{\bf U}_a\cap E_{\pi(p)}) & = & \Hom(\la x_3^4\ra,S /\la x_3^4\ra).
\end{eqnarray*}
Here we are abusing notation: $\CC[x_0,\ldots,x_3]_1\cdot x_3^3$ and  $S$   stand for their images  in $H^0(\cO_{{Q}_a}(4))$.  Since the kernel of the restriction map 
$H^0(\cO_{\PP^3}(4))\to H^0(\cO_{{Q}_a}(4))$ is equal to $U_{a,a}$, Equation~\eqref{releone} follows from  the equalities
\begin{eqnarray*}
\la \CC[x_0,\ldots,x_3]_1\cdot x_3^3 , S \ra \cap U_{a,a} & = & \{0\}, \\
 (\CC[x_0,\ldots,x_3]_1\cdot x_3^3)\cap S & = & \la x_3^4\ra 
\end{eqnarray*}
\end{proof}
The natural  map
\begin{equation}\label{isotta}
\psi\colon \wt{\bf U}_{a}\gquot  \Stab (p)   \lra   \gM
\end{equation}
 is \'etale at $p$ by Luna's \'etale slice Theorem. 
The result below is the analogue of~\Ref{clm}{normstrat}.
\begin{claim}\label{clm:normstratdue}
Keep notation and assumptions as above. Let $\zeta\colon \wt{\bf U}_{a}\to \gM$ be the composition of the quotient map $\wt{\bf U}_{a}\to \wt{\bf U}_{a}\gquot  \Stab (p)$ and the map $\psi$ in~\eqref{isotta}. Let $C\subset \wt{\bf U}_{a}$ be the strict transform of the line $\{V(f_{a,a}+t x_3^4) \mid t\in\CC\}$. 
Then $\zeta(C)\subset W_1$.
Moreover, let $\cU\subset \wt{\bf U}_{a}$ be a $\Stab(p)$-invariant open (in the classical topology)  neighborhood of $p$ such that the restriction of $\psi$ to 
$\cU\gquot  \Stab V(p)$ is an isomorphism onto $\psi(\cU\gquot  \Stab (p))$; then $x\in\cU$ is mapped to $W_1$ by $\zeta$ and has closed $\SL_2$-orbit if and only if 
$x=V(f_{a,a}+tx_3^4)$ for some $t\in\CC$.
\end{claim}
\begin{proof}
First $(q+(a+u)x_3^2)(q+(a-u)x_3^2)=f_{a,a}-u^2 x_3^4$ shows that $\zeta(C)\subset W_1$. For the remaining statement see the proof of~\Ref{clm}{normstrat}.
\end{proof}
\subsubsection{Moduli of $K3$ surfaces which are generic double cones}\label{subsubsec:doppiocono}
Let $\Lambda$ be the graded $\CC$-algebra
\begin{equation}\label{eccolambda}
\Lambda:=\Symm^{\bullet}(V(4)^{\vee}\oplus V(6)^{\vee}\oplus V(8)^{\vee}), 
\end{equation}
where $V(2d)^{\vee}$ has degree $d$. 
Then $\PSL(2)$ acts on $\Proj\,\Lambda$, and $\cO_{\Proj\Omega}(1)$ is naturally linearized. The involution 
\begin{equation*}
\begin{matrix}
\Proj\,\Lambda & \lra & \Proj\,\Lambda \\
[f,g,h] & \mapsto [f,-g,h]
\end{matrix}
\end{equation*}
commutes with the action of $\PSL(2)$, and hence there is a (faithful) action of
\begin{equation}\label{eccogici}
G_c:=\PSL(2)\times\ZZ/(2)
\end{equation}
on  $\Proj\,\Lambda$. We let 
\begin{equation}\label{gitc}
\gM_c:=\Proj\,\Lambda\gquot G_c
\end{equation}
be the  GIT quotient. 
We will show that $\gM_c$ is naturally a compactification of the moduli space of hyperelliptic quartic $K3$ surfaces which are double covers of a quadric cone with branch divisor not containing the vertex of the cone.
First, we think of $\SL_2$ as the double cover of $\SO(q)$, where $q=x_0^2+x_1^2+x_2^2$ is as in~\Ref{subsubsec}{tangquad}, and correspondingly 
 $V(2d)$ is a subrepresentation of $\CC[x_0,x_1,x_2]_d$. We associate to $\xi:=(f,g,h)\in V(4)\oplus V(6)\oplus V(8)$,  the quartic 
\begin{equation}\label{tschirn}
B_{\xi}:=V(x_3^4+  fx_3^2 + g x_3 +h).
\end{equation}
Thus $V(4)\oplus V(6)\oplus V(8)$ is identified with the set of such quartics. Both $G_c$ and 
the multiplicative group $\CC^{*}$ act on the set of such quartics (the second group acts by rescaling $x_3$). The quotient of $(V(4)\oplus V(6)\oplus V(8))\setminus\{0\}$ by the $\CC^{*}$ action is $\Proj\,\Lambda$, hence $\gM_c$ is identified with the quotient 
$(V(4)\oplus V(6)\oplus V(8))\setminus\{0\}$ by the full $G_c\times\CC^{*}$-action. Given $[\xi]\in \Proj\,\Lambda$, we let $X_{\xi}$ be the double cover of the cone $V(q)\subset\PP^3_{\CC}$ ramified over the restriction of $B_{\xi}$  to $V(q)$, and $L_{\xi}$ be the degree-$4$ polarization of $X_{\xi}$ pulled back from $\cO_{\PP^3}(1)$. 
\begin{proposition}
Let  $[\xi]\in\Proj\,\Lambda$ be such that $X_{\xi}$ has rational singularities. Then $[\xi]$ is $G_c$-stable. 
The open dense subset of $\gM_c$ parametrizing isomorphism classes of such $[\xi]$ is the moduli space of polarized quartics which are double covers of a quadric cone with branch divisor not containing the vertex of the cone.
\end{proposition}
\begin{proof}
Let $[\xi]=[f,g,h]\in\Proj\,\Lambda$ be a non-stable point. Then by the Hilbert-Mumford Criterion there exist a point $a\in\PP^1$ (where $\PP^1$ is identified with the conic $V(q,x_3)$ via the Veronese embedding)
such that 
\begin{equation}\label{duetrequattro}
\mult_a (f)\ge 2,\qquad \mult_a (g)\ge 3,\qquad \mult_a (h)\ge 4.
\end{equation}
The point $a\in\PP^1$ is identified with a point   $p\in V(q,x_3)$ (as recalled above), which belongs to the quartic $B_{\xi}$. 
The inequalities in~\eqref{duetrequattro} give that the multiplicity at $p$ of the divisor $B_{\xi}|_{V(q)}$ is at least $4$, and hence the corresponding double cover of $V(q)$ (i.e.~$X_{\xi}$) does \emph{not} have rational singularities. This proves the first statement.
The rest of the proof is analogous to Shah's proof  (see Theorem 4.3 in~\cite{shah}) that $\gM_u$ (see~\eqref{gitunig}) is a compactification of the moduli space  for unigonal $K3$ surfaces. 
The key point is that any quartic not containing the vertex $[0,0,0,1]$ has such an equation after a suitable projectivity $\varphi$  (a Tschirnhaus transformation) of the form $\varphi^{*}x_i=x_i$, $\varphi^{*}x_3=x_3+\ell(x_0,x_1,x_2)$ where  $\ell(x_0,x_1,x_2)$ is homogeneous of degree $1$. 
\end{proof}
Let $[\xi]\in \Proj\,\Lambda$ be generic; then $(X_{\xi},L_{\xi})$ is a polarized quartic $K3$ surface whose period point belongs to $H_h^{(2)}$, which (see~\cite{log1}) is identified with $\cF(17)$ via the embedding $f_{17,19}\colon\cF(17)\hra\cF$.
Thus we have a rational period map
\begin{equation}\label{coneperiods}
\gp_c\colon \gM_c \dra \cF(17)^{*}\subset\cF^{*}.
\end{equation}
A generic polarized quartic $K3$ surface is a double cover of the quadric cone unramified over the vertex, and hence is isomorphic to  $(X_{\xi},L_{\xi})$ for a certain 
$[\xi]\in \Proj\,\Lambda$. By the global Torelli Theorem for $K3$ surfaces, it follows that the period map $\gp_c$ is birational.  
\subsubsection{Partial extension of the period map on a weighted blow-up: the case of a point in $\wt{W}_1\setminus E_{\omega}$}\label{subsubsec:extnab}
Let  $(a,b)\in\CC^2$, with $a\not=b$. Let $N_{a,b}$ be the $\SL_2$ representation in~\eqref{enneab}, and let $M_{a,b}$ be the sub-representation
\begin{equation}\label{emmeab}
M_{a,b}:=V(8)\oplus V(6)\oplus R\cdot x_3^2.
\end{equation}
Let ${\bf N}_{a,b}$ be the normal slice of 
$V(f_{a,b})$ defined in~\Ref{subsubsec}{tangquad}, and let ${\bf M}_{a,b}\subset {\bf N}_{a,b}$ be the  subspace  
\begin{equation*}
{\bf M}_{a,b}:= \{V(f_{a,b}+g) \mid g\in M_{a,b}\}.
\end{equation*}
Notice that
\begin{equation*}
\dim {\bf M}_{a,b}=21.
\end{equation*}
Let  $(z_1,\ldots,z_{5})$ be  coordinates on  $V(4)$, let $(z_{6},\ldots,z_{12})$ be  coordinates on  $V(6)$, and let 
$(z_{13},\ldots,z_{21})$ be  coordinates on  $V(8)$; thus $(z_1,\ldots,z_{21})$ are  coordinates on  ${\bf M}_{a,b}$ (with a slight abuse of notation) centered at $V(f_{a,b})$. Let $\sigma$ be the weight defined by 
\begin{equation}
\sigma(z_i):=
\begin{cases}
2 & \text{if $i\in\{1,5\}$,} \\
3 & \text{if $i\in\{6,12\}$,} \\
4 & \text{if $i\in\{13,21\}$.} 
\end{cases}
\end{equation}
Let $\wh{\bf M}_{a,b}:=\Bl_{\sigma}({\bf M}_{a,b})$ be the corresponding weighted blow up, and 
 let $E_{a,b}$ be the exceptional set of $\wh{\bf M}_{a,b}\to {\bf M}_{a,b}$. Thus $E_{a,b}$ is the weighted projective space $\PP(2^5,3^{7},4^{9})\cong \Proj\,\Lambda$, where $\Lambda$ is the graded ring in~\eqref{eccolambda} (with grading defined right 
after~\eqref{eccolambda}). 
The action of $\Aut(V_{f_{a,b}})=G_c$ (here $G_c$  is as in~\eqref{eccogici}) on ${\bf M}_{a,b}$ lifts to an action on  $\wh{\bf M}_{a,b}$. Thus there is an associated GIT quotients  $\wh{\bf M}_{a,b}\gquot G_c$. 
The map   $\wh{\bf M}_{a,b}\to {\bf M}_{a,b}$ induces a map
\begin{equation}\label{duetheta}
 \wh{\theta}\colon\wh{\bf M}_{a,b}\gquot  G_c \lra {\bf M}_{a,b}\gquot  G_c.
\end{equation}
  Moreover,  we have the set-theoretic equality
\begin{equation}\label{macerata}
\wh{\theta}^{-1}(\ov{V(f_{a,b})})_{red}=\Proj\,\Lambda\gquot   G_c=\gM_c. 
\end{equation}
 Since the natural map ${\bf M}_{a,b}\gquot   G_c\to\gM$ is dominant, it makes  sense to compose it with the (rational) period map $\gp\colon\gM\dra\cF^{*}$. Composing with the birational map in~\eqref{duetheta}, we get a rational map
\begin{equation}\label{visso2}
 \wh{\gp}_{a,b}\colon\wh{\bf M}_{a,b}\gquot   G_c\dra \cF^{*}.
\end{equation}
\begin{proposition}\label{prp:redtwotan}
With notation as above, the restriction of  $\wh{\gp}_{a,b}$ to $\wh{\theta}^{-1}(\ov{V(f_{a,b})})_{red}=\gM_c$ is equal to the composition of the automorphism
\begin{equation}\label{padoan}
\begin{matrix}
\gM_c & \overset{\varphi_{a,b}}{\lra} & \gM_c \\
[f,g,h] & \mapsto & [f,-\frac{i}{2}(a-b)g,-\frac{1}{4}(a-b)^2 h]
\end{matrix}
\end{equation}
and the period map in~\eqref{coneperiods}. Moreover 
$\wh{\gp}_{a,b}$ is regular at all points of  $\wh{\theta}^{-1}(\ov{V(f_{a,b})})_{red}$ where $\gp_c$ is regular. 
\end{proposition}
\begin{proof}
Let  $[\xi]=[f,g,h]\in\Proj\,\Lambda=E_{a,b}$ be a $G_c$-semistable point with corresponding point $\overline{[\xi]}\in\gM_c$, and let $[\eta]=\varphi_{a,b}(\overline{[\xi]})$. Suppose that the  period map $\gp_c$ is regular at $[\eta]$. We will prove that  if $\Delta\subset\CC$ is a disc centered at $0$, and  $\Delta\to \wh{\bf M}_{a,b}$ is an analytic map mapping $0$ to  $[\xi]$ and no other point to the exceptional divisor $E_{a,b}$, then the period map is defined on a neighborhood of $0\in\Delta$, and its value at $0$ is equal to the period point of $(X_{\eta},L_{\eta})$. This will prove the Theorem, by~\Ref{crl}{regonblow}.
By~\Ref{prp}{regonblow} the statement that we just gave boils down to the following computation. First, we identify $V(2d)$ with the corresponding $\SO(q)$-sub-representation of $\CC[x_0,x_1,x_2]_d$; thus  
   $f,g,h\in\CC[x_0,x_1,x_2]$ are homogeneous of degrees  $2$, $3$ and $4$ respectively. Now let   $\cX\subset\PP^3\times\Delta$ be the hypersurface given by the equation
\begin{equation}\label{ipersup}
0=(q+ax_3^2)(q+bx_3^2)+t^2 x_3^2 (f+t F)+t^3 x_3 (g+tG)+t^4 (h+tH)
\end{equation}
where $F\in \CC[x_0,x_1,x_2]_2[[t]]$,  $G\in \CC[x_0,x_1,x_2]_3[[t]]$, and $H\in \CC[x_0,x_1,x_2]_4[[t]]$. 
Now consider the $1$-parameter subgroup of $GL_4(\CC)$ defined by $\lambda(t):=\diag(1,1,1,t)$. We let $\cY\subset\PP^3\times\Delta$ be the closure of 
\begin{equation*}
\{([x],t) \mid t\not=0,\quad \lambda(t)[x]\in\cX\}.
\end{equation*}
Then $Y_t \cong X_t$ for $t\not=0$, and $\cY$  has equation 
\begin{equation}\label{cioccolata}
0=q^2 +t^2(a+b) x_3^2 q+t^4(ab x_3^4+ x_3^2 f+x_3 g+ h)+t^5(\ldots)
\end{equation}
Let $\nu\colon\wt{\cY}\to \cY$ be the normalization of $\cY$.  Dividing~\eqref{cioccolata} by $t^4 x_i^4$, we get that the ring of regular functions of the affine set $\nu^{-1}(\cY \cap \PP^3_{x_i})$ is generated over $\CC[\cY\cap\PP^3_{x_i}]$ by the rational function $\xi_i:=q/(x_i^2 t^2)$, which satisfies the equation
\begin{equation*}
0=\xi_i^2+(a+b) \left(\frac{x_3}{x_i}\right)^2\xi_i+(ab x_3^4+ x_3^2 f+x_3 g+ h)/x_i^4
+t(\ldots)
\end{equation*}
It follows  that  for $t\to 0$  the quartics $X_t$ approach the double cover of $V(q)$ branched over the intersection with the quartic 
\begin{equation}
0=((a+b)x_3^2)^2-4(ab x_3^4+ x_3^2 f+x_3 g+ h)=(a-b)^2 x_3^4-4x_3^2   f-4 x_3 g -4 h.
\end{equation}
\end{proof}
\subsubsection{Partial extension of the period map on a weighted blow-up: the unique point in  $\wt{W}_1\cap E_{\omega}$}\label{subsubsec:exttua}
Let $a\not=0$, and
\begin{equation*}
\wt{\bf V}_a:=\wt{\bf U}_a\cap E_{\bf D}.
\end{equation*}
(We recall that  $E_{\bf D}$ is the exceptional divisor of the blow-up  $P\to |\cO_{\PP^3}(4)|$ with center the closed subset ${\bf D}$ parametrizing double quadrics.)  
Thus, letting $S$ be as in~\eqref{essea}, we have
\begin{equation}\label{tilvua}
\wt{\bf V}_a=\PP(S)=\PP(V(8)\oplus V(6)\oplus R\cdot x_3^2\oplus \la x_3^4\ra),\qquad \dim\wt{\bf V}_a=21.
\end{equation}
Let $p:=({Q}_a,x_3^4)\in \wt{\bf V}_a$, see~\Ref{notconv}{puntoinpi}. Then $\wt{\bf V}_a$ is mapped to itself by $\Stab(p)$, and by restriction of the map $\psi$ in~\eqref{isotta} we get a map
\begin{equation*}
 \wt{\bf V}_{a}\gquot  \Stab(p) \lra \wt{\gM}.
\end{equation*}
We define a weighted blow up of $\wt{\bf V}_a$ with center $p$ as follows. First, by~\eqref{tilvua} we have the following description of an affine neighborhood $T$ of $p\in \wt{\bf V}_a$: 
\begin{equation*}
\begin{matrix}
V(8)\oplus V(6)\oplus R\cdot x_3^2 & \lra & T \\
\alpha & \mapsto & [x_3^4+\alpha]
\end{matrix}
\end{equation*}
Let  $(z_1,\ldots,z_{5})$ be  coordinates on  $R\cdot x_3^2=V(4)$, let $(z_{6},\ldots,z_{12})$ be  coordinates on  $V(6)$, and let 
$(z_{13},\ldots,z_{21})$ be  coordinates on  $V(8)$; thus $(z_1,\ldots,z_{21})$ are  coordinates on  $T$ (with a slight abuse of notation) centered at the point $p$. Let $\sigma$ be the weight defined by 
\begin{equation}
\sigma(z_i):=
\begin{cases}
2 & \text{if $i\in\{1,5\}$,} \\
3 & \text{if $i\in\{6,12\}$,} \\
4 & \text{if $i\in\{13,21\}$.} 
\end{cases}
\end{equation}
(Note: we are proceeding exactly as in~\Ref{subsubsec}{extnab}.)
Let $\wh{\bf V}_{a}:=\Bl_{\sigma}(\wt{\bf V}_{a})$ be the corresponding weighted blow up, and 
 let $E_{a}$ be the corresponding exceptional divisor. Thus $E_{a}$ is the weighted projective space $\PP(2^5,3^{7},4^{9})\cong \Proj\,\Lambda$, where $\Lambda$ is the graded ring in~\eqref{eccolambda} (with grading defined right 
after~\eqref{eccolambda}). 
The action of $\Aut(p)$ on $\wt{\bf V}_{a}$ lifts to an action on  $\wh{\bf V}_{a}$ There is an associated GIT quotient  $\wh{\bf V}_{a}\gquot \Stab(p)$, and a regular map
\begin{equation*}
\wh{\eta}\colon\wh{\bf V}_{a}\gquot \Stab(p)\lra \wt{\bf V}_{a}\gquot  \Stab(p).
\end{equation*}
We have the set-theoretic equality
\begin{equation}\label{macerata2}
\wh{\eta}^{-1}(\ov{p})_{red}=\Proj\,\Lambda\gquot   G_c=\gM_c. 
\end{equation}
We have a rational map
\begin{equation}\label{visso}
 \wh{\gp}_{a}\colon\wh{\bf V}_{a}\gquot   \Aut(p)\dra \cF^{*}.
\end{equation}
\begin{proposition}\label{prp:perexc}
With notation as above, the restriction of  $\wh{\gp}_{a}$ to $\wh{\eta}^{-1}(\ov{p})_{red}=\gM_c$ is equal to the  period map in~\eqref{coneperiods}. Moreover 
$\wh{\gp}_{a}$ is regular at all points of  $\wt{\eta}^{-1}(\ov{p})_{red}$ where $\gp_c$ is regular. 
\end{proposition}
\begin{proof}
Let  $[\xi]=[f,g,h]\in\Proj\,\Lambda=E_{a}$ be a 
$G_c$-semistable point with corresponding point $\eta\in\gM_c$. Suppose that the  period map $\gp_c$ is regular at $\eta$. We will prove that  if $\Delta\subset\CC$ is a disc centered at $0$, and  $\Delta\to \wh{\bf V}_{a}$ is an analytic map mapping $0$ to  $[\xi]$ and no other point to the exceptional divisor $E_a$, then the period map is defined on a neighborhood of $0\in\Delta$, and its value at $0$ is equal to the period point of $(X_{\eta},L_{\eta})$. This will prove the Theorem, by~\Ref{crl}{regonblow}. By~\Ref{prp}{regonblow}, the previous statement boils down to the following computation.
Let  
   $f,g,h\in\CC[x_0,x_1,x_2]$ be homogeneous of degrees  $2$, $3$ and $4$ respectively, not all zero. Let $C_t\subset V(q+ax_3^2)$ be the intersection with the quartic
\begin{equation*}
x_3^4+t^2x_3^2 (f+tF)+t^3 x_3 (g+tG)+t^4 (h+tH)=0.
\end{equation*}
where $F\in \CC[x_0,x_1,x_2]_2[[t]]$,  $G\in \CC[x_0,x_1,x_2]_3[[t]]$, and $H\in \CC[x_0,x_1,x_2]_4[[t]]$. 
We will show that $C_t$ for $t\not=0$ approaches for $t\to 0$, the curve
\begin{equation*}
q=x_3^4+x_3^2 f+x_3 g+h=0.
\end{equation*}
In fact it suffices to  consider the limit for $t\to 0$ of $\lambda(t) C_t$, where $\lambda$ is the $1$-PS $\lambda(t)=(1,1,1,t)$. 
\end{proof}
\subsubsection{A global modification of $\wt{\gM}$ and partial extension of the period map}\label{subsubsec:unmezzo}
Let ${\bf T}\subset|\cO_{\PP^3}(4)|$ be the closure of the set of $\PGL(4)(\CC)$-translates of $V(f_{a,b})$, for all $(a,b)\in\CC^2$. Thus ${\bf T}$ is a closed, $\PGL(4)(\CC)$-invariant subset, containing ${\bf D}$ (the set of double quadrics), and
\begin{equation}
\dim{\bf T}=13.
\end{equation}
Let $\wt{\bf T}\subset P$ be the strict transform of ${\bf T}$ in the blow-up $\pi\colon P\to|\cO_{\PP^3}(4)|$ with center ${\bf D}$. The set of semistable points 
$\wt{\bf T}^{ss}\subset \wt{\bf T}$ (for a polarization  $\pi^{*}\cL(-\epsilon E_{\bf D})$ close to $\pi^{*}\cL$, see the proof of~\Ref{prp}{eomident}) is the union of 
 the set of points of $\wt{\bf T}\setminus E_{\bf D}$ which are mapped by $\pi$ to quartics $\PGL(4)(\CC)$-equivalent to $V(f_{a,b})$ for some $a\not=b$, and of  
 $\wt{\bf T}\cap E^{ss}_{\bf D}$. The latter set consists of the $\PGL(4)(\CC)$-translates of the points $(Q_a,x_3^4)$ defined  in~\Ref{notconv}{puntoinpi}.

In~\Ref{subsubsec}{extnab} and~\Ref{subsubsec}{exttua} we defined a weighted blow up of an explicit normal slice to $\wt{\bf T}$ at points $x\in \wt{\bf T}^{ss}$. That construction can be globalized: one obtains a modification $\wh{\pi}\colon\wh{P}\to P$ which is an isomorphism away from $P\setminus \wt{\bf T}$, and replaces  
$\wt{\bf T}^{ss}$ by a locally trivial fiber bundle over $\wt{\bf T}^{ss}$ with fiber isomorphic to  the weighted projective space $\PP(2^5,3^{7},4^{9})$. In fact the weighted blow up is isomorphic to the usual blow up of a suitable ideal, see Remark~2.5 of~\cite{andblow}, hence one may define an ideal $\cI$ co-supported on $\wt{\bf T}$ such that $\wh{P}=\Bl_{\cI}P$. 

Let $E_{\wt{\bf T}}$ be the exceptional divisor of $\wh{\pi}$. Letting $\cL_P:=\pi^{*}\cL(-\epsilon E_{\bf D})$ be a polarization of $P$ as above, we may consider the GIT quotient of $\wh{P}$ with $\PGL(4)(\CC)$-linearized polarization $\cL_{\wh{P}}:=\pi^{*}\cL_P(-t E_{\wt{\bf T}})$, call it $\wh{\gM}(t)$. 
For $0<t$ small enough, the map $\wh{\pi}$ induces a regular map $\wh{\gM}(t)\to\wt{\gM}$. From now on we drop the parameter $t$ from our notation; thus $\wh{\gM}$ denotes $\wh{\gM}$ for $t$ small. 

The image of  $E_{\wt{\bf T}}$ in $\wh{\gM}$ is a fiber bundle  
\begin{equation*}
\rho\colon E_1\to \wt{W}_1, 
\end{equation*}
with fiber $\gM_c$ over every point. Let $\wh{\gp}\colon \wh{\gM}\dra \cF^{*}$ be the period map. We claim that the restriction of $\wh{\gp}$ to the fiber of
 $E_1\to \wt{W}_1$ over $x$  is regular away from the indeterminacy locus of $\gp_c\colon\gM_c\dra\cF^{*}$, and it has the same value, provided we compose with the automorphism of $\gM_c$ given by~\eqref{padoan} if $x\notin E_{\omega}$ and $\wh{\pi}(x)=[V(f_{a,b})]$. 
 
 In order to prove the claim it suffices to prove the following. Let  
 $\Delta\subset\CC$ be a disc centered at $0$, and let $\Delta\to \wh{\gM}$ be an analytic map mapping $0$ to  a point $\wh{x}\in E_1$ such that the period map $\gp_c$  is regular at the point 
 $\eta\in\gM_c=\rho^{-1}(\rho(\wh{x})$ corresponding to $\wh{x}$, and suppose that  $(\Delta\setminus\{0\})$ is mapped to the complement of $E_1$ and into the locus where the period map is regular; then the value at $0$ of the extension of the period map on $\Delta\setminus\{0\}$ is equal to the period point of $(X_{\eta},L_{\eta})$. We may assume that  $\Delta\to \wh{\gM}$ lifts to an analytic map  $\tau\colon\Delta\to \wh{P}$ mapping $0$ 
 to a point of $E_{\wt{T}}$ with closed orbit (in the semistable locus) lifting $\wh{x}$. 
 In~\Ref{subsubsec}{extnab} and~\Ref{subsubsec}{exttua} we have checked that the value at $0$ of the extension behaves as required if $\wh{\pi}\circ\tau(\Delta)$ is contained in the normal slice to $\wt{T}$ at the point  $\wh{\pi}\circ\tau(0)$ (defined in~\Ref{subsubsec}{extnab} and~\Ref{subsubsec}{exttua} respectively). 
 
 It remains to prove that it behaves as required also if the latter condition does not hold. If  $\wh{\pi}\circ\tau(0)\notin E_{\bf D}$, then the argument is similar to that given in~\Ref{subsubsec}{extnab}; one simply replaces $a,b\in\CC$ by holomorphic functions $a(t),b(t)$ where $t\in\Delta$.   
 
 If  $\wh{\pi}\circ\tau(0)\in E_{\bf D}$, one needs a separate argument. 
 The relevant computation goes as follows.
Let  
   $f,g,h\in\CC[x_0,x_1,x_2]$ be homogeneous of degrees  $2$, $3$ and $4$ respectively, not all zero.  Let   $\cX\subset\PP^3\times\Delta$ be the hypersurface given by the equation
\begin{equation}\label{ipersup2}
(q+x_3^2)^2+t^{4k} x_3^4+t^{4k+6p} x_3^2 (f+tF)+t^{4k+9p} x_3 (g+tG)+t^{4k+12p} (h+tH)=0,
\end{equation}
where  
\begin{equation*}
 F\in \CC[x_0,x_1,x_2]_2[[t]],\quad  G\in \CC[x_0,x_1,x_2]_3[[t]],\quad H\in \CC[x_0,x_1,x_2]_4[[t]].
\end{equation*}
Let   $\lambda(t):=\diag(1,1,1,t^4)$, and let $\cY\subset\PP^3\times\Delta$ be the closure of 
\begin{equation*}
\{([x],t) \mid t\not=0,\quad \lambda(t)[x]\in\cX\}.
\end{equation*}
Thus $Y_t \cong X_t$ for $t\not=0$, and $\cY$  has equation 
\begin{equation}\label{cioccolata2}
q^2 +2t^8 q x_3^2 +t^{16} x_3^4+t^{4k+16} x_3^4+t^{4k+6p+8} x_3^2 (f+tF)+t^{4k+9p+4} x_3 (g+tG)+ t^{4k+12p} (h+tH)=0.
\end{equation}
Dividing the above equation by $t^{16}$ we find that  the rational function $\xi_i:=q/(x_i^2 t^8)$ satisfies the equation
\begin{equation*}
0=\xi_i^2+2\left(\frac{x_3}{x_i}\right)^2 \xi_i+ (x_3^4+t^{4k} x_3^4+ 
t^{4k+6p-8}   x_3^2 (f+tF)+t^{4k+9p-12} x_3(g+tG)+  t^{4k+12p-16}(h+tH))/x_i^4.
\end{equation*}
It follows that the fiber at $t=0$ of the normalization of $\cY$ is the double cover of $V(q)$ ramified over the intersection with the limit for $t\to 0$ of the quartic 
\begin{equation*}
4x_3^4-4( x_3^4+t^{4k} x_3^4+ 
t^{4k+6p-8}   x_3^2 (f+tF)+t^{4k+9p-12} x_3 (g+tG)+  t^{4k+12p-16}(h+tH))=0.
\end{equation*}
Replacing $x_3$ by $t^{-3p+4}x_3$ we get  that the fiber at $t=0$ of the normalization of $\cY$ is the double cover of $V(q)$ ramified over the intersection with the quartic
\begin{equation}\label{protozoo}
x_3^4+x_3^2f+x_3g+h=0.
\end{equation}
Let us explain why the above computation proves the required statement. Let $\epsilon\colon \Delta\to|\cO_{\PP^3}(4)|$ be the analytic map defined by $\epsilon(t):=X_t$. Then $\im(\epsilon)\subset S_1$, where $S_1$ is as in~\eqref{boldessea} (notice that $X_0=f_{1,1}$).  Let $\wt{U}_1$ be the blow up of $S_1\setminus Y_1$ with center $V(f_{1,1})$, see~\Ref{subsubsec}{norcia}, and let
$\wt{\epsilon}\colon\Delta\to \wt{U}_1$ be the lift of $\epsilon$ (by shrinking $\Delta$ we may assume that $\im(\epsilon)\cap Y_1=\es$). Then $\wt{\epsilon}(0)=p=(Q_1,x_3^4)$, notation as in~\Ref{notconv}{puntoinpi}. 

Now, choose a basis $\{a_0,\ldots,a_{21}\}$ of the $\SL_2$-representation $S$ given by~\eqref{essea} adapted to the decomposition 
in~\eqref{essea}; more precisely $a_0=x_3^4$,  $\{a_1,\ldots,a_{5}\}$ is a basis of $R\cdot x_3^2$, $\{a_6,\ldots,a_{12}\}$ is a basis of $V(6)$, and 
$\{a_9,\ldots,a_{21}\}$ is a basis of $V(8)$. Let $\{w_0,\ldots,w_{21}\}$ be the basis dual to  $\{a_0,\ldots,a_{21}\}$; then $(w_0,\ldots,w_{21})$ are coordinates on an affine neighborhood of $V(f_{a,a})$ in ${\bf S}_a$, centered at $V(f_{a,a})$. Next set $y_0=w_0$, and $y_i=w_i/w_0$ for $i\in\{1,\ldots,22\}$. Then $(y_0,\ldots,y_{21})$ are coordinates on an affine neighborhood of $p\in \wt{\bf U}_1$, centered at $p$. Let $(z_1,\ldots,z_{21})$ be the affine coordinates introduced in~\Ref{subsubsec}{exttua}; we may assume that 
$y_i|_{\wt{\bf V}_1}=z_i$ for  $i\in\{1,\ldots,22\}$. 

In the coordinates  $(y_0,\ldots,y_{21})$ we have 
\begin{equation*}
\scriptstyle
\wt{\epsilon}(t)=(t^{4k},t^{6p} (f_1+tF_1),\ldots,t^{6p} (f_5+tF_5),t^{9p} (g_5+tG_5),\ldots,t^{9p} (g_{12}+tG_{12}),t^{12p} (h_{13}+tH_{13}),\ldots,t^{12p} (h_{21}+tH_{21})),
\end{equation*}
with obvious notation: $(f_1,\ldots,f_5)$ are the coordinates of $f$ in the basis   $\{a_1,\ldots,a_{5}\}$, etc.
The computation above shows that the extension at $0$ of the period map is equal to the period point of the double cover of  $V(q)$ ramified over the intersection with the quartic defined by~\eqref{protozoo}, and hence the period map is regular at the point corresponding to $[f,g,h]$ by~\Ref{prp}{regonblow} and~\Ref{crl}{regonblow}. 
\subsubsection{The first flip and a contraction of $\wh{\gM}$.}\label{subsubsec:whyemhat}
The divisor $E_1\subset\wh{\gM}$ is isomorphic to $\wt{W}_1\times\gM_c$. The normal bundle of $E_1$ restricted to the fibers of the projection $E_1\to\gM_c$ is negative; it follows that  (in the analytic category) there exists a contraction $\wh{\gM}\to \gM_{1/2}$ of $E_1$ along the the fibers of $E_1\to\gM_c$.  We claim that $\gM_{1/2}$ must be isomorphic to $\cF(1/3,1/2)$.  In fact, let $\wh{\gp}\colon\wh{\gM}\dra\cF$ be the period map (notice: contrary to previous notation, the codomain is $\cF$, not $\cF^{*}$). The generic fiber of  $E_1\to\gM_c$   is in the regular locus of $\wh{\gp}$, and is mapped to a constant: it follows that 
\begin{equation}
0=\wh{\gp}^{*}(\lambda)\cdot (\wt{W}_1\times\{[f,g,h]\})=\wh{\gp}^{*}(\Delta)\cdot (\wt{W}_1\times\{[f,g,h]\}),\qquad [f,g,h]\in\gM_c.
\end{equation}
On the other hand, letting $p\in\wt{W}_1$, and adopting  the notation of~\cite{log1}, we have 
\begin{equation}
\wh{\gp}(\{p\}\times \gM_{c,reg})\subset \im(f_{17,19}).
\end{equation}
(Here $\gM_{c,reg}$ is the set of regular points of the period map $\wh{\gp}_c\colon\gM_c\dra \cF$; by~\Ref{prp}{perexc} it is equal to the intersection of  $\{p\}\times \gM_{c,reg}$ with the set of regular points of $\wh{\gp}$.)
By Proposition~5.3.7 of~\cite{log1} we have $f_{17,19}^{*}(\lambda+\beta\Delta)=(1-2\beta)\lambda(17)+\beta\Delta(17)$. Now, $\Delta(17)=H_h(17)/2$, and 
$\wh{\gp}(\{p\}\times \gM_c)$ avoids the support of $H_h(17)=\im f_{16,17}$. Thus 
\begin{equation}
\wh{\gp}^{*}(\lambda+\beta\Delta)|_{\{p\}\times \gM_c}=\wh{\gp}^{*}((1-2\beta)\lambda)|_{\{p\}\times \gM_c}.
\end{equation}
The conclusion is that $\wh{\gp}^{*}(\lambda+\beta\Delta)$ contracts all of $E_1$ to a point if $\beta\ge 1/2$ (and is trivial on $E_1$ if $\beta=1/2$), while if $\beta<1/2$, 
then the restriction of $\wh{\gp}^{*}(\lambda+\beta\Delta)$ to $E_1$  is the pull-back of an ample line bundle on $\gM_c$. Thus we expect that for $\beta<1/2$ close to $1/2$ 
the ($\QQ$) line-bundle  $\wh{\gp}^{*}(\lambda+\beta\Delta)$ is the pull-back of an ample  ($\QQ$) line bundle on $\gM_{1/2}$, and hence $\gM_{1/2}$ is identified with $\cF(\beta)$, because the period map would be birational map which is an isomorphism in codimension $2$ and pulls- back an ample line bundle to an ample line bundle.

%%%%%%%%%%%%%%%%%%%%%%%%%%%%%%%%%%%%%%%%%%%%

%%%%%%%%%%%%%%%%%%%%%%%%%%%%%%%%%%%%%%%%%%%%
\section{Semistable reduction for Dolgachev singularities, and the last three flips}\label{sec:dolgachev} 
In the present section, we will provide evidence in favour of the predictions that there are flips corresponding to $\beta\in\{\frac{1}{6},\frac{1}{7},\frac{1}{9}\}$ (the critical values of $\beta$ closest to $\beta=0$, which corresponds to $\cF^*$), with centers birational to the loci of quartics with $E_{14}$, $E_{13}$, and $E_{12}$ singularities respectively.
 There is a strong similarity with the first steps in the Hassett-Keel program. Specifically,  in the variation of log canonical models $\calM_g(\alpha)=\Proj(\overline \calM_g,K_{\overline {\calM}_g}+\alpha\Delta_{ \overline{\calM}_g})$ (for $\alpha\in[0,1])$ for the moduli space of genus $g$  curves $\overline{\calM}_g$, the first critical value is $\alpha=\frac{9}{11}$ which corresponds to replacing the curves with elliptic tails by cuspidal curves. Similarly, at the next critical value $\alpha=\frac{7}{10}$, the locus of curves with elliptic bridges is replaced by the locus of curves with tacnodes (see \cite{hh1,hh2} for details). In the proposed analogy, the singularities $E_{12}$, $E_{13}$, and $E_{14}$  (the simplest $2$ dimensional non-log canonical singularities) correspond to cusps and tacnodes, while, as we will see, certain  lattice polarized  $K3$ surfaces  correspond to elliptic tails and bridges.

\subsection{KSBA (semi)stable replacement} According to the general KSBA philosophy, for varieties of general type there exists a canonical compactification obtained by allowing degenerations with semi-log-canonical (slc) singularities and ample canonical bundle. In particular, any $1$-parameter degeneration has a canonical limit with slc singularities. However, when studying GIT one ends up with compactifications that allow non-slc singularities. For example, the GIT compactification for quartic curves will allow quartics with cusp singularities. Thus a natural question is: {\it given a degenerations $\sX/\Delta$ of varieties of general type such that the general fiber is smooth (or mildly singular), but such that $X_0$ does not have slc singularities, to find a stable KSBA replacement $X_0'$}. Of course, $X_0'$ depends on the original fiber $X_0$ and on the family $\sX/\Delta$ (i.e.~the choice of the curve in the moduli space with limit $X_0$). Motivated by the Hassett-Keel program, Hassett \cite{hassett} studied the influence of certain classes of curve singularities on the KSBA (semi)stable replacement (in this case, the usual nodal curve replacement). Hassett's perspective is to consider a curve $C_0$  with a unique non-slc (i.e.~non-nodal) singularity, and to examine $\mathscr C/\Delta$, a generic smoothing of $C_0$. The question is what can be said about the semi-stable replacement $C_0'$ of $C_0$. Of course, one component of $C_0'$ will be the normalization $\widetilde C_0$ of $C_0$ (assuming that this normalization is not a rational curve). The remaining components (and the gluing to $\widetilde C_0$) of $C_0'$ (the ``tail part'') will depend on the non-slc singularity of $C_0$ and its smoothing; one determines them by a local computation. The classical example is the semi-stable replacement for  curves with an ordinary cusp (see~\cite[\S3.C]{hm}), that we briefly review below. 

\begin{example}[Semi-stable replacement for cuspidal curves]\label{expl:cuspex} 
Locally (in the analytic topology) a curve in a neighborhood of an ordinary cusp has  equation $y^2+x^3=0$, and a generic $1$-parameter smoothing will be given by $\mathscr C:=V(t+y^2+x^3)\to \Delta_t$. After a base change of order $6$, which is necessary to make the local monodromy action unipotent, one obtains a surface $V(t^6+y^2+x^3)\subset (\bC^3,0)$ with a simple elliptic singularity at the origin. The weighted blow-up of the origin will resolve this singularity, and the resulting exceptional curve $E$ is an elliptic curve (explicitly it is $V(t^6+y^2+x^3)\subset W\bP(1,3,2)$). The new family $\mathscr C'$ (obtained by base change and weighted blow-up) will be a semi-stable family of curves, with the new central fiber consisting of the union of the normalization of $C_0$ and of the exceptional curve $E$ (``the elliptic tail'') glued at a single point. Note that instead of a weighted blow-up, one can use several regular blow-ups, these will lead first to a semi-stable curve with additional rational tails, which can be then contracted to give the stable model (with a single elliptic tail). The two blow-up (and then blow-down) processes are equivalent; the weighted blow-up has the advantage of being minimal, and it generalizes well in our situation.  
\end{example}

As mentioned above, Hassett \cite{hassett} has generalized this for certain types of planar curve singularities (essentially weighted homogeneous, and related). In higher dimension (e.g. surfaces), much less is known - there is a similar computation (for surfaces with triangle singularities) to the elliptic curve example contained in an unpublished letter of Shepherd-Barron to Friedman (in connection to \cite{friedman-mon} -- such examples tend to give degenerations with finite, or even trivial, monodromy). Similar computations appear in~\cite{gallardo}, and what is needed for our purposes will be reviewed below. 

Of course, we are concerned with degenerations of $K3$ surfaces, thus the KSBA replacement strictly speaking doesn't make sense (the main issue is 
non-uniqueness of the replacement).  Nonetheless, given a degeneration $\sX^*/\Delta^*$ with general fiber a $K3$, there exists a filling with $X_0'$ being a surface with slc singularities (and trivial dualizing sheaf). This follows from the Kulikov-Person-Pinkham theorem and Shepherd-Barron \cite{sbnef, sbpolarization}. Furthermore, if  $X_0$ has a unique non-log canonical singularity, we can  ask (mimicking Hassett~\cite{hassett}): {\it What is the KSBA replacement for a quartic surface $X_0$ with a single $E_{12}$ singularity?} In this case the resolution $\widehat X_0$ is rational (this is analogous to the fact that the normalization of a cuspidal cubic curve is rational), and thus the focus is on the ``tail'' part.  

\subsection{Dolgachev singularities} The  singularities that interests us are particular cases of Dolgachev singularities \cite{dolgachevs} (aka triangle singularities or exceptional unimodal singularities, the latter is the terminology used by Arnold et al.~\cite{arnold}).
They are arguably the simplest $2$ dimensional non-log canonical singularities, for this reason we view them as analogues of $1$ dimensional   ordinary cusps. {\it Dolgachev singularities} are hypersurface singularities with the property that they have a  (non-minimal) resolution with exceptional divisor $E+ E_1+ E_2+ E_3$, where $E^2=-1$, $E_1^2=-p$, $E_2^2=-q$, $E_3^2=-r$, and the curves $E_i$ only meet $E$ transversely (comb type picture).    By contracting the $E_i$'s, we obtain a partial resolution with a rational curve $E$ going through $3$ quotient singularities of types $\frac{1}{p}(1,1)$, 
 $\frac{1}{q}(1,1)$ and $\frac{1}{r}(1,1)$. While any $(p,q,r)$ (with $\frac{1}{p}+\frac{1}{q}+\frac{1}{r}<1$) gives a non-log canonical surface singularity, only  $14$ choices of integers $(p,q,r)$ lead to hypersurface singularities, these  are the Dolgachev singularities.  The \emph{Dolgachev numbers of the singularity} are 
 $p,q,r$. The cases relevant to us  are  $E_{12}$, $E_{13}$, and $E_{14}$, with Dolgachev numbers  $(2,3,7)$, $(2,4,5)$, and $(3,3,4)$ respectively.

\begin{remark}
Very relevant in this discussion is the so called $T_{p,q,r}$ graph (for $p,q,r$ positive integers). This consists of a central node, together with $3$ legs of lengths $p-1$, $q-1$, and $r-1$ respectively. As usual to such a graph, one can associate an even lattice by giving a generator of norm $-2$ for each node, and two generators are orthogonal unless the corresponding nodes are joined by an edge in the graph (in which case, we define the intersection number to be $1$). 
The cases $\frac{1}{p}+\frac{1}{q}+\frac{1}{r}>1$ corresponding precisely to the ADE Dynkin graphs (with ADE associated lattices). For example $(1,p,q)$ corresponds to $A_{p+q-1}$, while $(2,3,3)$ corresponds to $E_6$. Note also $\frac{1}{p}+\frac{1}{q}+\frac{1}{r}>1$ is equivalent to the associated lattice being negative semi-definite. The three cases with $\frac{1}{p}+\frac{1}{q}+\frac{1}{r}=1$ correspond to the extended Dynkin diagrams of type $\widetilde E_r$ ($r=6,7,8$), and in these cases the associated lattice is negative semi-definite. Finally, the cases with $\frac{1}{p}+\frac{1}{q}+\frac{1}{r}<1$ lead to a hyperbolic lattice. It is easy to compute that the absolute value of the discriminant will be $pqr\left(1-\left(\frac{1}{p}+\frac{1}{q}+\frac{1}{r}\right)\right)$. 
\end{remark}

The lattice of vanishing cycles associated to a Dolgachev singularity is $T_{p',q',r'}\oplus U$ for some integers $(p',q',r')$, which are called the Gabrielov numbers of the singularity. In particular, we note that $p'+q'+r'=(p'+q'+r'-2)+2=\mu$ is the Milnor number of the singularities (i.e. the rank of the lattice of vanishing cycles is the Milnor number). In other words, associated to a Dolgachev singularity there are two triples of integers: the Dolgachev numbers $(p,q,r)$ related to the resolution of the singularity, and the Gabrielov numbers  $(p',q',r')$ related to the lattice of vanishing cycles (and the local monodromy associated with the singularity). In Table \ref{table3} below we give these numbers for the cases relevant to us. 
Arnold observed that the $14$ Dolgachev singularities come in pairs of two with the property that the Dolgachev and Gabrielov numbers are interchanged. This is part of the so called strange duality (see \cite{ebeling} for a survey). The key point is that $T_{p,q,r}$ and $T_{p',q',r'}$ are mutually orthogonal in $E_8^2\oplus U^2$ (equivalently, after adding a $U$ to one of them, they can be interpreted as the Neron-Severi lattice and the transcendental lattice respectively for certain $K3$ surfaces, and thus one can view this as an instance of mirror symmetry for $K3$ surfaces, see \cite{dolgachevm}). 
 
\begin{table}[htb!]
\begin{tabular}{|c|c|c|}
\hline
Singularity& Dolgachev No. & Gabrielov No.\\
\hline
$E_{12}$& 2,3,7 & 2,3,7\\
$E_{13}$& 2,4,5 & 2,3,8\\
$E_{14}$& 3,3,4 & 2,3,9\\
\hline
\end{tabular}
\vspace{0.2cm}
\caption{The relevant Dolgachev Singularities}\label{table3}
\end{table}

\subsection{Deformations of Dolgachev singularities and periods of $K3$'s}
 Looijenga \cite{ltriangle1,ltriangle2} has studied the deformation space of Dolgachev singularities. Briefly, they  are unimodal, i.e.~they have $1$-parameter equisingular deformation. Within the equisingular deformation, there is a distinguished point corresponding to a singularity with $\bC^*$-action (equivalently the equation is quasi-homogeneous).  One can apply to that singularity Pinkham's theory of deformations of singularities with $\bC^*$-action. In this situation, there will be $1$-dimensional positive weight direction (i.e. there is an induced $\bC^*$ action on the tangent space to the mini-versal deformation, and the weights refer to this action) corresponding to the equisingular deformations. The remaining $(\mu-1)$ weights are negative and correspond to the smoothing directions. We denote by $S_{-}$ the germ corresponding to the negative weights. Because of the $\bC^*$-action, $S_{-}$ can be globalized and identified to an affine space. Thus $(S_{-}\setminus \{0\})/\bC^*$ is a weighted projective space of dimension $\mu-2$ (where $\mu$ is the Milnor number, e.g.~$\mu-2=10$ for $E_{12}$). The general theory of Pinkham states  that $(S_{-}\setminus \{0\})/\bC^*$ is to be interpreted as a moduli space of certain $2$ dimensional pairs $(X,H)$ ($H$ is to be interpreted as a hyperplane at infinity coming from a $\bC^*$-equivariant compactification of the singularity).
 Looijenga \cite{ltriangle1,ltriangle2} observed  that, in the case of Dolgachev singularities (with $\CC^{*}$-action),  the general point of  $(S_{-}\setminus \{0\})/\bC^*$ 
 parametrizes a couple $(X,H)$ where  $X$ is a (smooth) $K3$ surface, and $H$ is a $T_{p,q,r}$ configuration of rational curves ($(p,q,r)$ are the Dolgachev numbers of the singularity). In particular, the transcendental lattice of $X$ is $T_{p',q',r'}\oplus U$ (identified with the lattice of vanishing cycles for the triangle singularity), while $T_{p,q,r}$ is its Neron--Severi lattice. In conclusion, the weighted projective space  $(S_{-}\setminus \{0\})/\bC^*$ is birational to a locally symmetric variety $\sD/\Gamma$ corresponding to periods of $T_{p,q,r}$-marked $K3$ surfaces (the dimension is $20-(p+q+r-2)=22-(p+q+r)=p'+q'+r'-2=\mu-2$). Furthermore,  Looijenga \cite{ltriangle2} showed that the  structure of the Baily-Borel compactification $(\sD/\Gamma)^*$ is related to the adjacency of simple-elliptic and cusp singularities to the given Dolgachev singularity, and that the indeterminacy of the period map $(S_{-}\setminus \{0\})/\bC^*\dashrightarrow (\sD/\Gamma)^*$ is related to the triangle singularities adjacent to the given one (e.g. $E_{13}$ deforms to $E_{12}$ and this will lead to indeterminacy, that is resolved by Looijenga's theory; while, on the other hand $E_{12}$ deforms only to simple elliptic, cusp, or ADE singularities, and thus there is no indeterminacy).
 
 \begin{example}
The simplest case is the deformation of $E_{12}$. The singularity has equation $x^2+y^3+z^7=0$. In this situation, as explained, $E_{12}$ only deforms to log canonical singularities giving a regular period map, which in turn gives an isomorphism: 
$$W\bP(3,4,6,8,9,11,12,14,15,18,21)\cong \left(S_{-}\setminus \{0\}\right)/\bC^*\cong (\sD/\Gamma)^*.$$
The weights above are the negative weights with respect to the $\bC^*$-action on the tangent space to the mini-versal deformation of the singularity, which we recall can be identified with $\calO_{\CC^3,0}/J$, where $J:=(x,y^2,x^6)$ is the Jacobian ideal of $f:=x^2+y^3+z^7$.   In this example, $(\sD/\Gamma)^*$ is the Baily-Borel compactification for the moduli space of $T_{2,3,7}$-marked $K3$ surfaces (N.B.~$T_{2,3,7}\cong E_8\oplus U$; also, because of self-duality in this case, the transcendental lattice is $T_{2,3,7}\oplus U=E_8\oplus U^2$). 
\end{example}

\subsection{Relating the loci $W_6$, $W_7$ and $W_8$ to $Z^6$, $Z^7$ and $Z^9$}
Recall that  $W_8$, $W_7$ and $W_6$ are the closures in  $\gM$  of the loci parametrizing  polystable quartics with a singularity of type $E_{12}$, $E_{13}$ and $E_{14}$ respectively.
 The universal family of quartics gives a versal deformation for the $E_{12}$ singularity (this follows from Urabe's analysis  \cite{urabee12} of quartics with this type of singularities, or more generally from du Plessis--Wall \cite{dpw} and Shustin--Tyomkin \cite{shustin}), thus  at a quartic with $E_{12}$ singularities  such that the singularity has $\bC^*$-action, the germ of $(S_{-},0)$ can be interpreted as the normal direction to $W_8$. Then, $\left(S_{-}\setminus \{0\}\right)/\bC^*$ is nothing but the projectivized normal bundle, which is then the replacement via a (weighted) flip of the $W_8$ locus. On the other hand, as noted in the example above, $\left(S_{-}\setminus \{0\}\right)/\bC^*$ can be interpreted as the moduli of $T_{2,3,7}$-marked $K3$s, which is the same as our $Z^9$ locus in $\cF$ (the moduli of quartic $K3$ surfaces). The same considerations apply to the case of $E_{13}$ and $E_{14}$ singularities, but in those cases the identification of $\left(S_{-}\setminus \{0\}\right)/\bC^*$ with the moduli of $T_{2,4,5}$ (and $T_{3,3,4}$  respectively) marked $K3$s (which then correspond to $Z^7$ and $Z^6$ respectively) involves one (or respectively two) flips (corresponding to the fact that $E_{13}$ deforms to $E_{12}$, and similarly for $E_{14}$). This is  exactly as predicted in~\cite{log1}. 

The argument above almost establishes our claim that a flip replace the $Z^9$ locus in $\cF$  by $W_8$ (the $E_{12}$ locus) in $\gM$ (and similarly for $E_{13}$ and $E_{14}$). In the following subsection, we strengthen the evidence towards this claim by a one-parameter computation (which shows that indeed the generic KSBA replacement for a quartic with $E_{12}$ singularities (with $\CC^{*}$-action) is a $T_{2,3,7}$-marked $K3$). 

\begin{example} 
Let $[w,x,y,z]$ be homogeneous coordinates on a $3$ dimensional projective space, and let $X$ be the quartic defined by the equation
\begin{equation*}
x^2w^2-2xz^2w^2+y^3w+x^3z+z^4=0.
\end{equation*}
Computing partial derivatives, one finds that the singular set of $X$ consists of the single point $p:=[1,0,0,0]$.  In fact $X$ has an $E_{12}$ singularity at $p$, with $\CC^{*}$ action. To see why, we let $w=1$, and hence $(x,y,z)$ become affine coordinates. Then $p$ is the origin, and a local equation of $X$ near $p$ is
\begin{equation*}
(x-z^2)^2+y^3+x^3z=0.
\end{equation*}
Let  $(s,y,z)$ be new anaytic  coordinates, where $s=x-z^2$; the new equation is
\begin{equation*}
s^2+y^3+z^7+s^3z+3s^2 z^3+3sz^5=0.
\end{equation*}
One recognizes $s^2+z^7+s^3z+3s^2 z^3+3sz^5=0$ as an $A_6$ singularity (assign weight $1/2$ to $s$ and weight $1/7$ to $z$; since all other monomials appearing in the equation have weight strictly larger than $1$, it follows that the equation is analytically equivalent to $u^2+v^7=0$), and hence in a neighborhood of $p$, the quartic $X$ has anaytic equation $u^2+y^3+v^7=0$. This is exactly the local equation of an $E_{12}$ singularity with $\CC^{*}$ action. Since $X$ has no other singularity, it is stable by Shah, and $[X]$ belongs to $W_8$. 
\end{example}
\subsection{The semistable replacement for quartics with an $E_{12}$, $E_{13}$ or $E_{14}$ quasi-homogeneous singularity}\label{ssekcase}
 We are assuming that we are given a quartic surface $X_0$ with a unique $E_{k}$ singularity (for $k=12,13,14$) and such that the singularity has a $\bC^*$-action (the singularity, in local analytic coordinates,  is given  given by the equation in Table \ref{table4}). We are considering a generic smoothing $\sX/\Delta$ and we are asking what is the KSBA replacement associated to this family. The computation is purely local, similar to that occurring in Hassett \cite{hassett}. We will mimick the algorithm described in~\Ref{expl}{cuspex}. A generic smoothing is locally given by $$V(f(x,y,z)+t)\subset (\bC^4,0),$$ where $f$ is the local equation of the singularity as in Table \ref{table4}. We make a base change $t\to t^N$ so that the local monodromy is unipotent. Arnold et al (see \cite[Table on p. 113]{arnold}) have computed the spectrum of the singularities for the simplest type of hypersurface singularities, including ours. The spectrum encodes the log of the eigenvalues of the local monodromy, thus from Arnold's list it is immediate to find the base change giving unipotent monodromy; the relevant order $N$ for the base change is given in Table \ref{table4} below. It turns out that the resulting $3$-fold $\sX=V(f(x,y,z)+t^N)\subset (\bC^4,0)$ has a simple $K3$-singularity (analogue of simple elliptic) at the origin in the sense of Yonemura \cite{yonemura}. It follows that a suitable weighted  blow-up of $\sX$ at the origin will resolve this singularity, giving a $K3$ tail.  The tail $T$ will be one of the weighted $K3$ surfaces in the sense of M. Reid. What is specific in the situation analyzed here is that $T$  has $3$  singularities of type $A$  lying on the  exceptional divisor of the weighted blow-up  (a rational curve). A routine analysis (see Gallardo \cite{gallardo} for further details) gives the following result.

\begin{table}[htb!]
\begin{tabular}{|c|r|c|c|}
\hline
Singularity& Equation (with $\bC^*$-action) & Order $N$ for base change&Weights $(t,x,y,z)$\\
\hline
$E_{12}$& $x^2+y^3+z^7=0$ & $42$& $(1, 21,14,6)$\\
$E_{13}$& $x^2+y^3+yz^5=0$& $30$& $(1,15,10,4)$\\
$E_{14}$& $x^3+y^2+yz^4=0$ & $24$& $(1,8,12,3)$\\
\hline
\end{tabular}
\vspace{0.2cm}
\caption{The relevant Dolgachev Singularities}\label{table4}
\end{table}

\begin{proposition}
Let $\sX/\Delta$ be a generic smoothing of a Dolgachev singularity of type $E_k$ ($k=12,13,14$). Then, after a base change of order $N$ (as given in Table \ref{table4}), followed by a weighted blow-up with weights as given in the table, gives a new central fiber $X_0'$ which is the union of the partial resolution $\widehat{X}_0$ of $X_0$ (with quotient singularities  given by the Dolgachev numbers $(p,q,r)$) and a $K3$ surface $T$ with $3$  singularities of types $A_{p-1}$, $A_{q-1}$, and $A_{r-1}$ liying on the (rational) curve $C=T\cap \widehat{X}_0$.  Thus, the minimal resolution $\widetilde{T}$ of $T$ is a $T_{p,q,r}$-marked $K3$ surface, where $(p,q,r)$ are the  Dolgachev numbers of the  $E_k$ singularity. 
\end{proposition}
\begin{proof}
The equation of the tail is simply 
$$V(f(x,y,z)+t^N)\subset W\bP(1,w_x,w_y,w_z)$$
with $f$, $N$, and the weights as given in Table  \ref{table4}. Note that this is a weighted degree $N$ hypersurface in a weighted projective space such that the sum of weights satisfies
$$1+w_x+w_y+w_z=N.$$
This is precisely the $K3$ condition. 
\end{proof}
\begin{remark}
Computations and arguments of similar nature have been done in the thesis of P. Gallardo (some of them appearing in \cite{gallardo}), who was advised by the first author. We have learned about similar computations done by Shepherd-Barron from an unpublished letter to R. Friedman. 
\end{remark}

In conclusion we see that  the replacement of the quartics with quasi-homogeneous $E_{12}$, $E_{13}$, or $E_{14}$ singularities are $T_{2,3,7}$, $T_{2,4,5}$, $T_{3,3,4}$ marked $K3$ surfaces respectively. These are parametrized by  points of $Z^9,Z^8,Z^7$ respectively, see~\cite{log1}. Specifically, we have the following result.
\begin{proposition}\label{prop:zloci}
The loci $Z^9,Z^8,Z^7$ are naturally identified  with the moduli spaces of $T_{2,3,7}$, $T_{2,4,5}$, $T_{3,3,4}$-polarized  (in the sense of~\cite{dolgachevm}) $K3$ surfaces respectively. 
\end{proposition} 
\begin{proof}
As already noted, for $\frac{1}{p}+\frac{1}{q}+\frac{1}{r}<1$,  $T_{p,q,r}$ are hyperbolic lattices of signature $(1, p+q+r-3)$. Furthermore, the absolute value of their discriminant is $pqr-pq-pr-qr=pqr\left(1-\frac{1}{p}-\frac{1}{q}-\frac{1}{r}\right)$ (giving values $1,2,3$ respectively in our situation). It follows, that the three $T_{p,q,r}$ lattices considered here are isometric to $E_8\oplus U$, $E_7\oplus U$, and $E_6\oplus U$ respectively. Each of them has a unique embedding into the $K3$ lattice $E_8^2\oplus U^3$, and the corresponding orthogonal complements are $E_8\oplus U^2$, $E_8\oplus U^2\oplus A_1$, and $E_8\oplus U^2\oplus A_2$ respectively. This coincides with our definition of the $Z^9,Z^8,Z^7$ loci from \cite{log1}.
\end{proof}

%%%%%%%%%%%%%%%%%%%%%%%%%%%%%%%%%%%%%%%

%%%%%%%%%%%%%%%%%%%%%%%%%%%%%%%%%%%%%%%
\section{Looijenga's $\bQ$-factorialization}\label{sec:LooijengaQ}
The predictions of our previous paper \cite{log1} are concerned with the birational transformations that occur in the  period domain $\cF=\sD/\Gamma$. Our working assumption is that all the modifications that occur at the boundary of the Baily-Borel compactification $\cF^*$ are explained by  Looijenga's $\QQ$-factorialization~\cite{looijengacompact}, together with 
 the modifications occurring in $\cF$. More precisley we predict that, for $0<\epsilon_0<1/9$,  the birational map $\cF(\epsilon_0)\dra \cF^{*}$ is  regular, small, and an isomorphism over $\cF=\sD/\Gamma$, and that the strict transform of $\Delta$ is a relatively ample $\bQ$-Cartier divisor. In particular, we get a well defined birational model $\widehat \cF$ of $\cF^{*}$ by setting  $\widehat \cF:=\cF(\epsilon_0)$    for  $0<\epsilon_0<1/9$ - this is  Looijenga's $\QQ$-factorialization. 
   Our expectation is that, for a critical $\beta\in[1/9, 1]$, the center of the birational map $\cF(\beta-\epsilon)\dra \cF(\beta+\epsilon)$ is the proper transform of the appropriate $Z^j$ appearing in~\eqref{zedtower} ($Z^9$ for $\beta=1/9$, $Z^8$ for $\beta=1/7$, and so on) via the birational map  $\cF(\beta-\epsilon)\dra\cF$.
In particular, the above expectation predicts that the numbers of irreducible components of $\gM^{II}$ and $\gM^{III}$, and their dimensions, can be determined once one has a description of the inverse images in $\wh{\cF}$ of Type II, Type III strata,  and their intersections with the strict transforms of the $Z^j$'s. 
   In the present section we will spell out the predictions regarding $\gM^{II}$, and we will see that they match the computations of~\Ref{sec}{refshah}. 
   
Before we proceed with our computations, we note that there is a glaring discrepancy that seems to be against our predictions above: there are $8$ Type II components in $\gM$, while there are $9$ in $\widehat{\cF}$. In fact there is no  contradiction, as we will see that the missing component is contained in the closure of one of the $Z^k\subset \cF$ strata and thus will disappear in the associated flip (and it will be hidden in the Type IV locus in $\gM$). We note that compared with the case of degree $2$ $K3$ surfaces (\cite{shah}, \cite{looijengavancouver}) or cubic fourfolds (\cite{cubic4fold}, \cite{lcubic}), this is a new phenomenon which points to the interesting nature of the quartic example. 

\subsection{Looijenga's $\QQ$-factorialization and its Type II  boundary components}
The locally symmetric variety $\cF$ has at worst finite quotient singularity, and thus it is $\bQ$-factorial. Since the boundary $\cF^{*}\setminus \cF$ is of high codimension, any divisor of $\cF$ extends uniquely as a Weil divisor, but typically not a  $\bQ$-Cartier divisor. Looijenga \cite{looijengacompact} has constructed a $\bQ$-factorialization associated to any arithmetic hyperplane arrangement (or equivalently pre-Heegner divisor in the terminology of \cite{log1}). Here, we are interested in the $\bQ$-factorialization of the closure of $\Delta=\frac{1}{2} (H_u+H_h)$. 

\begin{definition}
Let $\widehat \cF\to \cF^*$ be the Looijenga $\bQ$-factorialization associated to the hyperplane arrangement $\mathscr H=\pi^{-1}(H_h\cup H_u)$ (where $\pi:\sD\to \sD/\Gamma$ is the natural projection) of hyperelliptic and unigonal pre-Heegner divisors. 
\end{definition}

From our perspective, it is immediate to see that the $\bQ$-factorialization coincides with one of our models:
\begin{proposition}
Let  $0<\epsilon\ll 1$.  Then the composition of birational maps $\widehat \cF\to \cF^*$  and $\cF^{*}\dra\cF(\epsilon)$ is an isomorphism $\wh{\cF}\overset{\sim}{\lra} \cF(\epsilon)$.
\end{proposition}
\begin{proof}
By construction, $\widehat \cF$ has the property that $\lambda+\epsilon \Delta$ extends to  a $\bQ$-Cartier and ample divisor class $\wh{\lambda}+\epsilon \wh{\Delta}$. (N.B. the relative ampleness of $\Delta$ is not explicitly stated in \cite{looijengacompact}, but this is precisely what Looijenga checks). Hence the ring of sections $R(\wh{\cF},\wh{\lambda}+\epsilon \wh{\Delta})$ makes sense (and is finitely generated). Since $\widehat \cF\to \cF^*$ is a small map, and $\wh{\cF}$ is normal (by construction), the restriction of sections to $\cF\subset \wh{\cF}$ defines an isomorphism  $R(\wh{\cF},\wh{\lambda}+\epsilon \wh{\Delta})\cong R(\cF,\lambda+\epsilon\Delta)$.
\end{proof}
\begin{remark}
According to the discussion of \cite[Ch. 6]{km}, the $\bQ$-factorialization of $\Delta$ is unique: it is either $\cF(\epsilon)$ or $\cF(-\epsilon)$ (depending on the requested relative ampleness). The main issue is that the finite generation of the ring of sections defining $\cF(\epsilon)$ is not a priori guaranteed. Looijenga \cite{looijengacompact} makes use of the special structure of the Baily-Borel  compactification (e.g.~the tube domain structure near the boundary, and the existence of toroidal compactifications) to obtain that the $\bQ$-factorialization is well defined, and furthermore to get an explicit description of it. 
\end{remark}

\begin{remark}
The results of~\cite{log1} (see esp. Proposition 5.4.5) predict that the above proposition  holds for $0<\epsilon< \frac{1}{9}$. 
\end{remark}

The stratification $\cF^{*}=\cF^{I}\sqcup \cF^{II}\sqcup\cF^{III}$ defines by pull-back a stratification $\wh{\cF}=\wh{\cF}^{I}\sqcup \wh{\cF}^{II}\sqcup\wh{\cF}^{III}$. The \emph{boundary strata} of  $\wh{\cF}$ are the irreducible components of the above strata. We are interested in the number   of the Type II boundary strata  of $\widehat \cF$ and their dimensions. 

We start by recalling that the structure of the Baily-Borel compactification for quartic surfaces was worked out by Scattone \cite{scattone}: there are $9$ Type II boundary components, and a single Type III boundary component. 
\begin{proposition}[Scattone \cite{scattone}]\label{prp:boundk3}

The boundary of the Baily-Borel compactification $\cF^*$ of the moduli space of quartic surfaces consists of $9$ Type II boundary components, and a single Type III component. The Type II boundary components are naturally labeled  by a rank $17$ negative definite lattice as follows: $D_{17}$, $D_9\oplus E_8$,  $D_{12}\oplus D_5$, $D_3\oplus(E_7)^2$, $A_{15}\oplus D_2$,  , $A_{11}\oplus E_6$,  $(D_8)^2\oplus D_1$, $D_{16}\oplus D_1$, and $(E_8)^2\oplus D_1$ respectively. 
\end{proposition}
%
%\begin{remark}
%We recall that the Type II components for $\cF$ are in one-to-one correspondence with the $\Gamma$-orbits of the rank $2$ isotropic subspaces $E$ of the lattice $\Lambda$ ($ E_8^2\oplus U\oplus \langle -4\rangle$ for quartics) occuring in the definition of $\cF$. A basic invariant for $E$ is the negative lattice $E^\perp_\Lambda/E$. The Type II boundary components in the proposition above are labeled by the root sublattice  in $E^\perp_\Lambda/E$.
%\end{remark}

\begin{proposition}\label{prp:dimqfact}
The dimensions of the Type II strata in the compactification $\widehat\cF$ are given in Table \ref{tabletype2dim}.  
  \begin{table}[htb!]
\renewcommand{\arraystretch}{1.60}
\begin{tabular}{|c|c||c|c||c|c|}
\hline
$D_{17}$& 1 &$D_9\oplus E_8$& 10 & $D_{12}\oplus D_5$&6\\ 
\hline
$D_3\oplus(E_7)^2$& 4 & $A_{15}\oplus D_2$& 3 & $A_{11}\oplus E_6$& 1 \\ 
\hline
$(D_8)^2\oplus D_1$& 2 & $D_{16}\oplus D_1$& 6 & $(E_8)^2\oplus D_1$& 2\\
\hline 
\end{tabular}
\vspace{0.2cm}
\caption{Dimension of the boundary strata in $\widehat \cF$}\label{tabletype2dim}
\end{table}

\end{proposition}
\begin{proof} A type II boundary component is determined by the choice an isotropic rank $2$ primitive sublattice $E\subset \Lambda(\cong E_8^2\oplus U\oplus \langle -4\rangle)$ (up to the action of the monodromy group). The label associated to a Type II boundary component is the root sublattice contained in the negative definite rank $17$ lattice $E^{\perp}_{\Lambda}/E$ (with the convention of including also $D_1=\langle -4\rangle$ in the root lattice). According to \cite{scattone}, this is a complete invariant for a Type II boundary component in the case of quartic surfaces. 

The construction of Looijenga \cite{looijengacompact} (see esp. Section 3 and  Proposition 3.3 of loc. cit.) depends on the linear space
$$L:=\left(\cap_{H\in \sH, E\subset H} (H\cap E^\perp)\right)/E\subset E^\perp/E.$$
More precisely, let $M:=E_\Lambda^\perp/E$. Note $M$ is a negative definite rank $17$ lattice. Then, we recall that the fiber over a point $j$ in the type II boundary component (recall each Type II boundary component is a modular curve, here $\mathfrak{h}/\SL(2,\bZ)$) associated to $E$ is simply the quotient of the abelian variety $J(\mathcal E_j)\otimes_\bZ M$ by a finite group (here $\mathcal E_j$ denotes the elliptic curve of modulus $j$, and $J(\mathcal E_j)$ its Jacobian). What Looijenga has observed is that $L$ is the null-space of the restriction to the toroidal boundary (of Type II) of the linear system determined by the hyperplane arrangement $\sH$. And thus, the fiber for the $\bQ$-factorialization (which as discussed above corresponds to the $\Proj$ of the ring of sections of $\lambda+\epsilon \Delta$; also recall (the pull-back of) $\lambda$ restricts to trivial on the toroidal boundary) over the point $j$ in the Type II boundary component associated to $E$ is (up to finite quotient) $J(\mathcal E_j)\otimes_\bZ M/L$. 

Now, we recall that the lattice $\Lambda$ can be primitively embedded into the Borcherds lattice $II_{2,26}$ with orthogonal complement $D_7$ (in a unique way). We fix 
$$\Lambda\hookrightarrow II_{2,26}$$
and $R=\Lambda^{\perp}\cong D_7$. With respect to this embedding,  a hyperelliptic hyperplane corresponds to an extension of $R$ to a (primitively embedded) $D_8$ into $II_{2,26}$, while a unigonal divisor to a $E_8$. Successive intersections of hyperplanes from $\sH$ correspond to extensions of $R=(D_7)$ into $D_k$ lattices. Similarly, if $E$ is rank $2$ isotropic (primitively embedded), then we recall that $M=E^\perp/E$ can be embedded into one the $24$ Niemeier lattices (i.e. rank $24$ negative definite even unimodular lattices) with orthogonal complement $D_7$. The same considerations as before apply: a hyperelliptic divisor correspond to an extension to $D_8$ (and repeated intersections to $D_{7+k}$), while a unigonal one corresponds to an extension to $E_8$.  By inspecting the possible embeddings of $D_k$ lattices into Niemeier lattices, one obtains the dimensions claimed in Table \ref{tabletype2dim}. The only exception is the case $D_{17}$ (in which case $D_7$ extends to $D_{24}$) for which $L=0\subset M$, and thus the Heegner divisor is already $\bQ$-Cartier (and no modification is necessary; see \cite[Cor. 3.5]{looijengacompact}). 
\end{proof}
\subsection{Matching Type II strata}
In order to understand the matching of the GIT and Baily-Borel Type II strata, one needs to consider a generic smoothing $\sX/\Delta$ of a Type II quartic surface $X_0$ and compute the limit MHS with $\bZ$-coefficients. The analogous case of $K3$ surfaces of degree $2$ was analyzed by Friedman in~\cite{friedmanannals}. Inspired by Friedman's analysis, we make the following definition: 
 \begin{definition}\label{dfn:heuristics}
Let $X_0$ be a Type II polystable quartic surface. The \emph{associated (isomorphism class of) lattice}  is the direct sum of the following lattices: 
\begin{itemize}
\item One copy of  $E_r$ for each $\widetilde E_r$ singularity of $X$.
\item One copy of $D_{4d+4}$ for each degree $d$ rational  curve  in the singular set of $X$.
\item One copy of  $A_{4d-1}$ for each degree $d$ elliptic curve  in the singular set of $X$.
\item
The lattice  $\langle h_{\tilde X},K_{\tilde X}\rangle^\perp \subset \Pic (\tilde X)$ where $\tilde X$ is the minimal resolution of the normalization of $X$ and $h_{\tilde X}$ is the polarization class on $\tilde X$ (e.g.~if $\tilde X$ is a degree $2$ del Pezzo with the anticanonical polarization, we add $E_7$).
\end{itemize}
 \end{definition}
  \begin{remark}
  To understand the meaning of the lattice associated to a Type II degeneration $X_0$, one needs to consider a generic smoothing $\sX/\Delta$ of $X_0$, followed by a semi-stable (Kulikov type) resolution $\widetilde \sX/\Delta$. The lattice introduced in the definition above is essentially $(W_2/W_1)_{\textrm{prim
 }}$ from \cite[(5.1)]{friedmanannals}. The main point here (similar to the discussion of \Ref{sec}{dolgachev}) is that one has quite a good understanding of the semistable replacement in the Type II case. For instance, the simple elliptic singularities $\widetilde E_r$ ($r=6,7,8$) will be replaced by degree $9-r$ del Pezzo tails (this leads to the first item of \Ref{dfn}{heuristics}). 
  \end{remark}
 \begin{remark}
By going through our list of Type II components of $\gM$, one checks the following:
\begin{enumerate}
\item
Two polystable quartic surfaces belonging to the same Type II component of $\gM$ have  isomorphic associated  lattices.
\item
The  lattice associated to a polystable quartic surface  of Type II has rank $17$, is  negative definite, even, and belongs to the list of lattices associated to Type II boundary components of $\cF^{*}$, see~\Ref{prp}{boundk3}. 
\item
By associating to a Type II component of $\gM$ the lattice associated to any polystable quartic in the component (see Item~(1)), we get a one to one correspondence between the set of Type II components of $\gM$ and the set of lattices appearing in~\Ref{prp}{boundk3}, provided we remove the $D_{17}$ lattice.
\end{enumerate}
 \end{remark}
The geometric meaning of the lattice associated to a polystable quartic of Type II is provided by our next result, which is proved  by mimicking the arguments of Friedman in~\cite{friedmanannals}  (see esp.~\cite[Rem. 5.6]{friedmanannals}).
\begin{proposition}
Let $X$ be a polystable Type II quartic surface. The period point $\gp([X])$ belongs to the Baily-Borel Type II boundary component labeled by the lattice
  associated to $X$. 
\end{proposition}
 So far we have proved that  the set of lattices appearing in~\Ref{prp}{boundk3}, once  we remove the $D_{17}$ lattice, parametrizes both the components of $\gM^{II}$, and  the Type II boundary components of $\cF^{*}$, with the exclusion of one. The two parametrizations are compatible with respect to the period map. Of course the same set of lattices parametrizes Type II boundary components of $\wh{\cF}$, with the exception of one. 
 \begin{proposition}
Let $L$ be one of the lattices appearing in~\Ref{prp}{boundk3}, with the exception of $D_{17}$. The dimension of the Type II boundary component of $\wh{\cF}$ indicized by $L$ is equal to the dimension of the Type II component of $\gM$ indicized by the same $L$.
\end{proposition}
 \begin{proof}
 We illustrate the computation of dimensions  in the highest dimensional case: II(5), i.e. quartics that have a single $\widetilde E_8$ singularity such that no line passes through this singularity. Let $X_0$ be a generic surface of this type; then $X_0$ has a singularity of type $\widetilde E_8$ at some point $p$ and is smooth away from $p$,
  see~\Ref{rmk}{e8tildecase}. Let $\widetilde X_0\to X_0$ be the minimal resolution. By~\Ref{rmk}{e8tildecase}, the exceptional divisor $D$ is an elliptic curve with self-intersection $-1$,  $D$ is an anti-canonical section,  and $\rho(\widetilde X_0)=11$ (i.e.~$\widetilde X_0$ is the blow-up of $\bP^2$ along $10$ points on an elliptic curve). Thus, $H^2(\widetilde X_0)$ is nothing else than the lattice $I_{1,10}$. By the discussion in~\Ref{rmk}{arithmetice8}, it follows that $\langle K_{\widetilde X_0}+h_{\widetilde{X_0}}\rangle^\perp_{H^2(\widetilde X_0)}$ is isometric to $D_{9}$. Since $X_0$ has an $\widetilde E_8$ singularity, by our rule (\Ref{dfn}{heuristics}), the associated label is $D_9\oplus E_8$. 
 
From Shah \cite{shah4}, a generic surface  $X_0$ of Type II(5) is GIT stable. On the other hand, 
 the results of \cite{dpw} and \cite{shustin}  imply in particular (loc.~cit.~give general conditions in terms of total Tjurina number) that the universal family of quartic surfaces versally unfolds the $\widetilde E_8$ singularity. From these two results, it follows that the codimension of the locus with a fixed $\widetilde E_8$ singularity is $10=\mu(\widetilde E_8)$ (where $\mu$ is the Milnor, and also, in this case, the Tjurina number), but there is an additional $1$-dimensional deformation corresponding to moduli of  simple elliptic singularities. Summin up, the II(5) locus has codimension $9$ (i.e.~dimension $10$) in the GIT quotient $\gM$. (The same dimension count also follows from the geometric description given in~\Ref{rmk}{e8tildecase}.) 
 
 The computation of the dimension of the stratum labeled by $E_8\oplus D_9$ in $\wh{\cF}$ has been carried out  in~\Ref{prp}{dimqfact}. Here we point out that this case corresponds to the Niemeier lattice containing the root system $D_{16}\oplus E_8$. In that situation, the maximally embedded $D_l$ is $D_{16}$, which means (using the notation of \Ref{prp}{dimqfact}) $\dim M/L=9$ (N.B. $16=7+9$). Thus the fiber of $\widehat \cF\to \cF^*$ over a point $j$  in the Type II component labeled by $E_8\oplus D_9$ is $9$. Then again, by varying $j$, we obtain a $10$-dimensional component (this time in $\widehat \cF$; thus the dimensions in $\gM$ and $\widehat \cF$ match). 
  
 To get further geometric understanding of the matching of the GIT component II(5) and of the component labeled by $D_9\oplus E_8$ in $\widehat \cF$, we note that there exists an extended period map. Specifically, recall that $X_0$ carries a mixed Hodge structure (MHS), and that there exists also a limit mixed Hodge structure (LMHS). The Baily-Borel compactification $\cF^*$ encodes the graded pieces of the LHMS (in this situation, the modulus of the elliptic curve $C$, and a discrete part, i.e. the choice of Type II component, or equivalently the label of the component). As previously discussed, the graded pieces of the LMHS can be read off from those of the MHS on degeneration $X_0$ (the weight $1$ part follows from \Ref{thm}{dubois}, while the discrete weight $2$ part is the rule given by \Ref{dfn}{heuristics}). On the other hand, a toroidal compactification $\overline{\cF}^\Sigma$ (which is unique over the Type II stratum) encodes the full LMHS (i.e. the graded pieces, plus the extension data; see Friedman \cite{friedmanannals} for a full discussion). Finally,  the semitoric compactifications of Looijenga are sitting between the Baily-Borel and the toroidal compactifications: $\overline{\cF}^\Sigma\to \widehat \cF\to \cF^*$. Thus, from a Hodge theoretic perspective, $\widehat \cF$ retains the graded pieces of the LHMS, plus partial extension data. As explained below, this partial extension data is exactly the extension data that can be read off from the central fiber $X_0$ (without passing to the Kulikov model).

 Specifically, in the case that we discuss here (Type II(5)), the  Kulikov model  is $\widetilde X_0\cup_E T$, where (as above) $\widetilde X_0$ is the resolution of the quartic surface with an $\widetilde E_8$ singularity, $T$ is a ``tail'' (depending on the direction of the smoothing). In this situation, $T$ is a degree $1$ del Pezzo surface, whose primitive cohomology is $E_8$. $\widetilde X_0$ is a rational surface with primitive cohomology $D_9$. Finally, the gluing curve $E$ is an elliptic curve (with self-intersection $1$ on $T$ and $-1$ on $\widetilde{X_0}$), which gives the modulus $j$ discussed above. Fixing $j$, the modulus of these type of surfaces (up to the monodromy action) is the $17$ dimensional abelian variety $(E_8\oplus D_9)\otimes_\bZ J(\mathcal E_j)$ (this is precisely the fiber of the toroidal compactification $\overline{\cF}^\Sigma\to \cF^*$ over the appropriate Type II Baily-Borel boundary point). When passing to Looijenga $\bQ$-factorialization, the fiber of $\widehat \cF\to \cF^*$ becomes $D_9\otimes_\bZ M/L$ (N.B.  $M/L=D_9$ in this case). This fiber can be identified with the moduli space of $X_0$ (or equivalently $(\widetilde X_0, D)$ with fixed $j$-invariant for $D$). More precisely, it is possible to see that the restriction of the extended period map 
 $$\gM\dashrightarrow \widehat \cF\to \cF^*$$
(which extends over the Type II and III locus) to the locus II(5) is nothing else but the period map for the anticanonical pair $(\widetilde X_0, D)$ (see \cite{ghk} and \cite{friedmananti} for a general modern discussion of the period map for anticanonical pairs, and \cite[Section 5]{urabee12} for the specific case discussed here; all of this originates with work of Looijenga \cite{lanti}). In conclusion, we get a perfect matching between the Type II(5) stratum in $\gM$ and  the Type II stratum in $\widehat \cF$ labeled by $D_9+E_8$. \footnote{In fact, while we do not check it here (as it involves subtle arithmetic considerations), we fully expect that the extended period map is an isomorphism (at the generic points) between the II(5) and II($D_9+E_8$) strata (and similarly for the other Type II strata).} 
 \end{proof}
 
  \begin{table}[htb!]
\renewcommand{\arraystretch}{1.60}
\begin{tabular}{|c|c|c|}
\hline
GIT stratum& BB stratum& Dimension\\
\hline
II(1)&$(E_8)^2\oplus D_1$&2\\
II(2)&$(E_7)^2\oplus A_3$&4\\
II(3)&$(D_8)^2\oplus D_1$&2\\
II(4)&$E_6\oplus A_{11}$&1\\
\hline
II(5)&$E_8\oplus D_9$&10\\
II(6)&$D_{12}\oplus D_5$&6\\
II(7)&$D_{16}\oplus D_1$&6\\
II(8)&$A_{15}\oplus (A_1)^2$&3\\
\hline
\end{tabular}
\vspace{0.2cm}
\caption{Matching of the Type II strata}\label{tabletype2}
\end{table}

\begin{remark}[Kulikov models]
It is not hard to produce Kulikov models for each of the Type II degenerations above. For instance, in a semi-stable degeneration, each of the $\widetilde E_r$ singularities will be replaced by a del Pezzo of degree $(9-r)$. As an example, the case II(1) corresponding to a quartic with $2$ $\widetilde E_8$ singularities will give $2$ degree $1$ del Pezzo surfaces, glued to a elliptic ruled surface (which is in the fact the resolution of the singular quartic; the del Pezzo surface are ``tails'' coming from the singularity; see \S\ref{ssekcase} below for related computations). 
The case of quartics singular along a curve typically are obtained by projection from a rational surfaces (frequently del Pezzo). For instance the case II(6) is obtained by projecting a degree $4$ del Pezzo from a point in $\bP^4$. The associated label $D_{12}\oplus D_5$  has the following meaning: $D_5(=E_5)$ is the primitive cohomology of the associated degree $4$ del Pezzo. On the other hand $D_{12}$ is coming from the singularity of the quartic (in this case a conic) and the rule (\Ref{dfn}{heuristics}) given above. \end{remark}
 
\subsection{The missing GIT Type II component}
The reader might be puzzled by the fact the BB stratum corresponding to $D_{17}$ does not  occur in the list of Type II components of $\gM$. We can explain this as follows. First of all as noted in \Ref{prp}{dimqfact}, along this stratum the hyperelliptic divisor is $\bQ$-factorial and thus it is not affected by the $\bQ$-factorialization. Moreover, this is precisely the boundary component that is contained in all elements of the $D$-tower. (this is the component that survives when we go to low dimensions). As discussed the entire $\Delta^{(k)}$ (when we get to codimension $9$) is contracted and then flipped (i.e. there is no deeper flip). This is indeed compatible with a theorem of Looijenga which identifies $\cF(10)^*$ with a certain weighted projective space and with the moduli space of $T_{2,3,7}(=E_8\oplus U)$ marked $K3$s (see~\Ref{sec}{dolgachev}). In conclusion, the 9th boundary component is flipped all at once together with a big stratum, and thus will not be visible in $\gM^{II}$. It is \lq\lq hidden\rq\rq\  in the $E_{12}$ stratum (i.e.~IV(8)) in $\gM^{IV}$.

%%%%%%%%%%%%%%%%%%%%%%%%%%%%%%%%%%%%%%%
 \bibliography{refk3}

\providecommand{\bysame}{\leavevmode\hbox to3em{\hrulefill}\thinspace}
\providecommand{\MR}{\relax\ifhmode\unskip\space\fi MR }
\providecommand{\MRhref}[2]{%
  \href{http://www.ams.org/mathscinet-getitem?mr=#1}{#2}
}
\providecommand{\href}[2]{#2}
\begin{thebibliography}{AGLV98}

\bibitem[ACT11]{act}
D.~Allcock, J.~A. Carlson, and D.~Toledo, \emph{The moduli space of cubic
  threefolds as a ball quotient}, Mem. Amer. Math. Soc. \textbf{209} (2011),
  no.~985, xii+70.

\bibitem[AGLV98]{arnold}
V.~I. Arnol'd, V.~V. Goryunov, O.~V. Lyashko, and V.~A. Vasil{'}ev,
  \emph{Singularity theory. {I}}, Springer-Verlag, Berlin, 1998.

\bibitem[Ale02]{alexeevag}
V.~Alexeev, \emph{Complete moduli in the presence of semiabelian group action},
  Ann. of Math. (2) \textbf{155} (2002), no.~3, 611--708.

\bibitem[And16]{andblow}
M.~Andreatta, \emph{Lifting weighetd blow-ups},
  https://arxiv.org/abs/1609.00156, 09 2016.

\bibitem[Dol75]{dolgachevs}
I.~V. Dolgachev, \emph{Automorphic forms, and quasihomogeneous singularities},
  Funkcional. Anal. i Prilo\v zen. \textbf{9} (1975), no.~2, 67--68.

\bibitem[Dol96]{dolgachevm}
\bysame, \emph{Mirror symmetry for lattice polarized {$K3$} surfaces}, J. Math.
  Sci. \textbf{81} (1996), no.~3, 2599--2630, Algebraic geometry, 4.

\bibitem[dPW00]{dpw}
A.~A. du~Plessis and C.~T.~C. Wall, \emph{Singular hypersurfaces, versality,
  and {G}orenstein algebras}, J. Algebraic Geom. \textbf{9} (2000), no.~2,
  309--322.

\bibitem[Ebe99]{ebeling}
W.~Ebeling, \emph{Strange duality, mirror symmetry, and the {L}eech lattice},
  Singularity theory ({L}iverpool, 1996), London Math. Soc. Lecture Note Ser.,
  vol. 263, Cambridge Univ. Press, Cambridge, 1999, pp.~xv--xvi, 55--77.

\bibitem[Fri83]{friedman-mon}
R.~Friedman, \emph{A degenerating family of quintic surfaces with trivial
  monodromy}, Duke Math. J. \textbf{50} (1983), no.~1, 203--214.

\bibitem[Fri84]{friedmanannals}
\bysame, \emph{A new proof of the global {T}orelli theorem for {$K3$}
  surfaces}, Ann. of Math. (2) \textbf{120} (1984), no.~2, 237--269.

\bibitem[Fri13]{friedmananti}
\bysame, \emph{On the ample cone of a rational surface with an anticanonical
  cycle}, Algebra Number Theory \textbf{7} (2013), no.~6, 1481--1504.

\bibitem[Gal13]{gallardo}
P.~Gallardo, \emph{On the {GIT} quotient of quintic surfaces}, arXiv:1310.3534,
  2013.

\bibitem[GHK15]{ghk}
M.~Gross, P.~Hacking, and S.~Keel, \emph{Moduli of surfaces with an
  anti-canonical cycle}, Compos. Math. \textbf{151} (2015), no.~2, 265--291.

\bibitem[Hac04]{hacking}
P.~Hacking, \emph{Compact moduli of plane curves}, Duke Math. J. \textbf{124}
  (2004), no.~2, 213--257.

\bibitem[Has00]{hassett}
B.~Hassett, \emph{Local stable reduction of plane curve singularities}, J.
  Reine Angew. Math. \textbf{520} (2000), 169--194.

\bibitem[Has05]{hassg2}
\bysame, \emph{Classical and minimal models of the moduli space of curves of
  genus two}, Geometric methods in algebra and number theory, Progr. Math.,
  vol. 235, Birkh{\"a}user Boston, Boston, MA, 2005, pp.~169--192. \MR{2166084}

\bibitem[HH09]{hh1}
B.~Hassett and D.~Hyeon, \emph{Log canonical models for the moduli space of
  curves: the first divisorial contraction}, Trans. Amer. Math. Soc.
  \textbf{361} (2009), no.~8, 4471--4489.

\bibitem[HH13]{hh2}
\bysame, \emph{Log minimal model program for the moduli space of stable curves:
  the first flip}, Ann. of Math. (2) \textbf{177} (2013), no.~3, 911--968.

\bibitem[HM98]{hm}
J.~Harris and I.~Morrison, \emph{Moduli of curves}, Graduate Texts in
  Mathematics, vol. 187, Springer-Verlag, New York, 1998.

\bibitem[HN17]{hn}
L.H. Halle and J.~Nicaise, \emph{Motivic zeta functions of degenerating
  {C}alabi-{Y}au varieties}, to appear in Math. Ann. (arXiv:1701.09155), 2017.

\bibitem[Igu62]{igusag2}
J.~Igusa, \emph{On {S}iegel modular forms of genus two}, Amer. J. Math.
  \textbf{84} (1962), 175--200.

\bibitem[IKKR15]{ikkr}
A.~Iliev, G.~Kapustka, M.~Kapustka, and K.~Ranestad, \emph{{EPW} cubes},
  arXiv:1505.02389, 2015.

\bibitem[Kir85]{kirwan}
F.~C. Kirwan, \emph{Partial desingularisations of quotients of nonsingular
  varieties and their {B}etti numbers}, Ann. of Math. (2) \textbf{122} (1985),
  no.~1, 41--85.

\bibitem[Kir89]{kirwanhyp}
\bysame, \emph{Moduli spaces of degree {$d$} hypersurfaces in {${\bf P}_n$}},
  Duke Math. J. \textbf{58} (1989), no.~1, 39--78.

\bibitem[KK10]{kk}
J.~Koll{\'a}r and S.~J. Kov{\'a}cs, \emph{Log canonical singularities are {D}u
  {B}ois}, J. Amer. Math. Soc. \textbf{23} (2010), no.~3, 791--813.

\bibitem[KL04]{kimlee}
H.~Kim and Y.~Lee, \emph{Log canonical thresholds of semistable plane curves},
  Math. Proc. Cambridge Philos. Soc. \textbf{137} (2004), no.~2, 273--280.

\bibitem[KLSV17]{klsv}
J.~Koll\'ar, R.~Laza, G.~Sacc\`a, and C.~Voisin, \emph{Remarks on degenerations
  of hyper-{K}\"ahler manifolds}, arXiv:1704.02731, 2017.

\bibitem[KM92]{kollmori3flips}
J.~Koll{{\'a}}r and S.~Mori, \emph{Classification of three-dimensional flips},
  J. Amer. Math. Soc. \textbf{5} (1992), no.~3, 533--703.

\bibitem[KM98]{km}
J.~Koll{\'a}r and S.~Mori, \emph{Birational geometry of algebraic varieties},
  Cambridge Tracts in Mathematics, vol. 134, Cambridge University Press,
  Cambridge, 1998.

\bibitem[KSB88]{ksb}
J.~Koll{\'a}r and N.~I. Shepherd-Barron, \emph{Threefolds and deformations of
  surface singularities}, Invent. Math. \textbf{91} (1988), no.~2, 299--338.

\bibitem[Laz09]{gitcubic}
R.~Laza, \emph{The moduli space of cubic fourfolds}, J. Algebraic Geom.
  \textbf{18} (2009), no.~3, 511--545.

\bibitem[Laz10]{cubic4fold}
\bysame, \emph{The moduli space of cubic fourfolds via the period map}, Ann. of
  Math. (2) \textbf{172} (2010), no.~1, 673--711.

\bibitem[Laz16]{k3pairs}
\bysame, \emph{The {KSBA} compactification for the moduli space of degree two
  {$K3$} pairs}, J. Eur. Math. Soc. (JEMS) \textbf{18} (2016), no.~2, 225--279.

\bibitem[LO16]{log1}
R.~Laza and K.~G. O'Grady, \emph{Birational geometry of the moduli space of
  quartic {$K3$} surfaces}, arXiv:1607.01324, 07 2016.

\bibitem[Loo81]{lanti}
E.~Looijenga, \emph{Rational surfaces with an anticanonical cycle}, Ann. of
  Math. (2) \textbf{114} (1981), no.~2, 267--322.

\bibitem[Loo83]{ltriangle1}
\bysame, \emph{The smoothing components of a triangle singularity. {I}},
  Singularities, {P}art 2 ({A}rcata, {C}alif., 1981), Proc. Sympos. Pure Math.,
  vol.~40, Amer. Math. Soc., Providence, RI, 1983, pp.~173--183.

\bibitem[Loo84]{ltriangle2}
\bysame, \emph{The smoothing components of a triangle singularity. {II}}, Math.
  Ann. \textbf{269} (1984), no.~3, 357--387.

\bibitem[Loo86]{looijengavancouver}
\bysame, \emph{New compactifications of locally symmetric varieties},
  Proceedings of the 1984 Vancouver conference in algebraic geometry
  (Providence, RI), CMS Conf. Proc., vol.~6, Amer. Math. Soc., 1986,
  pp.~341--364.

\bibitem[Loo03a]{looijenga1}
\bysame, \emph{Compactifications defined by arrangements. {I}. {T}he ball
  quotient case}, Duke Math. J. \textbf{118} (2003), no.~1, 151--187.

\bibitem[Loo03b]{looijengacompact}
\bysame, \emph{Compactifications defined by arrangements. {II}. {L}ocally
  symmetric varieties of type {IV}}, Duke Math. J. \textbf{119} (2003), no.~3,
  527--588.

\bibitem[Loo09]{lcubic}
\bysame, \emph{The period map for cubic fourfolds}, Invent. Math. \textbf{177}
  (2009), no.~1, 213--233.

\bibitem[LS07]{ls}
E.~Looijenga and R.~Swierstra, \emph{The period map for cubic threefolds},
  Compos. Math. \textbf{143} (2007), no.~4, 1037--1049.

\bibitem[Oda12]{odakacy}
Y.~Odaka, \emph{The {C}alabi conjecture and {K}-stability}, Int. Math. Res.
  Not. IMRN (2012), no.~10, 2272--2288.

\bibitem[Oda13]{odaka}
\bysame, \emph{The {GIT} stability of polarized varieties via discrepancy},
  Ann. of Math. (2) \textbf{177} (2013), no.~2, 645--661.

\bibitem[O'G15]{epwperiods}
K.~G. O'Grady, \emph{Periods of double {EPW}-sextics}, Math. Z. \textbf{280}
  (2015), no.~1-2, 485--524.

\bibitem[O'G16]{epw}
\bysame, \emph{Moduli of double {EPW}-sextics}, Memoirs of the American
  Mathematical Society \textbf{240} (2016), no.~1136, 1--172.

\bibitem[SB83a]{sbnef}
N.~I. Shepherd-Barron, \emph{Degenerations with numerically effective canonical
  divisor}, The birational geometry of degenerations (Cambridge, Mass., 1981),
  Progr. Math., vol.~29, Birkh\"auser Boston, Boston, MA, 1983, pp.~33--84.

\bibitem[SB83b]{sbpolarization}
\bysame, \emph{Extending polarizations on families of {$K3$} surfaces}, The
  birational geometry of degenerations (Cambridge, Mass., 1981), Progr. Math.,
  vol.~29, Birkh\"auser Boston, Mass., 1983, pp.~135--171.

\bibitem[Sca87]{scattone}
F.~Scattone, \emph{On the compactification of moduli spaces for algebraic
  {$K3$} surfaces}, Mem. Amer. Math. Soc. \textbf{70} (1987), no.~374, x+86.

\bibitem[Sha79]{shahinsignificant}
J.~Shah, \emph{Insignificant limit singularities of surfaces and their mixed
  {H}odge structure}, Ann. of Math. (2) \textbf{109} (1979), no.~3, 497--536.

\bibitem[Sha80]{shah}
\bysame, \emph{A complete moduli space for {$K3$} surfaces of degree {$2$}},
  Ann. of Math. (2) \textbf{112} (1980), no.~3, 485--510.

\bibitem[Sha81]{shah4}
\bysame, \emph{Degenerations of {$K3$} surfaces of degree {$4$}}, Trans. Amer.
  Math. Soc. \textbf{263} (1981), no.~2, 271--308.

\bibitem[ST99]{shustin}
E.~Shustin and I.~Tyomkin, \emph{Versal deformation of algebraic hypersurfaces
  with isolated singularities}, Math. Ann. \textbf{313} (1999), no.~2,
  297--314.

\bibitem[Ste81]{steenbrinkinsignificant}
J.~H.~M. Steenbrink, \emph{Cohomologically insignificant degenerations},
  Compositio Math. \textbf{42} (1980/81), no.~3, 315--320.

\bibitem[Tha96]{thaddeus}
M.~Thaddeus, \emph{Geometric invariant theory and flips}, J. Amer. Math. Soc.
  \textbf{9} (1996), no.~3, 691--723.

\bibitem[Ura84]{urabee12}
T.~Urabe, \emph{On quartic surfaces and sextic curves with singularities of
  type {$\tilde E_8,\;T_{2,3,7},\;E_{12}$}}, Publ. Res. Inst. Math. Sci.
  \textbf{20} (1984), no.~6, 1185--1245.

\bibitem[WX14]{wangxu}
X.~Wang and C.~Xu, \emph{Nonexistence of asymptotic {GIT} compactification},
  Duke Math. J. \textbf{163} (2014), no.~12, 2217--2241.

\bibitem[Yon90]{yonemura}
T.~Yonemura, \emph{Hypersurface simple {$K3$} singularities}, Tohoku Math. J.
  (2) \textbf{42} (1990), no.~3, 351--380.

\end{thebibliography}
 \end{document}